%% file: main_nilima.tex
\documentclass[]{elsarticle}
\usepackage{lineno,hyperref,fullpage,tikz-cd}

\usepackage{graphicx}
\usepackage{bm}
\usepackage{multirow}
\usepackage{amsmath}
\usepackage{amsfonts}
\usepackage{amssymb}
\usepackage{amsthm}
\usepackage{secdot}
\usepackage{accents}

\usepackage{multicol}
\usepackage{booktabs}
\usepackage{lscape}

\journal{Computers \& Mathematics with Applications}

\theoremstyle{plain}
\newtheorem{theorem}{Theorem}[section]
\newtheorem{lemma}[theorem]{Lemma}

\theoremstyle{definition}

\theoremstyle{remark}
\newtheorem{remark}{Remark}

\numberwithin{equation}{section}
\numberwithin{theorem}{section}
\numberwithin{remark}{section}

\catcode`\@=11
\newdimen\cdsep
\def\cdstrut{\vrule height .6\cdsep width 0pt depth .4\cdsep}
\def\@cdstrut{{\advance\cdsep by 2em\cdstrut}}

\def\arrow#1#2{
  \ifx d#1
    \llap{$\scriptstyle#2$}\left\downarrow\cdstrut\right.\@cdstrut\fi
  \ifx u#1
    \llap{$\scriptstyle#2$}\left\uparrow\cdstrut\right.\@cdstrut\fi
  \ifx r#1
    \mathop{\hbox to \cdsep{\rightarrowfill}}\limits^{#2}\fi
  \ifx l#1
    \mathop{\hbox to \cdsep{\leftarrowfill}}\limits^{#2}\fi
}
\catcode`\@=12

\usepackage{tikz}
\usetikzlibrary{positioning}

\newcommand{\C}[1]{\mathfrak{H}^#1} 
\newcommand{\T}[1]{\mathfrak{T}^#1} 
 
\newcommand{\Tr}[2]{\mathrm{Tr}[#1](#2)}  
\newcommand{\tr}[2]{\mathrm{tr}[#1](#2)}  
\newcommand{\W}[1]{\mathfrak{N}^#1} 

\newcommand{\Vkcirc}[2]{\accentset{\circ}{V}_{#1} \Lambda^#2} 
\newcommand{\Vk}[2]{V_{#1} \Lambda^{#2}} 
\newcommand{\Sk}[2]{\Sigma^{#1,#2}} 

\newcommand{\VtoM}[1]{\mathcal{L}\left(#1\right)} 

\begin{document}

\begin{frontmatter}

\title{Conforming Finite Element Function Spaces in Four Dimensions, Part II: The Pentatope and Tetrahedral Prism}


\author{David M. Williams \corref{mycorrespondingauthor}} 
\address{Department of Mechanical Engineering, The Pennsylvania State University, University Park, Pennsylvania 16802, United States}

\cortext[mycorrespondingauthor]{Corresponding author}
\ead{david.m.williams@psu.edu}

\author{Nilima Nigam} 
\address{Department of Mathematics, Simon Fraser University, Burnaby, British Columbia BC V5C 2V3, Canada}


\begin{abstract}
In this paper, we present explicit expressions for conforming finite element function spaces, basis functions, and degrees of freedom on the pentatope and tetrahedral prism elements. More generally, our objective is to construct finite element function spaces that maintain conformity with infinite-dimensional spaces of a carefully chosen de Rham complex. This paper is a natural extension of the companion paper entitled ``Conforming Finite Element Function Spaces in Four Dimensions, Part I: Foundational Principles and the Tesseract" by Nigam and Williams, (2023). In contrast to Part I, in this paper we focus on two of the most popular elements which do not possess a full tensor-product structure in all four coordinate directions. We note that these elements appear frequently in existing space-time finite element methods. In order to build our finite element spaces, we utilize powerful techniques from the recently developed `Finite Element Exterior Calculus'. Subsequently, we translate our results into the well-known language of linear algebra (vectors and matrices) in order to facilitate implementation by scientists and engineers. 
\end{abstract}

\begin{keyword}
space-time; finite element methods; tetrahedral prism; pentatope; four dimensions; finite element exterior calculus
\MSC[2010] 14F40, 52B11, 58A12, 65D05, 74S05 
\end{keyword}

\end{frontmatter}

\section{Introduction}
Finite Element Exterior Calculus (FEEC) is a powerful and elegant framework for constructing exact sequences of finite element approximation spaces in arbitrary dimensions, (see for instance, the landmark paper~\cite{arnold2010finite}). There has been considerable literature dedicated to the development and analysis of FEEC on simplicial, tensorial, and prism-like elements. Our goal in this and a companion paper~\cite{nigam2023conforming}, is to restrict these results to the specific case of $\mathbb{R}^4$, and to present an explicit construction of these families of finite elements. Notable previous contributions in this direction are due to \cite{gopalakrishnan2018auxiliary} and \cite{fuentes2015orientation}, and several more important contributions will be discussed below. Broadly speaking, our construction uses a different de Rham complex than that of the previous work, (essentially, the adjoint of the complex which was used in~\cite{gopalakrishnan2018auxiliary}). In the companion paper~\cite{nigam2023conforming}, we presented several practical examples in $\mathbb{R}^4$ to motivate our definition of a de Rham complex and the associated traces. We shall only briefly review this material in the present paper.  

In this work, we focus on developing conforming finite element function spaces for the \emph{pentatope} and \emph{tetrahedral prism}. The pentatope is a generalization of the triangle to four dimensions, and the tetrahedral prism is a generalization of the triangular prism to four dimensions. In what follows, we review some of the relevant literature on these elements.

\subsection{Background}

The pentatope and tetrahedral prism have been frequently used in space-time finite element methods. For example, pentatopes have been used by Behr and coworkers to simulate linear and non-linear advection-diffusion problems for fluid dynamics applications with moving boundaries~\cite{behr2008simplex,karyofylli2018simplex,karyofylli2019simplex,von2019simplex,von2021four}. In their work, a \emph{partially}-unstructured pentatope mesh is formed by extruding a three-dimensional tetrahedral mesh in the temporal direction to create four-dimensional tetrahedral prism elements, and thereafter, these tetrahedral prism elements are subdivided into pentatope elements in accordance with a Delaunay criterion~\cite{behr2008simplex}. In addition, there is considerable interest in generating \emph{fully}-unstructured pentatope meshes, as evidenced by the efforts of Foteinos and Chrisochoides~\cite{foteinos20154d}, Caplan et al.~\cite{caplan2017anisotropic,caplan2019extension,caplan2019four,caplan2020four}, Frontin et al.~\cite{frontin2021foundations}, and Anderson et al.~\cite{anderson2023surface}. Broadly speaking, this latter work focuses on developing a more direct, Delaunay-based approach for generating unstructured meshes of pentatopes, (in contrast to the less direct extrusion technique of Behr and coworkers). 


Let us now turn our attention to the tetrahedral prism. Meshes of tetrahedral prisms can be generated in a very straightforward fashion, as we only need to extrude an existing tetrahedral mesh in order to generate a completely valid, boundary-conforming mesh of tetrahedral prisms (see above). 
We note that Tezduyar, Bazilevs, and coworkers have performed extensive work on space-time methods for tetrahedral prisms~\cite{tezduyar2006space,tezduyar2007modelling,tezduyar2010space,tezduyar2011space,takizawa2012space,kuraishi2019space,terahara2020heart,bazilevs2020ale}. They have used these methods to solve a host of fluid-structure interaction problems for biomedical, turbomachinery, and wind-turbine applications, amongst others. The sheer volume of their research on this topic is quite impressive, and we will not attempt to cover it all here. However, the interested reader is encouraged to consult~\cite{tezduyar2019space} for a concise review. 

To the authors' knowledge, no one has explicitly constructed high-order finite element spaces which are the equivalents of H(curl)- or H(div)-conforming spaces on pentatopes or tetrahedral prism elements.  Now, it is important to note that there are inherent difficulties associated with constructing these finite element spaces due to the absence of a complete tensor-product structure in all four coordinate directions. 
Fortunately, this exercise is still made possible by the tools of FEEC~\cite{arnold2006finite,arnold2010finite,arnold2018finite}. We refer the interested reader to part I of this paper for a detailed review of FEEC and its related publications. In this work, we will only focus on a few of these publications that are directly relevant. It turns out that FEEC techniques have already been previously employed to \emph{implicitly} construct high-order conforming finite element spaces on the pentatope (see Arnold et al.~\cite{arnold2006finite}), and on the tetrahedral prism (see Natale~\cite{natale2017structure} and McRae et al.~\cite{mcrae2016automated}). The goal of this paper is to extend this work, and generate \emph{explicit} expressions for the high-order conforming finite element spaces, basis functions, and degrees of freedom for the pentatope and tetrahedral prism. 
We believe that these explicit representations are essential to facilitating implementation and utilization of the elements by scientists and engineers. 

\subsection{Overview of the Paper}

The remainder of this paper is outlined as follows. In section 2, we introduce some notation and essential ideas. In section 3, we introduce our particular de Rahm complex, and the associated derivative operators, Sobolev spaces, and maps. In sections 4 and 5, we present explicit conforming finite element spaces on the pentatope and tetrahedral prism, respectively. Finally, in section 6, we summarize the contributions of this paper. 

\input{guiding}

\section{Sobolev Spaces and Associated Mappings}\label{sobolev_sec}

In accordance with the standard FEEC approach, we introduce the \emph{de Rahm} complex for smooth functions in three dimensions
\begin{align*}
	\begin{matrix}
		C^{\infty}(\Omega, \mathbb{R}) & \arrow{r}{\nabla} & C^{\infty}(\Omega, \mathbb{R}^{3}) & \arrow{r}{\nabla \times} & C^{\infty}(\Omega, \mathbb{R}^{3}) & \arrow{r}{\nabla \cdot} & C^{\infty}(\Omega, \mathbb{R}).
	\end{matrix}
\end{align*}
Here, we have used the conventional derivative operators for three dimensions: namely, $\nabla$ denotes the gradient, $\nabla \times$ denotes the curl, and $\nabla \cdot$ denotes the divergence.

Next, we can also define the de Rahm complex for smooth functions in four dimensions
\begin{align*}
	\begin{matrix}
		C^{\infty}(\Omega, \mathbb{R}) & \arrow{r}{\mathrm{grad}} & C^{\infty}(\Omega, \mathbb{R}^{4}) & \arrow{r}{\mathrm{skwGrad}} & C^{\infty}(\Omega, \mathbb{K}) & \arrow{r}{\mathrm{curl}} & C^{\infty}(\Omega, \mathbb{R}^{4})  & \arrow{r}{\mathrm{div}} & C^{\infty}(\Omega,\mathbb{R}).
	\end{matrix}
\end{align*}
Here, we have introduced new first-derivative operators for functions in $\mathbb{R}^4$. We will provide precise definitions for these operators in what follows.

In four dimensions, `$\text{grad}$' is the standard gradient operator which can be applied to a scalar, $u \in L^{2}\left(\Omega, \mathbb{R} \right)$, such that $\left[\text{grad} u\right]_{i} = \partial_{i} u$ for $i = 1, 2, 3, 4$. In addition,  `$\text{skwGrad}$' is an antisymmetric gradient operator which can be applied to a 4-vector, $E \in L^{2}\left(\Omega, \mathbb{R}^{4} \right)$, as follows
\begin{align*}
    \left[ \text{skwGrad} \, E \right] = \frac{1}{2} \left( 
    \left[ \text{Grad} \, E\right]^{T} - \left[ \text{Grad} \, E\right] \right),
\end{align*}
where $[\mathrm{Grad} E]_{ij} = \partial_{j}E_{i}$ for $i = 1, 2, 3, 4$ and $j = 1, 2, 3, 4$.
Next, `$\text{curl}$' is a derivative operator which can be applied to a $4 \times 4$ skew-symmetric matrix, $F \in L^{2}\left(\Omega, \mathbb{K} \right)$ as follows
\begin{align*}
    \left[\text{curl} \, F \right]_{i} = \sum_{k,l=1}^{4} \varepsilon_{ijkl} \partial_{j} F_{kl},
\end{align*}
where $\varepsilon_{ijkl}$ is the Levi-Civita tensor. Lastly, `$\text{div}$' is the standard divergence operator which acts on a 4-vector, $G \in L^{2}\left(\Omega, \mathbb{R}^{4} \right)$, such that $\left[\text{div} \, G \right] = \partial_{i} G_{i}$ for $i = 1, 2, 3, 4$. 

For the sake of completeness, we can also define the `Curl' and `Div' operators, which are isomorphic to the `skwGrad' and `curl' operators, respectively. In particular, `$\text{Curl}$' is a derivative operator which can be applied to a 4-vector, $E \in L^{2}(\Omega, \mathbb{R}^{4})$, as follows
\begin{align*}
    \left[\text{Curl} \, E \right]_{ij} = \sum_{k,l=1}^{4} \varepsilon_{ijkl} \partial_{k} E_{l},
\end{align*}
and `$\text{Div}$' is a derivative operator which can be applied to a $4 \times 4$ skew-symmetric matrix, $F \in L^{2}(\Omega, \mathbb{K})$, as follows
\begin{align*}
    \left[\text{Div} \, F \right]_{i} = \sum_{j=1}^{4} \partial_{j} F_{ij}.
\end{align*}

It turns out that the first-derivative operators (above) satisfy the following relations
\begin{align*}
    \Upsilon_{1} \left( d^{\left(0\right)} \omega \right) &= \text{grad} \left( \Upsilon_{0} \omega \right), \qquad \qquad \qquad \; \omega \in \Lambda^{0}(\Omega) :=\mathcal{D}'(\Omega,\Lambda^{0}),\\
    \Upsilon_{2} \left( d^{\left(1\right)} \omega \right) &= \text{skwGrad} \left( \Upsilon_{1} \omega \right), \qquad \qquad \omega \in \Lambda^{1}(\Omega) := \mathcal{D}'(\Omega,\Lambda^{1}), \\
    \Upsilon_{3} \left( d^{\left(2\right)} \omega \right) &= \text{curl} \left( \Upsilon_{2} \omega \right), \qquad \qquad \qquad \; \omega \in \Lambda^{2}(\Omega) := \mathcal{D}'(\Omega,\Lambda^{2}), \\
    \Upsilon_{2} \left( d^{\left(3\right)} \omega \right) &= \text{div} \left( \Upsilon_{3} \omega \right), \qquad \qquad \qquad \; \, \omega \in \Lambda^{3}(\Omega) := \mathcal{D}'(\Omega,\Lambda^{3}).
\end{align*}
In accordance with these relations, the following diagram commutes
\begin{align*}
\begin{matrix}
\mathcal{D}'(\Omega,\Lambda^{0}) & \arrow{r}{d^{\left(0\right)}} & \mathcal{D}'(\Omega,\Lambda^{1}) & \arrow{r}{d^{\left(1\right)}} & \mathcal{D}'(\Omega,\Lambda^{2}) & \arrow{r}{d^{\left(2\right)}} & \mathcal{D}'(\Omega,\Lambda^{3})  & \arrow{r}{d^{\left(3\right)}} & \mathcal{D}'(\Omega,\Lambda^{4})                 \cr
\arrow{d}{\Upsilon_0} &                      & \arrow{d}{\Upsilon_1} &   & \arrow{d}{\Upsilon_2} & & \arrow{d}{\Upsilon_3} & & \arrow{d}{\Upsilon_4} \cr
\mathcal{D}'(\Omega, \mathbb{R})                   & \arrow{r}{\text{grad}} & \mathcal{D}'(\Omega, \mathbb{R}^{4}) &  \arrow{r}{\text{skwGrad}} & \mathcal{D}'(\Omega, \mathbb{K}) &  \arrow{r}{\text{curl}} & \mathcal{D}'(\Omega, \mathbb{R}^{4})  &  \arrow{r}{\text{div}} & \mathcal{D}'(\Omega, \mathbb{R})            \cr
\end{matrix}
\end{align*}

In addition, the first-derivative operators can be used to construct the following Sobolev spaces
\begin{align*}
    H\left( \text{grad}, \Omega, \mathbb{R} \right) &= \left\{u \in L^{2}\left(\Omega, \mathbb{R} \right) : \text{grad} \, u \in L^{2} \left(\Omega,\mathbb{R}^{4} \right) \right\}, \\
    H\left( \text{skwGrad}, \Omega, \mathbb{R}^{4} \right) &= \left\{E \in L^{2}\left(\Omega, \mathbb{R}^{4} \right) : \text{skwGrad} \, E \in L^{2} \left(\Omega, \mathbb{K} \right) \right\}, \\
    H\left( \text{curl}, \Omega, \mathbb{K} \right) &= \left\{F \in L^{2}\left(\Omega, \mathbb{K} \right) : \text{curl} \, F \in L^{2} \left(\Omega,\mathbb{R}^{4} \right) \right\}, \\
    H\left( \text{div}, \Omega, \mathbb{R}^{4} \right) &= \left\{G \in L^{2}\left(\Omega, \mathbb{R}^{4} \right) : \text{div} \, G \in L^{2} \left(\Omega,\mathbb{R} \right) \right\}, 
\end{align*}
and
\begin{align*}   
    H\left( \text{Curl}, \Omega, \mathbb{R}^{4} \right) &= \left\{E \in L^{2}\left(\Omega, \mathbb{R}^{4} \right) : \text{Curl} \, E \in L^{2} \left(\Omega, \mathbb{K} \right) \right\}, \\
    H\left( \text{Div}, \Omega, \mathbb{K} \right) &= \left\{F \in L^{2}\left(\Omega, \mathbb{K} \right) : \text{Div} \, F \in L^{2} \left(\Omega,\mathbb{R}^{4} \right) \right\}.
\end{align*} 
In accordance with these definitions, we can introduce the L2 de Rahm complex in four dimensions
\begin{align*}
	\begin{matrix}
		H(\mathrm{grad}, \Omega, \mathbb{R}) & \arrow{r}{\mathrm{grad}} & H(\mathrm{skwGrad}, \Omega, \mathbb{R}^{4}) & \arrow{r}{\mathrm{skwGrad}} & H(\mathrm{curl}, \Omega, \mathbb{K}) & \arrow{r}{\mathrm{curl}} & H(\mathrm{div},\Omega, \mathbb{R}^{4})  & \arrow{r}{\mathrm{div}} & L^{2}(\Omega,\mathbb{R}).
	\end{matrix}
\end{align*}

Next, it is important for us to characterize the behavior of our function spaces on the boundary of the domain, $\partial \Omega$. With this in mind, we can introduce the following trace identity for 1-forms
\begin{align}
    \nonumber \left(\text{tr}^{(1)} E \right)(F) &= \int_{\partial \Omega} \left( n \times E \right) : F \, ds \\[1.0ex]
     &= \int_{\Omega} \left(\text{Curl} \, E \right) : F \, dx - \int_{\Omega} \left(\text{curl} \, F \right) \cdot E \, dx,  \label{ibp_one_A}
\end{align}
where $E \in H\left( \text{Curl}, \Omega, \mathbb{R}^{4} \right)$ and $F \in H\left( \text{curl}, \Omega, \mathbb{K} \right)$. 
Similarly,
\begin{align}
    \nonumber \left(\text{tr}^{(1)} E \right)(F) &= \frac{1}{2} \int_{\partial \Omega}  \left[  E \otimes n - n \otimes E \right] : F \, ds \\[1.0ex]
    &= \int_{\Omega} \left( \text{Div} \, F \right) \cdot E \, dx - \int_{\Omega} F : \left(\text{skwGrad} \, E \right) \, dx, \label{ibp_one_C}
\end{align}
where $E \in H\left( \text{skwGrad}, \Omega, \mathbb{R}^{4} \right)$ and $F \in H\left( \text{Div}, \Omega, \mathbb{K} \right)$. Next, the following trace identity holds for 2-forms
\begin{align}
    \nonumber \left(\text{tr}^{(2)} F \right)(E) &= \int_{\partial \Omega} \left( n \times F \right) \cdot E \, ds \\[1.0ex]
     &= \int_{\Omega} \left(\text{curl} \, F \right) \cdot E \, dx -\int_{\Omega} \left(\text{Curl} \, E \right) : F \, dx,  \label{ibp_two_A}
\end{align}
where $E \in H\left( \text{Curl}, \Omega, \mathbb{R}^{4} \right)$ and $F \in H\left( \text{curl}, \Omega, \mathbb{K} \right)$.
Finally, the following trace identity holds for 3-forms
\begin{align}
    \nonumber \left(\text{tr}^{(3)} G \right)(u) &= \int_{\partial \Omega} \left(G \cdot n \right) u \, ds \\[1.0ex]
    &= \int_{\Omega} \left(\text{div} \,G \right) u \, dx + \int_{\Omega} G \cdot \left( \text{grad} \, u \right) dx, \label{ibp_three}
\end{align}
where $G \in H(\text{div}, \Omega, \mathbb{R}^{4})$ and $u \in H(\text{grad}, \Omega, \mathbb{R})$. 

There are two cross-product operators which are defined in the trace identities above. In particular, the cross-product operator between a pair of 4-vectors is given by 
\begin{align*}
    \left[M \times N \right]_{ij} = \sum_{k,l=1}^{4} \varepsilon_{ijkl} M_k N_l,
\end{align*}
where $M \in \mathbb{R}^{4}$ and $N \in \mathbb{R}^{4}$. Furthermore, the cross-product operator between a 4-vector and a $4 \times 4$ skew-symmetric matrix is given by
\begin{align*}
    \left[M \times U\right]_{i} = \sum_{k,l=1}^{4} \varepsilon_{ijkl} M_{j} U_{kl},
\end{align*}
 where $M \in \mathbb{R}^{4}$ and $U \in \mathbb{K}$.

In accordance with the equations above, the traces for 0-forms, 1-forms, 2-forms, and 3-forms can be defined as follows
\begin{align*}
    \text{0-forms} \qquad u &= \Upsilon_{0} \omega, \qquad \text{tr}(u) = u\vert_{\partial \Omega}, \\[1.0ex]
    \text{1-forms} \qquad E &= \Upsilon_{1} \omega, \qquad \text{tr}(E) = \frac{1}{2}\left(E \otimes n - n \otimes E \right)\vert_{\partial \Omega}, \\[1.0ex]
    \text{2-forms} \qquad F &= \Upsilon_{2} \omega, \qquad \text{tr}(F) = \left(n \times F\right)\vert_{\partial \Omega}, \\[1.0ex]
    \text{3-forms} \qquad G &= \Upsilon_{3} \omega, \qquad \text{tr}(G) = \left(G\cdot n\right)\vert_{\partial \Omega},
\end{align*}
where
\begin{align*}
    &u \in H \left(\text{grad}, \Omega, \mathbb{R} \right), \qquad \qquad \; \; \, \text{tr}(u) \in H^{1/2}\left(\partial \Omega, \mathbb{R} \right), \\[1.0ex]
    &E \in H \left(\text{skwGrad}, \Omega, \mathbb{R}^{4} \right), \qquad \text{tr}(E) \in  H^{-1/2}\left(\partial \Omega, \mathbb{K} \right), \\[1.0ex]
   & F \in H \left(\text{curl}, \Omega, \mathbb{K} \right), \qquad \qquad \; \; \, \text{tr}(F) \in H^{-1/2}\left(\partial \Omega, \mathbb{R}^{4} \right), \\[1.0ex]
    &G \in H \left(\text{div}, \Omega, \mathbb{R}^{4} \right), \qquad \qquad \; \, \text{tr}(G) \in H^{-1/2}\left(\partial \Omega, \mathbb{R} \right).
\end{align*}
We note that the traces of 4-forms are not well-defined.

It may not be immediately obvious how the trace quantities behave by simply examining the identities above. In order to fix ideas, let us consider an example in which a simply connected Lipschitz domain $\Omega$ has a boundary that (non-trivially) intersects with the hyperplane $x_4=0$. We can set $\partial \Omega \cap \{x_4=0\} = \mathcal{F}$, where $\mathcal{F}$ denotes a facet. In addition, we observe that the unit normal of the facet is $n=[0,0,0,1]^T$. Under these circumstances, we consider a sufficiently smooth $s$-form, $\omega$:
\begin{itemize}
	\item If $s=0$ and $u = \Upsilon_{0} \omega$, then
    \begin{equation}
        \tr{\mathcal F}{u} = u\vert_{\mathcal{F}} = u(x_1,x_2,x_3,0), 
    \end{equation}
    is the restriction of $u$ on to $\mathcal{F}$. The trace can be identified with a scalar field $\Tr{\mathcal F}{u}$, which is a 0-form proxy on $\mathcal{F}$.
	\item If $s=1$ and $E = \Upsilon_{1} \omega$, then
    \begin{align}
        \nonumber \tr{\mathcal F}{E} &= \frac{1}{2} \left(E \otimes n - n \otimes E\right)\vert_{\mathcal{F}} \\[1.0ex] 
        \nonumber &= \frac{1}{2} \begin{bmatrix}
            0 & 0 & 0 & E_1(x_1,x_2,x_3,0) \\[1.0ex]
            0 & 0 & 0 & E_2(x_1,x_2,x_3,0) \\[1.0ex]
            0 & 0 & 0 & E_3(x_1,x_2,x_3,0) \\[1.0ex]
            -E_1(x_1,x_2,x_3,0) & -E_2(x_1,x_2,x_3,0) & -E_3(x_1,x_2,x_3,0) & 0
        \end{bmatrix} \\[1.0ex]
        &= \frac{1}{2} \VtoM{\left[0, 0, E_1(x_1,x_2,x_3,0), 0, E_2(x_1,x_2,x_3,0), E_3(x_1,x_2,x_3,0) \right]^{T}},
        \label{trace_formula_one}
    \end{align}
    is the bivector trace of $E$ on to $\mathcal{F}$. The trace can be identified with a 3-vector $\Tr{\mathcal F}{E}$, which is a 1-form proxy on $\mathcal{F}$. 
    \item If $s=2$ and $F = \Upsilon_{2} \omega$, then
    \begin{align}
        \tr{\mathcal F}{F} = \left(n \times F\right)\vert_{\mathcal{F}} = 2 \begin{bmatrix}
            F_{23}(x_1,x_2,x_3,0) \\[1.0ex]
            -F_{13}(x_1,x_2,x_3,0) \\[1.0ex]
            F_{12}(x_1,x_2,x_3,0) \\[1.0ex]
            0
        \end{bmatrix}, \label{trace_formula_two}
    \end{align}
    is the tangential trace of $F$ on to $\mathcal{F}$. The trace can be identified with a 3-vector $\Tr{\mathcal F}{F}$, which is a 2-form proxy on $\mathcal{F}$.
    \item If $s=3$ and $G = \Upsilon_{3} \omega$, then
    \begin{equation}
        \tr{\mathcal F}{G} = \left(G \cdot n\right)\vert_{\mathcal{F}} = G_{4}(x_1,x_2,x_3,0),
        \label{trace_formula_three}
    \end{equation}
    is the normal trace of $G$ on to $\mathcal{F}$. The trace of a 3-form can be identified with a scalar field $\Tr{\mathcal F}{G}$, which is a 3-form proxy on $\mathcal{F}$.
    \item If $s =4$, then the trace is not well-defined.
\end{itemize}

Lastly, having establishing the Sobolev spaces and the corresponding derivative and trace identities, we introduce the pullback operator $\phi^{\ast}$ of the differential forms $\omega$, as follows
\begin{align}
    u &= \Upsilon_0 \omega, \quad \forall u \in H \left(\text{grad}, \Omega, \mathbb{R} \right), \qquad \qquad \Upsilon_0 \phi^{\ast} \omega = u \circ \phi, \\[1.0ex]
    E &= \Upsilon_{1} \omega, \quad \forall E \in H \left(\text{skwGrad}, \Omega, \mathbb{R}^{4} \right), \quad \;  \Upsilon_{1} \phi^{\ast} \omega = D \phi^{T} \left[E \circ \phi \right], \\[1.0ex] 
    F &= \Upsilon_{2} \omega, \quad \forall F \in H \left(\text{curl}, \Omega, \mathbb{K} \right), \qquad \qquad \; \Upsilon_{2} \phi^{\ast} \omega =  D\phi^{T} \left[F \circ \phi \right] D\phi, \label{new_map_orig} \\[1.0ex] 
    G &= \Upsilon_{3} \omega, \quad \forall G \in H \left(\text{div}, \Omega, \mathbb{R}^{4} \right), \quad \quad \qquad \Upsilon_{3} \phi^{\ast} \omega = \left| D\phi \right| D \phi^{-1} \left[ G \circ \phi \right], \label{piola_map_orig} \\[1.0ex]
    q &= \Upsilon_{4} \omega, \quad \forall q \in L^{2} \left(\Omega, \mathbb{R} \right), \qquad \qquad \qquad \; \; \Upsilon_{4} \phi^{\ast} \omega = \left| D \phi \right| \left[ q \circ \phi \right].
\end{align} 
Here $[D\phi]_{ij} = \partial_{j}\phi_i$ is the Jacobian matrix.

\input{pentatope}

\input{prism}

\section{Conclusion}

This paper has introduced fully explicit H(skwGrad)-, H(curl)-, and H(div)-conforming finite element spaces, basis functions, and degrees of freedom for pentatopes and tetrahedral prism elements. This exercise has been performed with the aid of FEEC. In order to facilitate the implementation of these methods, whenever possible, we have dispensed with the language of differential forms and instead used the language of linear algebra in order to simplify the presentation. We hope that the resulting finite elements will be used extensively in space-time finite element methods. In particular, the finite element spaces on the pentatope should help facilitate space-time methods on both partially-unstructured and fully-unstructured pentatopic meshes. In addition, the finite element spaces on tetrahedral prisms will help facilitate space-time methods on extruded, partially-unstructured, tetrahedral-prismatic meshes. 

Looking ahead, there are many opportunities for continued development of four-dimensional finite elements. For example, we note that the recent paper of Petrov et al.~\cite{petrov2022enabling} develops four-dimensional, hybrid meshes composed of cubic pyramids, bipentatopes, and tesseracts. To our knowledge, finite element spaces on the cubic pyramid and bipentatope have not been investigated outside of Petrov et al.'s work, (which itself only covers the H1- and L2-conforming cases for cubic pyramids). We anticipate that FEEC techniques can be used to construct a broader range of conforming finite element spaces on these non-simplicial elements in the near future.

\section*{Declaration of Competing Interests}

The authors declare that they have no known competing financial interests or personal relationships that could have appeared to influence the work reported in this paper.



\section*{Funding}

This research did not receive any specific grant from funding agencies in the public, commercial, or not-for-profit sectors.

{\footnotesize\bibliography{technical-refs}}

\appendix

\section{Construction of Spaces on the Pentatope} \label{pent_exp}
Following the procedure of~\cite{arnold2006finite}, we begin by defining
\begin{align*} \Vk{k}{s}(\T{4}) = P^{k-1}\Lambda^s \oplus \kappa \tilde{P}^{k-1} \Lambda^{s+1},
\end{align*}
where $\kappa$ is the Koszul differential. We will consider the cases of $s =1,2,3,4$ in what follows. For each case, we will define the Koszul differential and introduce the associated spaces.

\begin{itemize}
	\item {\bf 1-forms:} associated with vectors in $\mathbb{R}^4$
	\begin{align*}
	\omega &= a_1 dx^{1} + a_2 dx^{2} + a_3 dx^{3} + a_4 dx^{4}, \\
	\Upsilon_{1} \omega &:= \left[a_1, a_2, a_3, a_4 \right]^{T}.
	\end{align*}
	It follows that $\kappa \omega$ is a 0-form given as
	\begin{align*}
	\kappa \omega &= a_1 x_{1} + a_2 x_{2} + a_3 x_{3} + a_4 x_{4}, \\
	\Upsilon_{0}\left(\kappa \omega\right) &= a \cdot x.
	\end{align*}
	
\item {\bf 2-forms:} associated with skew-symmetric matrices in $\mathbb{K}$
\begin{align*}
	\omega &= a_{12}dx^{1}\wedge dx^{2} + a_{13} dx^{1} \wedge dx^{3} + a_{14} dx^{1} \wedge dx^{4} \\ 
	&+ a_{23}dx^{2}\wedge dx^{3} + a_{24}dx^{2} \wedge dx^{4} + a_{34} dx^{3} \wedge dx^{4}, \\[1.0ex]
	 \Upsilon_{2} \omega & := \frac{1}{2}\left[\begin{array}{cccc}
	 	0&a_{12}&a_{13}&a_{14}\\[1.0ex]
	 	-a_{12} &0&a_{23}&a_{24}\\[1.0ex]
	 	-a_{13}&-a_{23}&0& a_{34}\\[1.0ex]
	 	-a_{14}&-a_{24}&-a_{34}&0
	 \end{array}\right].
\end{align*}

A Koszul differential on a 2-form is given by %
\begin{align*}
\kappa dx^{i} \wedge dx^{j} = x_i dx^{j} - dx^{i} x_j = x_i dx^{j} - x_j dx^{i}.
\end{align*}
Therefore, the Koszul differential of $\omega$ (given above) can be expanded as follows
\begin{align*}
	\kappa \omega&= a_{12} \kappa(dx^{1} \wedge dx^{2}) + a_{13} \kappa(dx^{1} \wedge dx^{3}) + a_{14} \kappa(dx^{1} \wedge dx^{4}) + a_{23}\kappa(dx^{2}\wedge dx^{3})\\
	& + a_{24} \kappa(dx^{2} \wedge dx^{4}) + a_{34} \kappa(dx^{3} \wedge dx^{4})\\
	&= a_{12}(x_{1}dx^{2} - x_{2} dx^{1})+a_{13}(x_{1}dx^{3} - x_{3}dx^{1})+a_{14}(x_{1}dx^{4}-x_{4}dx^{1})\\
	&+a_{23}(x_{2}dx^{3} - x_{3}dx^{2})+a_{24}(x_{2}dx^{4}-x_{4}dx^{2})+a_{34}(x_{3}dx^{4}-x_{4}dx^{3})\\
	&= (-x_{2}a_{12}-x_{3}a_{13}-x_{4}a_{14})dx^{1} + (x_{1}a_{12}-x_{3}a_{23}-x_{4}a_{24})dx^{2}\\
	& + (x_{1}a_{13}+x_{2}a_{23}-x_{4}a_{34})dx^{3} +(x_{1}a_{14} +x_{2}a_{24}+x_{3}a_{34})dx^{4}\\[1.0ex]
	\Upsilon_{1} \left(\kappa \omega\right) &=  \left[\begin{array}{cccc}
		0&a_{12}&a_{13}&a_{14}\\[1.0ex]
		-a_{12} &0&a_{23}&a_{24}\\[1.0ex]
		-a_{13}&-a_{23}&0& a_{34}\\[1.0ex]
		-a_{14}&-a_{24}&-a_{34}&0
	\end{array}\right] \left[\begin{array}{c}
	-x_{1}\\[1.0ex]
	-x_{2}\\[1.0ex]
	-x_{3}\\[1.0ex]
	-x_{4}
\end{array}\right].
\end{align*}

Hence,
\begin{align*}
 \Vk{k}{1}(\T{4})= P^{k-1} \Lambda^1 \oplus \kappa \tilde{P}^{k-1}\Lambda^2 =  (P^{k-1}(\T{4}))^4 \oplus
 \mathcal{L}((\tilde{P}^{k-1}(\T{4}))^{6})x,
\end{align*}
where the latter term describes skew-symmetric matrices of homogeneous degree $k-1$ multiplying the coordinate vector. This is a strict subset of the set of 4-vectors of degree $k$.

Notice that if we take the dot product of the coordinate vector with our matrix-vector product
\begin{align*}
\left[\begin{array}{cccc}
0&a_{12}&a_{13}&a_{14}\\[1.0ex]
-a_{12} &0&a_{23}&a_{24}\\[1.0ex]
-a_{13}&-a_{23}&0& a_{34}\\[1.0ex]
-a_{14}&-a_{24}&-a_{34}&0
\end{array}\right] \left[\begin{array}{c}
-x_{1}\\[1.0ex]
-x_{2}\\[1.0ex]
-x_{3}\\[1.0ex]
-x_{4}
\end{array}\right],
\end{align*}
we get zero. So another way to write our space is given by
\begin{align*}
\Vk{k}{1}(\T{4}) = (P^{k-1}(\T{4}))^4 \oplus \{ p \in (\tilde{P^k}(\T{4}))^4 \vert p \cdot x =0\}.
\end{align*}

\item {\bf 3-forms:}
associated with vectors in $\mathbb{R}^{4}$
\begin{align*} \omega &= a_{123} dx^{1} \wedge dx^{2} \wedge dx^{3} + a_{134} dx^{1}\wedge dx^{3} \wedge dx^{4} \\
&+ a_{234}dx^{2}\wedge dx^{3} \wedge dx^{4} + a_{124}dx^{1} \wedge dx^{2} \wedge dx^{4}, \\[1.0ex]
\Upsilon_{3}\omega &:= \left[a_{234}, -a_{134}, a_{124},-a_{123}\right]^{T}. \end{align*}
A Kozsul differential on a 3-form is 
$$ \kappa (dx^{i}\wedge dx^{j} \wedge dx^{k}) = x_i dx^{j}\wedge dx^{k} - x_j dx^{i} \wedge dx^{k} + x_k dx^{i} \wedge dx^{j}.$$

Therefore, 
\begin{align*}
	\kappa \omega &= a_{123} \kappa(dx^{1} \wedge dx^{2} \wedge dx^{3}) + a_{134} \kappa (dx^{1}\wedge dx^{3} \wedge dx^{4}) \\
	&+ a_{234}\kappa(dx^{2}\wedge dx^{3} \wedge dx^{4}) + a_{124} \kappa(dx^{1} \wedge dx^{2} \wedge dx^{4})\\
	&= a_{123}(x_{1} dx^{2} \wedge dx^{3}-x_{2}dx^{1} \wedge dx^{3} + x_{3}dx^{1} \wedge dx^{2})\\
	& + a_{134}(x_{1}dx^{3} \wedge dx^{4}-x_{3}dx^{1} \wedge dx^{4}+x_{4} dx^{1} \wedge dx^{3})\\
	& +a_{124}(x_{1}  dx^{2} \wedge dx^{4}-x_{2}dx^{1}  \wedge dx^{4}+x_{4}dx^{1} \wedge dx^{2})\\
	& + a_{234}(x_{2}dx^{3} \wedge dx^{4}-x_{3}dx^{2} \wedge dx^{4}+x_{4}dx^{2}\wedge dx^{3})\\
	&= (x_{3}a_{123}+x_{4}a_{124})dx^{1} \wedge dx^{2} + (-x_{2}a_{123}+x_{4}a_{134})dx^{1} \wedge dx^{3}\\
	& + (-x_{3}a_{134}-x_{2}a_{124}) dx^{1} \wedge dx^{4} + (x_{1}a_{123}+x_{4}a_{234}) dx^{2} \wedge dx^{3}\\
	& + (x_{1}a_{124}-x_{3}a_{234}) dx^{2} \wedge dx^{4} + (x_{1}a_{134}+x_{2}a_{234}) dx^{3} \wedge dx^{4}\\[1.0ex]
	\Upsilon_{2}\left(\kappa \omega\right)& = \frac{1}{2} \left[\begin{array}{cccc}
		0&(x_{3}a_{123}+x_{4}a_{124})&(-x_{2}a_{123}+x_{4}a_{134})&(-x_{3}a_{134}-x_{2}a_{124})\\
		*&0&(x_{1}a_{123}+x_{4}a_{234})& (x_{1}a_{124}-x_{3}a_{234})\\
		*&*&0& (x_{1}a_{134}+x_{2}a_{234})\\
		*&*&*&0		
	\end{array}\right]\\[1.0ex]
&= \frac{a_{123}}{2} \left[\begin{array}{cccc}
	0&x_{3}&-x_{2}&0\\
	-x_{3}&0&x_{1}&0\\
	x_{2}&-x_{1}&0&0\\
	0&0&0&0
\end{array}\right]+\frac{a_{134}}{2} \left[\begin{array}{cccc}
	0&0&x_{4}&-x_{3}\\
	0&0&0&0\\
	-x_{4}&0&0&x_{1}\\
	x_{3}&0&-x_{1}&0
\end{array}\right]\\[1.0ex]
&+\frac{a_{124}}{2} \left[\begin{array}{cccc}
	0&x_{4}&0&-x_{2}\\
	-x_{4}&0&0&x_{1}\\
	0&0&0&0\\
	x_{2}&-x_{1}&0&0
\end{array}\right]+\frac{a_{234}}{2}\left[\begin{array}{cccc}
0&0&0&0\\
0&0&x_{4}&-x_{3}\\
0&-x_{4}&0&x_{2}\\
0&x_{3}&-x_{2}&0
\end{array}\right]\\[1.0ex]
&:= \frac{1}{2} \left( -a_{123}B_4-a_{134}B_2 + a_{124}B_3 + a_{234}B_1 \right).
\end{align*}

Hence,
\begin{align*}
\Vk{k}{2}(\T{4}) &= P^{k-1}\Lambda^2 \oplus \kappa \tilde{P}^{k-1}\Lambda^3 \\
&= \mathcal{L}((P^{k-1}(\T{4}))^6) \oplus \mathrm{span}\left\{ \tilde{P}^{k-1}(\T{4}) B_1 + \tilde{P}^{k-1}(\T{4}) B_2+\tilde{P}^{k-1}(\T{4})B_3+\tilde{P}^{k-1}(\T{4})B_4\right\}.
\end{align*}
It turns out that the space of skew-symmetric matrices of homogeneous degree $k-1$ above has a nice characterization. Each of the matrices $B_i$ looks like a rotation in a coordinate hyperplane. Also, each $B_ix=0$. Therefore, we can write 
\begin{align*}
\Vk{k}{2}(\T{4})= \mathcal{L}((P^{k-1}(\T{4}))^6)  \oplus \left\{B\in\mathcal{L}((\tilde{P}^{k}(\T{4}))^6) \vert B x = 0\right\}.
\end{align*} 

It is worth computing the dimension of 
\begin{align*}
\mathcal{S}_k=\mathrm{span}\left\{ \tilde{P}^{k-1}(\T{4})B_1 + \tilde{P}^{k-1}(\T{4})B_2+\tilde{P}^{k-1}(\T{4})B_3+\tilde{P}^{k-1}(\T{4})B_4\right\}.
\end{align*}
We first write
\begin{align*}
    B_1= \VtoM{\begin{bmatrix}
    0\\0\\0\\x_4\\-x_3\\x_2
\end{bmatrix}}, \; B_2 =\VtoM{\begin{bmatrix}
    0\\-x_4\\ x_3\\0\\0\\ -x_1
\end{bmatrix}}, \; B_3 = \VtoM{\begin{bmatrix}
   x_4\\0\\-x_2\\0\\x_1\\ 0
\end{bmatrix}}, \;  B_4 = \VtoM{\begin{bmatrix}
   -x_3\\x_2\\0\\-x_1\\0\\ 0
\end{bmatrix}}.
\end{align*}
Any particular element $q \in \mathcal{S}_{k}$ can therefore be written as
\begin{equation} \label{eq:appendixproof} q = p_1 \VtoM{\begin{bmatrix}
    0\\0\\0\\x_4\\-x_3\\x_2
\end{bmatrix}}+ p_2\VtoM{\begin{bmatrix}
    0\\-x_4\\ x_3\\0\\0\\ -x_1
\end{bmatrix}}+p_3\VtoM{\begin{bmatrix}
   x_4\\0\\-x_2\\0\\x_1\\ 0
\end{bmatrix}}+p_4\VtoM{\begin{bmatrix}
   -x_3\\x_2\\0\\-x_1\\0\\ 0\end{bmatrix}},
\end{equation}
where 
$p_1,p_2,p_3$, and $p_4$ are each homogeneous polynomials in $\tilde{P}^{k-1}(\T{4})$. The dimension of $\mathcal{S}_{k}$ can then be computed by examining the rank  of the matrix 
\begin{align*}
C =C(x):=\begin{bmatrix}
    0&0&x_4&-x_3\\
    0&-x_4&0&x_2\\
    0&x_3&-x_2&0\\
    x_4&0&0&-x_1\\
    -x_3&0&x_1&0\\
    x_2&-x_1&0&0
\end{bmatrix}.
\end{align*}
The matrix $C$ has rank 3 and nullity 1. We can now examine the kernel. Suppose $[p_1,p_2,p_3,p_4]^T \in (\tilde{P}^{k-1}(\T{4}))^4,$ and $C [p_1,p_2,p_3,p_4]^T=0$. By inspection, we see that 
\begin{align*}
&x_4 p_3-x_3p_4=0, \quad -x_4p_2+x_2p_4=0, \quad x_4 p_1 - x_1p_4=0, \\[1.0ex] &\Rightarrow p_1 =\frac{x_1}{x_4}p_4, \quad p_2 =\frac{x_2}{x_4}p_4, \quad p_3 =\frac{x_3}{x_4}p_4.
\end{align*}
Therefore, $p_4$ must be divisible by $x_4$. Now, the set of homogeneous polynomials in $(x_1,x_2,x_3,x_4)$ which are divisible by $x_4$ have a dimension given by $\mathrm{dim}(\tilde{P}^{k-1}(\T{4})) - \mathrm{dim} (\tilde{P}^{k-1}(\T{3}))$. 

Therefore, the dimension of the subset of $(\tilde{P}^{k-1}(\T{4}))^4$ which is associated with $\mathcal{S}_k$ is precisely 
\begin{align*}
\mathrm{dim}( (\tilde{P}^{k-1}(\T{4}))^{4}) - \left( \mathrm{dim}(\tilde{P}^{k-1}(\T{4})) - \mathrm{dim}(\tilde{P}^{k-1}(\T{3})) \right) = \mathrm{dim}((\tilde{P}^{k-1}(\T{4}))^{3}) + \mathrm{dim}(\tilde{P}^{k-1}(\T{3})). 
\end{align*}

\item {\bf 4-forms:} associated with scalar functions in $\mathbb{R}$. The generic 4-form is 
\begin{align*}
    \omega = a_{1234} dx^{1}\wedge dx^{2}\wedge dx^{3} \wedge dx^{4},   
\end{align*}
and the Koszul differential is
\begin{align*}
\kappa \omega &= \kappa (a_{1234} dx^{1}\wedge dx^{2}\wedge dx^{3} \wedge dx^{4}) \\
&= a_{1234} \left(x_{1} dx^{2}\wedge dx^{3} \wedge dx^{4}-x_{2}dx^{1}\wedge dx^{3} \wedge dx^{4}+x_{3}dx^{1}\wedge dx^{2} \wedge dx^{4}-x_{4}dx^{1}\wedge dx^{2}\wedge dx^{3}\right), \\[1.0ex]
\Upsilon_{3}\left(\kappa \omega\right) &=a_{1234} x.
\end{align*}

Therefore, 
\begin{align*}
\Vk{k}{4}(\T{4}) = P^{k-1}\Lambda^3 \oplus \kappa \tilde{P}^{k-1}\Lambda^{4} = (P^{k-1}(\T{4}))^4 \oplus \tilde{P}^{k-1}(\T{4}) x.
\end{align*}

\end{itemize}

\section{Construction of Spaces on the Tetrahedral Prism} \label{tet_exp}

In accordance with the notation in Eq.~\eqref{cross_def}, the $U$ terms are given by
\begin{align*}
    U_0 &:= P^{k} \left(\T{3}\right), \\[1.0ex]
    U_1 &:= \Big\{p = p_1 dx^1 + p_2 dx^2 + p_3 dx^3 =  \tilde{P}^{k}\left(\T{3}\right) dx^1 + \tilde{P}^{k}\left(\T{3}\right) dx^2 + \tilde{P}^{k}\left(\T{3}\right) dx^3, \\[1.0ex] & x = x_1 dx^1 + x_2 dx^2 + x_3 dx^3 \Big| \left(p,x\right) =0 \Big\} \\[1.0ex]
    &\oplus P^{k-1} \left(\T{3}\right) dx^1 + P^{k-1} \left(\T{3}\right) dx^2 + P^{k-1} \left(\T{3}\right) dx^3, \\[1.0ex]
    U_2 &:= P^{k-1}\left(\T{3}\right) dx^1 + P^{k-1}\left(\T{3}\right) dx^2 + P^{k-1}\left(\T{3}\right) dx^3 \oplus \tilde{P}^{k-1}\left(\T{3}\right) \left(x_1 dx^1 + x_2 dx^2 + x_3 dx^3\right), \\[1.0ex]
     &= P^{k-1}\left(\T{3}\right) dx^2 \wedge dx^3 + P^{k-1}\left(\T{3}\right) dx^1 \wedge dx^3 + P^{k-1}\left(\T{3}\right) dx^1 \wedge dx^2 \\[1.0ex] 
     &\oplus \tilde{P}^{k-1}\left(\T{3}\right) \left(x_1 dx^2 \wedge dx^3 - x_2 dx^1 \wedge dx^3 + x_3 dx^1 \wedge dx^2\right), \\[1.0ex]
    U_3 &:= P^{k-1}\left(\T{3}\right) dx^1 \wedge dx^2 \wedge dx^3,
\end{align*}
and the $W$ terms are given by
\begin{align*}
    W_0 & := P^{k}\left(\T{1}\right), \\[1.0ex]
    W_1 & := P^{k-1}\left(\T{1}\right) dx^4.
\end{align*}
We note that the two formulations of $U_2$ (above) are related to one another via the Hodge star operator $\star$.
With these definitions in mind, we can now construct the following tensor-product spaces
\begin{align*}
    \left(U \times W\right)_{0} = & P^{k}\left(\T{1}\right) \times P^{k}\left(\T{3}\right), \\[1.0ex]
    \left(U \times W\right)_{1} = &P^{k} \left(\T{1}\right) \times \Bigg( \Big\{ p = p_1 dx^1 + p_2 dx^{2} + p_3 dx^{3} = \tilde{P}^{k}\left(\T{3}\right) dx^1 + \tilde{P}^{k}\left(\T{3}\right) dx^{2} + \tilde{P}^{k}\left(\T{3}\right) dx^{3} \Big| \\[1.0ex] 
    & \qquad \qquad x = x_1 dx^1 + x_2 dx^2 + x_3 dx^3 + x_4 dx^4 , \left(p,x\right) =0 \Big\}  \\[1.0ex]
    & \qquad \qquad \oplus \left(P^{k-1} \left(\T{3}\right) dx^1 + P^{k-1} \left(\T{3}\right) dx^2 + P^{k-1} \left(\T{3}\right) dx^3 \right) \Bigg) \\[1.0ex]
    & \oplus P^{k-1}\left(\T{1}\right) \times P^{k} \left(\T{3}\right)  dx^4,
\end{align*}
\begin{align*}    
    \left(U \times W\right)_{2} = & P^{k-1}\left(\T{1}\right) \times \Bigg( \Big\{p_1 dx^1\wedge dx^4 + p_2 dx^2 \wedge dx^4 + p_3 dx^3 \wedge dx^4, \\[1.0ex]
    & \qquad \qquad \quad = \tilde{P}^{k}\left(\T{3}\right) dx^1 \wedge dx^4 + \tilde{P}^{k}\left(\T{3}\right) dx^{2} \wedge dx^4 + \tilde{P}^{k}\left(\T{3}\right) dx^{3} \wedge dx^4 \Big| \\[1.0ex]
    & \qquad \qquad \quad x = x_1 dx^1 + x_2 dx^2 + x_3 dx^3 + x_4 dx^4, \left(p,x\right) =0 \Big\} \\[1.0ex]
    &\qquad \qquad \quad \oplus \left(P^{k-1}\left(\T{3}\right) dx^1 \wedge dx^4 + P^{k-1}\left(\T{3}\right) dx^{2} \wedge dx^4 + P^{k-1}\left(\T{3}\right) dx^{3} \wedge dx^4 \right) \Bigg) \\[1.0ex]
    &\oplus P^{k}\left(\T{1}\right) \times \Bigg( \left(P^{k-1}\left(\T{3}\right) dx^2 \wedge dx^3 + P^{k-1}\left(\T{3}\right) dx^1 \wedge dx^3 + P^{k-1}\left(\T{3}\right) dx^1 \wedge dx^2 \right) 
    \\[1.0ex]
    &\qquad \qquad \oplus \tilde{P}^{k-1}\left(\T{3}\right) \left( x_1 dx^2 \wedge dx^3 - x_2 dx^1 \wedge dx^3 +  x_3 dx^1 \wedge dx^2 \right)  \Bigg),
\end{align*}
\begin{align*}
    \left(U \times W\right)_{3} =& P^{k} \left(\T{1}\right) \times P^{k-1} \left(\T{3}\right)  dx^1 \wedge dx^2 \wedge dx^3 \\[1.0ex]
    &\oplus P^{k-1}\left(\T{1}\right) \times \Bigg(  \Big( P^{k-1}\left(\T{3}\right) dx^2 \wedge dx^3 \wedge dx^4 + P^{k-1}\left(\T{3}\right) dx^1 \wedge dx^3 \wedge dx^4 \\[1.0ex]
    &\qquad \qquad \quad + P^{k-1}\left(\T{3}\right) dx^1 \wedge dx^2 \wedge dx^4 \Big)  \\[1.0ex]
    &\qquad \qquad \quad \oplus \tilde{P}^{k-1}\left(\T{3}\right) \Big( x_1 dx^2 \wedge dx^3 \wedge dx^4 - x_2 dx^1 \wedge dx^3 \wedge dx^4 +  x_3 dx^1 \wedge dx^2 \wedge dx^4 \Big)  \Bigg), \\[1.0ex]
    \left(U \times W\right)_{4} =& P^{k-1}\left(\T{1}\right) \times P^{k-1}\left(\T{3}\right)  dx^1 \wedge dx^2 \wedge dx^3 \wedge dx^4.
\end{align*}

\end{document}

%% file: guiding.tex
\section{Notation and Preliminaries}

This section begins by introducing the reference pentatope and tetrahedral prism elements which will be extensively used throughout the paper. Thereafter, we review some well-known degrees of freedom on tetrahedra and triangular prisms, which will be used extensively in our finite element constructions. 

\subsection{The Reference Pentatope} \label{ref_pentatope_section}

Consider the following definition of a reference pentatope
\begin{align*}
    \widehat{K} := \T{4} := \left\{ x = \left(x_1, x_2, x_3, x_4 \right) \in \mathbb{R}^4 \big| -1 \leq x_1, x_2, x_3, x_4 \leq 1, \; x_1 + x_2 + x_3 + x_4 \leq -2 \right\},
\end{align*}
with vertices
\begin{align*}
    v_{1} &= [1,-1,-1,-1]^T,  \quad   
    v_{2} =[-1,1,-1,-1]^T,   \quad   
    v_{3} = [-1,-1,1,-1]^T, \\[1.0ex]   
    v_{4} &=  [-1,-1,-1,1]^T, \quad 
    v_{5} = [-1,-1,-1,-1]^T.
\end{align*}
Next, we introduce the definition of an arbitrary pentatope $K$ with vertices $v_{1}', v_{2}', \ldots, v_{5}'$.
There exists a bijective mapping between the reference pentatope and the arbitrary pentatope $\phi: \widehat{K} \rightarrow K$, such that
\begin{align*}
    \phi \left(x_1, x_2, x_3, x_4 \right) = \sum_{i=1}^{5} v_{i}' N_{i}\left(x_1, x_2, x_3, x_4 \right),
\end{align*}
where 
\begin{align*}
    N_1 &= \frac{x_1 + 1}{2}, \quad N_2 = \frac{x_2 + 1}{2}, \quad N_3 = \frac{x_3 + 1}{2}, \quad N_4 = \frac{x_4 + 1}{2}, \quad N_5 = -\frac{x_1 + x_2 + x_3 + x_4}{2} - 1.
\end{align*}

\subsection{The Reference Tetrahedral Prism} \label{ref_tetrahedral_prism_section}

Consider the following definition of a reference tetrahedral prism
\begin{align*}
    \widehat{K} := \W{4} := \left\{ x = \left(x_1, x_2, x_3, x_4 \right) \in \mathbb{R}^4 \big| -1 \leq x_1, x_2, x_3, x_4 \leq 1, \; x_1 + x_2 + x_3 \leq -1  \right\},
\end{align*}
with vertices
\begin{align*}
    v_{1} &= [
         1 ,
         -1 ,
         -1 ,
         -1
    ]^{T}, \quad 
    v_{2} = [
         -1 ,
         1 ,
         -1 ,
         -1
    ]^{T}, \quad
    v_{3} = [
         -1 ,
         -1 ,
         1 ,
         -1
    ]^{T}, \quad
    v_{4} = [
         -1 ,
         -1 ,
         -1 ,
         -1
    ]^{T}, \\[1.0ex]
    v_{5} &= [
         1 ,
         -1 ,
         -1 ,
         1
    ]^{T}, \quad 
    v_{6} = [
         -1 ,
         1 ,
         -1 ,
         1
    ]^{T}, \quad
    v_{7} = [
         -1 ,
         -1 ,
         1 ,
         1
    ]^{T}, \quad
    v_{8} = [
         -1 ,
         -1 ,
         -1 ,
         1
    ]^{T}.
\end{align*}
Next, we introduce the definition of an arbitrary tetrahedral prism $K$ with vertices $v_{1}', v_{2}', \ldots, v_{8}'$.
There exists a bijective mapping between the reference tetrahedral prism and the arbitrary tetrahedral prism $\phi: \widehat{K} \rightarrow K$, such that
\begin{align*}
    \phi \left(x_1, x_2, x_3, x_4 \right) =  \sum_{i=1}^{8} v_{i}' N_{i}\left(x_1, x_2, x_3, x_4 \right),
\end{align*}
where
\begin{align*}
    N_1 &= \frac{1}{4} \left(x_1 + 1\right)\left(x_4-1\right), \quad N_2 = \frac{1}{4} \left(x_2 + 1\right)\left(x_4-1\right), \\
    N_3 &= \frac{1}{4}\left(x_3 + 1\right)\left(x_4-1\right) , \quad N_4 = -\frac{1}{4}\left(x_1 + x_2 + x_3 + 1\right)\left(x_4-1\right), \\
    N_5 &= \frac{1}{4} \left(x_1 + 1\right)\left(x_4+1\right), \quad N_6 = \frac{1}{4} \left(x_2 + 1\right)\left(x_4+1\right), \\
    N_7 &= \frac{1}{4}\left(x_3 + 1\right)\left(x_4+1\right) , \quad N_8 = -\frac{1}{4}\left(x_1 + x_2 + x_3 + 1\right)\left(x_4+1\right).
\end{align*}
Figures~\ref{fig:pentatope} and \ref{fig:tetrahedral_prism} illustrate a generic pentatope and tetrahedral prism, respectively.
\begin{figure}[h!]
    \centering
    \includegraphics[width=10cm]{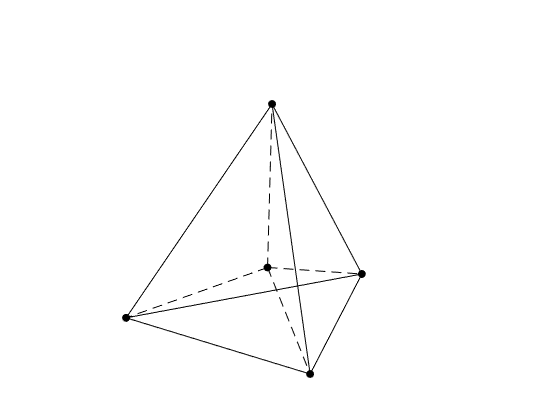}
    \caption{Illustration of a generic pentatope.}
    \label{fig:pentatope}
\end{figure}
\begin{figure}[h!]
    \centering
    \vspace{2cm}
    \includegraphics[width=5.75cm]{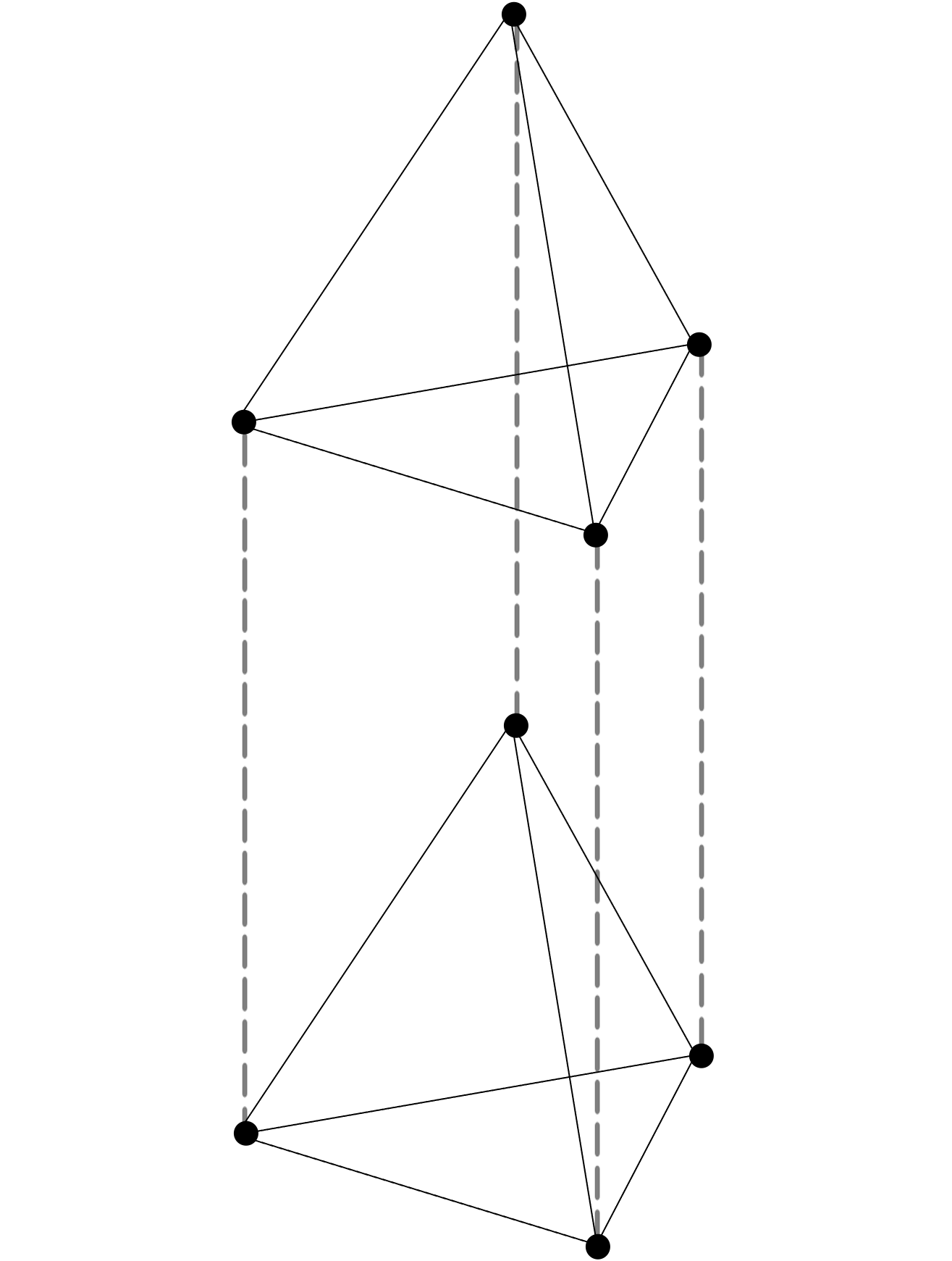}
    \caption{Illustration of a generic tetrahedral prism.}
    \label{fig:tetrahedral_prism}
\end{figure}
We summarize geometric information regarding these four-dimensional elements in~\autoref{Table:4DelementsSubs}. Here, the $d$-dimensional simplex, simplicial prism, and cube are denoted respectively by $\T{d}$, $\W{d}$, and $\C{d}$.
\begin{table}[h!]
    \centering
    \begin{tabular}{c|cc}\hline
    & Pentatope & Tetrahedral Prism  \\
    & $\T{4}$&$\W{4}$\\ \hline
    Vertices&5&8\\ \hline
    Edges&10&16\\ \hline
    Triangular faces $\T{2}$&10&8\\ \hline
    Quadrilateral faces $\C{2}$&0 &6\\ \hline
    Tetrahedral facets $\T{3}$&5&2\\ \hline
    Triangular prism facets $\W{3}$&0&4
    \end{tabular}
    \caption{Geometric information for reference pentatope and tetrahedral prism.}
    \label{Table:4DelementsSubs}
\end{table}

\subsection{Notation}
Let us begin by introducing the following generic set of differential forms on the contractible region $\Omega$
\begin{align*}
    \text{0-forms}, \qquad \omega \in \Lambda^{0}(\Omega), \qquad \omega &= \omega, \\[1.0ex]
    \text{1-forms}, \qquad \omega \in \Lambda^{1}(\Omega), \qquad \omega &= \omega_{1} dx^{1} + \omega_{2} dx^{2} + \omega_{3} dx^{3} + \omega_{4} dx^{4}, \\[1.0ex]
    \text{2-forms} , \qquad \omega \in \Lambda^{2}(\Omega), \qquad \omega &= \omega_{12} dx^{1} \wedge dx^{2} +\omega_{13} dx^{1} \wedge dx^{3} + \omega_{14} dx^{1} \wedge dx^{4} \\
    & + \omega_{23} dx^{2} \wedge dx^{3} + \omega_{24} dx^{2} \wedge dx^{4} + \omega_{34} dx^{3} \wedge dx^{4}, \\[1.0ex]
    \text{3-forms}, \qquad \omega \in \Lambda^{3}(\Omega), \qquad \omega &= \omega_{123} dx^{1} \wedge dx^{2} \wedge dx^{3} + \omega_{124} dx^{1} \wedge dx^{2} \wedge dx^{4} \\
    &+ \omega_{134} dx^{1} \wedge dx^{3} \wedge dx^{4} + \omega_{234} dx^{2} \wedge dx^{3} \wedge dx^{4}, \\[1.0ex]
    \text{4-forms}, \qquad \omega \in \Lambda^{4}(\Omega), \qquad \omega &= \omega_{1234} dx^{1} \wedge dx^{2} \wedge dx^{3} \wedge dx^{4},
\end{align*}
where $\omega \in \Lambda^{s}(\Omega)$ for $s = 0, 1, 2, 3, 4$ is the space of $s$-forms. For the sake of convenience, each differential form has a proxy which is obtained by applying a conversion operator (denoted by $\Upsilon_k$, $k = 0, 1, 2, 3, 4$) to each such that
\begin{align*}
    \Upsilon_{0} \omega &= \omega, \qquad  \Upsilon_{4} \omega  = \omega_{1234}, \\[1.0ex]
    \Upsilon_{1} \omega &= \begin{bmatrix} 
    \omega_1 \\[1.0ex] \omega_2 \\[1.0ex] \omega_3 \\[1.0ex] \omega_4 \end{bmatrix}, \quad 
    \Upsilon_{2} \omega = \frac{1}{2} \begin{bmatrix}
    0 & \omega_{12} & \omega_{13} & \omega_{14} \\[1.0ex]
    -\omega_{12} & 0 & \omega_{23} & \omega_{24} \\[1.0ex]
    -\omega_{13} & -\omega_{23} & 0 & \omega_{34} \\[1.0ex]
    -\omega_{14} & -\omega_{24} & -\omega_{34} & 0
    \end{bmatrix}, \quad
    \Upsilon_{3} \omega = \begin{bmatrix} \omega_{234} \\[1.0ex] -\omega_{134} \\[1.0ex] \omega_{124} \\[1.0ex] -\omega_{123} \end{bmatrix}.
\end{align*}
Let us denote by $\Vk{k}{s}(\Omega)$ the space of $k$-th order polynomial shape functions for the $s$-forms on~$\Omega$. Next, we let $\Sk{k}{s}(\Omega)$ denote the {\it degrees of freedom (dofs)}. Technically speaking, the dofs are a collection of linear functionals on $\Vk{k}{s}(\Omega)$ which is dual to this space. Throughout this paper, we will construct explicit descriptions of finite element triples of the following form
\begin{align*}
    \left(\widehat{K},\Vk{k}{s}(\widehat{K}), \Sk{k}{s}(\widehat{K})\right).
\end{align*}
Here, $\widehat{K}= \T{4}$ and $\W{4}$ are the reference elements. 

In accordance with standard conventions, we let $P^k(x_1,x_2,x_3,x_4)$ denote the space of polynomials of degree $\leq k$. In addition, we use $\tilde{P}^k$ to represent the space of homogeneous polynomials of total degree exactly $k$. Generally speaking, if we set $x=(x_1,x_2,x_3,x_4)$ then it follows that
\begin{align*}
    \sum_{|\alpha| \leq k } a_{\alpha} x^\alpha \in P^k(x_1,x_2,x_3,x_4), \qquad  \sum_{|\alpha| = k} a_{\alpha} x^\alpha \in \tilde{P}^k(x_1,x_2,x_3,x_4),
\end{align*}
where $\alpha$ is the multi-index, and $a_{\alpha}$ are constants. In almost all cases, we will suppress the arguments $(x_1,x_2,x_3,x_4)$ for the sake of brevity. 

Next, the symbol $Q^{l,m,n,q}(x_1,x_2,x_3,x_4)$ denotes standard tensorial polynomials of maximal degree $l,m,n,q$. These polynomials can be written explicity as follows
\begin{align*}
 Q^{l,m,n,q}(x_1,x_2,x_3,x_4)=P^l(x_1)P^m(x_2)P^n(x_3)P^q(x_4).
\end{align*}

In addition, consider a bijective map from 6-vectors to skew-symmetric matrices  in $\mathbb{K}$:
\begin{align}\label{eq:Ldefine}
\VtoM{\cdot}:\mathbb{R}^6 \rightarrow \mathbb{K} : \VtoM{\begin{bmatrix}
	w_{12}\\
	w_{13}\\
	w_{14}\\
	w_{23}\\
	w_{24}\\
	w_{34}
	\end{bmatrix}}:= \begin{bmatrix}
0&w_{12}&w_{13}&w_{14}\\
-w_{12}&0&w_{23}&w_{24}\\
-w_{13}&-w_{23}&0&w_{34}\\
-w_{14}&-w_{24}&-w_{34}&0
\end{bmatrix}.
\end{align}
Lastly, we introduce the following two operators that denote the trace of a quantity $u$ on to a $n$-dimensional submanifold $f$:
\begin{align*}
   \mathrm{tr}[f](u), \qquad  \Tr{f}{u}.
\end{align*}
In almost all cases, the argument $[f]$ is omitted when the submanifold of interest is clear. The first trace operator $\mathrm{tr}[f](u)$ refers to a well-defined restriction of $u$ to $f$, where the restriction is a scalar, 4-vector, or $4 \times 4$ matrix. The second trace operator $\mathrm{Tr}[f](u)$ refers to a well-defined restriction of $u$ to $f$, where the restriction is a scalar, $n$-vector, or $n \times n$ matrix. Generally speaking, there is (at least) a surjective map between the ranges of the two trace operators, such that
\begin{align*}
    \Xi: \mathrm{tr}[f](u) \longrightarrow \mathrm{Tr}[f](u).
\end{align*}
Here, we mean that the scalar, 4-vector, or $4\times 4$ matrix which is denoted by $\mathrm{tr}[f](u)$ can always be identified with a scalar, $n$-vector, or $n\times n$ matrix which is denoted by $\mathrm{Tr}[f](u)$.

\subsection{Degrees of Freedom}
Throughout this paper, our construction of finite element triples will strongly depend on the use of well-known dofs on tetrahedral and triangular-prismatic {\it facets}, and their associated faces, edges, and vertices. For the sake of completeness, we review these dofs in what follows. 
 
\subsubsection{Vertex Degrees of Freedom}
In accordance with standard principles from differential geometry, vertex degrees of freedom are only well-defined for 0-forms. For the pentatope and tetrahedral prism, we specify the vertex degrees of freedom $\Sigma^{k,0}(v)$ as the vertex values of the polynomial 0-form. We note that there are 5 such degrees of freedom for the pentatope $\T{4}$ and 8 such degrees of freedom for the tetrahedral prism $\W{4}$.

\subsubsection{Edge Degrees of Freedom}
Next, we recall that the edge degrees of freedom are only defined for 0- and 1-forms. One may define $e$ as an edge of an element $\widehat{K}$, where $\widehat{K}$ = $\T{4}$ or $\W{4}$, and let $u \in \Vk{k}{0}(\widehat{K})$ be a 0-form proxy. Next, one may construct edge degrees of freedom for $u$ as follows
\begin{equation}\label{eq:edge-0}
M_{e}(u) :=\left\{ \int_{e} \text{Tr}(u) q, \, \qquad \forall q \in P^{k-2}(e), \qquad  \mbox{ for each edge } \, e \, \mbox{ of }\, \widehat{K} \right\}.
\end{equation}
For pentatopes $\T{4}$ there are $10(k-1)$ such degrees of freedom, and for tetrahedral prisms $\W{4}$ there are $16(k-1)$ such degrees of freedom.

Consider a 1-form proxy $U \in \Vk{k}{1}(\widehat{K})$. One may define its edge degrees of freedom as follows
\begin{equation} \label{eq:edge-1}
M_{e}(U) :=\left\{ \int_{e} \mathrm{Tr} (U)\cdot \tau q, \, \qquad \forall q \in P^{k-1}(e), \qquad  \mbox{ for each edge } \, e \, \mbox{ of }\, \widehat{K} \right\},
\end{equation} 
where $\tau$ is a unit vector in the direction of $e$.
For $\T{4}$, there are $10k$ such degrees of freedom, and for $\W{4}$, there are $16k$ such degrees of freedom.

\subsubsection{Face Degrees of Freedom}
In addition, we recall that face degrees of freedom are only defined for 0-, 1- and 2-forms. We let $f$ denote a single face of an element $\widehat{K}$; in the context of the present paper, this face can be either triangular or quadrilateral.

{\bf Triangular faces:}
Consider polynomial 0-forms $u\in \Vk{k}{0}(\widehat{K})$. The associated face degrees of freedom on a triangular face $f=\T{2}$ are defined as
\begin{equation}\label{eq:triangle-0}
M_f(u):=\left\{ \int_f \text{Tr}(u)q, \quad \forall q \in P^{k-3}(f) \right\}.
\end{equation}
Consider polynomial 1-forms $U \in \Vk{k}{1}(\widehat{K})$ for which the face degrees of freedom are
\begin{equation}\label{eq:triangle-1}
M_f(U):=\left\{ \int_f (\mathrm{Tr}(U)\times \nu) \cdot (q \times \nu), \quad \forall q \in (P^{k-2}(f))^{2}, \quad q \cdot \nu =0 \right\},
\end{equation}
where $\nu$ denotes a unit normal vector to the face $f$. These definitions are slightly different from the standard definitions of edge degrees of freedom for Nedelec-type elements, but can be shown to be the same, (see Remark 5.31 in~\cite{monkbook}). 

Finally, consider polynomial 2-forms, $U \in \Vk{k}{2}(\widehat{K})$ for which the face degrees of freedom are 
\begin{equation}\label{eq:triangle-2}
M_f(U):=\left\{ \int_f (\mathrm{Tr}(U)\cdot \nu) q, \quad \forall q \in P^{k-1}(f) \right\}.
\end{equation}

{\bf Quadrilateral faces:}
Consider polynomial 0-forms $u\in \Vk{k}{0}(\widehat{K})$. The associated face degrees of freedom on a quadrilateral face $f = \C{2}$ are defined as
\begin{equation}\label{eq:square-0}
M_f(u):=\left\{ \int_f \text{Tr}(u)q, \quad \forall q \in Q^{k-2,k-2}(f) \right\}.
\end{equation}
Consider polynomial 1-forms $U \in \Vk{k}{1}(\widehat{K})$ for which the face degrees of freedom are
\begin{equation}\label{eq:square-1}
M_f(U):=\left\{ \int_f (\mathrm{Tr}(U)\times \nu) \cdot q, \quad \forall q \in Q^{k-2,k-1}\times Q^{k-1,k-2}(f) \right\}.
\end{equation}
Finally, consider polynomial 2-forms $U \in \Vk{k}{2}(\widehat{K})$ for which the face degrees of freedom are 
\begin{equation}\label{eq:square-2}
M_f(U):=\left\{ \int_f (\mathrm{Tr}(U)\cdot \nu) q, \quad \forall q \in Q^{k-1,k-1}(f) \right\}.
\end{equation}

\subsubsection{Facet Degrees of Freedom}
We recall that facet degrees of freedom are only defined for 0-, 1-, 2-, and 3-forms. The elements under consideration in this paper will only have tetrahedral or triangular-prismatic facets.  

{\bf Tetrahedral facets:}
Let $\mathcal{F}=\T{3}$ denote a tetrahedral facet.
For 0-forms $u \in \Vk{k}{0}(\widehat{K})$, we can specify facet degrees of freedom as
\begin{equation} \label{eq:tet0} M_{\mathcal{F}}(u) := \left\{ \int_{\mathcal{F}} \mathrm{Tr}(u) q, \qquad  q \in P^{k-4}(\mathcal{F}) \right\}. \end{equation}
For polynomial 1-forms $U \in \Vk{k}{1}(\widehat{K})$, we specify the facet degrees of freedom as
\begin{equation}\label{eq:tet1}
M_{\mathcal{F}}(U) := \left\{ \int_{\mathcal{F}} \mathrm{Tr}(U) \cdot  q, \qquad  q \in (P^{k-3}(\mathcal{F}))^3\right\}.
\end{equation}
For polynomial 2-forms $U \in \Vk{k}{2}(\widehat{K})$, we specify the facet degrees of freedom as
\begin{equation}\label{eq:tet2}
M_{\mathcal{F}}(U) := \left\{ \int_{\mathcal{F}} \mathrm{Tr}(U) \cdot  q, \qquad  q \in (P^{k-2}(\mathcal{F}))^3\right\}.
\end{equation}
Lastly, for polynomial 3-forms $U \in \Vk{k}{3}(\widehat{K})$, we specify the facet degrees of freedom as
\begin{equation}\label{eq:tet3}
M_{\mathcal{F}}(U) := \left\{ \int_{\mathcal{F}} \text{Tr}(U)  q, \qquad  q \in P^{k-1}(\mathcal{F})\right\}.
\end{equation}

{\bf Triangular-prismatic facets:}
Let $\mathcal{F}=\W{3}$ denote a triangular-prismatic facet.
For 0-forms $u \in \Vk{k}{0}(\widehat{K})$, we can specify facet degrees of freedom as
\begin{equation} M_{\mathcal{F}}(u) := \left\{ \int_{\mathcal{F}} \mathrm{Tr}(u) q, \qquad  q \in Q^{k-2}(\C{1}) \times P^{k-3}(\T{2}) \right\}.\label{eq:triprism0} \end{equation}
For polynomial 1-forms $U \in \Vk{k}{1}(\widehat{K})$, we specify the facet degrees of freedom as
\begin{align}
\nonumber M_{\mathcal{F},1}(U) &:= \left\{ \int_{\mathcal{F}} \mathrm{Tr}(U) \cdot  q, \qquad  q \in Q^{k-2}(\C{1}) \times \left[P^{k-2}(\T{2}), P^{k-2}(\T{2}), 0\right]^{T} \right\}, \\[1.0ex]
\nonumber M_{\mathcal{F},2}(U) &:= \left\{ \int_{\mathcal{F}} \mathrm{Tr}(U) \cdot q, \qquad  q \in Q^{k-1}(\C{1}) \times \left[0,0, P^{k-3}(\T{2}) \right]^{T} \right\}, \\[1.0ex]
M_{\mathcal{F}}(U) &:= M_{\mathcal{F},1}(U) \cup M_{\mathcal{F},2}(U). \label{eq:triprism1}
\end{align}
For polynomial 2-forms $U \in \Vk{k}{2}(\widehat{K})$, we specify the facet degrees of freedom as
\begin{align}
\nonumber M_{\mathcal{F},1}(U) &:= \left\{ \int_{\mathcal{F}} \mathrm{Tr}(U) \cdot q, \qquad  q \in Q^{k-1}(\C{1}) \times \left[ P^{k-2}(\T{2}), P^{k-2}(\T{2}), 0\right]^{T} \right\}, \\[1.0ex]
\nonumber M_{\mathcal{F},2}(U) &:= \left\{ \int_{\mathcal{F}} \mathrm{Tr}(U) \cdot q, \qquad  q \in Q^{k-2}(\C{1}) \times \left[0, 0, P^{k-1}(\T{2}) \right]^{T} \right\}, \\[1.0ex]
M_{\mathcal{F}}(U) &:= M_{\mathcal{F},1}(U) \cup M_{\mathcal{F},2}(U). \label{eq:triprism2}
\end{align}
Lastly, for polynomial 3-forms $U \in \Vk{k}{3}(\widehat{K})$, we specify the facet degrees of freedom as
\begin{equation} M_{\mathcal{F}}(U) := \left\{ \int_{\mathcal{F}} \mathrm{Tr}(U) q, \qquad  q \in Q^{k-1}(\C{1}) \times P^{k-1}(\T{2}) \right\}.\label{eq:triprism3} \end{equation}

%% file: pentatope.tex
\section{Finite Elements on a Reference Pentatope}
In this section, we record explicitly, finite element spaces and degrees of freedom on a pentatope, $\T{4}$. The construction we choose for $\Vk{k}{s}(\T{4})$ is based on those presented in \cite{monkbook}; these are directly analogous to the $P_k^-\Lambda^s$ spaces on a tetrahedron, as described in  \cite{arnold2006finite}.

We require that our spaces $\Vk{k}{s}(\T{4}) $ satisfy the relation
\begin{align*}
	\begin{matrix}
		\Vk{k}{0}(\T{4}) & \arrow{r}{d^{\left(0\right)}} & \Vk{k}{1}(\T{4}) & \arrow{r}{d^{\left(1\right)}} & \Vk{k}{2}(\T{4}) & \arrow{r}{d^{\left(2\right)}} & \Vk{k}{3}(\T{4})  & \arrow{r}{d^{\left(3\right)}} & \Vk{k}{4}(\T{4}).
	\end{matrix}
\end{align*}
With this in mind, we require that
\begin{subequations}
\begin{align}
	\Vk{k}{0}(\T{4}) &:= P^{k}(\T{4}), \label{rt_pentatope_0} \\[1.0ex]
	\Vk{k}{1}(\T{4}) &:= ({P}^{k-1}(\T{4}))^4 \oplus \left\{ p\in (\tilde{P}^k(\T{4}))^4 \vert p \cdot x =0\right\}, \label{rt_pentatope_1} \\[1.0ex]
	\Vk{k}{2}(\T{4}) &:= \mathcal{L}\left((P^{k-1}(\T{4}))^6\right) \oplus \tilde{P}^{k-1}(\T{4})B_1 \oplus \tilde{P}^{k-1}(\T{4})B_2\oplus\tilde{P}^{k-1}(\T{4})B_3\oplus\tilde{P}^{k-1}(\T{4})B_4, \label{rt_pentatope_2}\\
 \nonumber \\
 \nonumber \mbox{where}\\
 \nonumber
 B_1&:=\begin{bmatrix}
		0&0&0&0\\
		0&0&x_4&-x_3\\
		0&-x_4&0&x_2\\ 
		0& x_3 & -x_2 & 0
	\end{bmatrix},
	\qquad B_2:=\begin{bmatrix}
		0&0&-x_4&x_3\\
		0&0&0&0\\
		x_4&0&0&-x_1\\
		-x_3&0&x_1&0
	\end{bmatrix},\\
 \nonumber
	B_3&:=\begin{bmatrix}
		0&x_4&0&-x_2\\
		-x_4&0&0&x_1\\
		0&0&0&0\\
		x_2&-x_1&0&0
	\end{bmatrix}, 
 \qquad B_4:=\begin{bmatrix}
		0&-x_3&x_2&0\\
		x_3&0&-x_1&0\\
		-x_2&x_1&0&0\\
		0&0&0&0
	\end{bmatrix}, 
	\\[3.0ex]
	\Vk{k}{3}(\T{4}) &:= (P^{k-1}(\T{4}))^4 \oplus \tilde{P}^{k-1}(\T{4}) x, \label{rt_pentatope_3} \\[1.0ex]
	\Vk{k}{4}(\T{4})&:= P^{k-1}(\T{4}) \label{rt_pentatope_4}. 
\end{align} 
\end{subequations}

\begin{remark}\label{remark:pent}
	It is easily seen that the space of 2-forms (Eq.~\eqref{rt_pentatope_2}) can be described as follows \begin{align*}
	\Vk{k}{2}(\T{4}):=\mathcal{L}\left((P^{k-1}(\T{4}))^6\right)  \oplus \left\{B\in \mathcal{L}((\tilde{P}^{k}(\T{4}))^6) \vert B x = 0\right\}.
\end{align*}
\end{remark}
For details on the derivation of these spaces, we refer the interested reader  to~\ref{pent_exp}. The exactness of the sequence follows directly.

It remains for us to identify the bubble spaces $\Vkcirc{k}{s}(\T{4})$. These take the following form
\begin{subequations}
    \begin{align}
        \Vkcirc{k}{0}(\T{4}) &:= \text{span} \left\{ \vartheta_{ij\ell m}(x_1, x_2, x_3, x_4) \right\}, \label{rt_pentatope_0_bubble} \\[1.0ex]
        \Vkcirc{k}{1}(\T{4}) &:= \text{span} \left\{\Phi_{ij\ell m}^{r}(x_1, x_2, x_3, x_4) \right\}, \label{rt_pentatope_1_bubble}  \\[1.0ex]
        \Vkcirc{k}{2}(\T{4}) &:= \text{span} \left\{\Theta_{ij\ell m}^{r} (x_1, x_2, x_3, x_4) \right\}, \label{rt_pentatope_2_bubble}  \\[1.0ex]
        \Vkcirc{k}{3}(\T{4}) &:= \text{span} \left\{\Psi_{ij\ell m}^{r}(x_1, x_2, x_3, x_4) \right\}, \label{rt_pentatope_3_bubble} 
    \end{align}
\end{subequations}
where the basis functions $\vartheta$, $\Phi$, $\Theta$, and $\Psi$ and the associated indexes $i$, $j$, $\ell$, $m$, and $r$ are defined below.

Consider the following H1-conforming interior functions $\vartheta_{ij \ell m}$ of degree $k$
\begin{align*}
    \vartheta_{ij\ell m}(x_1, x_2, x_3, x_4) = &L_{i}\left(\frac{\lambda_{2}}{\lambda_{1} + \lambda_{2}} \right)L_{j}^{2i}\left(\frac{\lambda_{3}}{\lambda_{1}+\lambda_{2}+\lambda_{3}} \right) L_{\ell}^{2(i+j)}\left(\frac{\lambda_{4}}{\lambda_{1}+\lambda_{2}+\lambda_{3}+\lambda_{4}} \right) L_{m}^{2(i+j+\ell)} \left(\lambda_{5}\right) \\[1.0ex] &\cdot 
     \left(\lambda_{1} + \lambda_{2}\right)^{i} \left(\lambda_{1} + \lambda_{2} + \lambda_{3} \right)^{j} \left(\lambda_{1} + \lambda_{2} + \lambda_{3} + \lambda_{4} \right)^{\ell},
\end{align*}
where $i \geq 2$, $j\geq 1$, $\ell \geq 1$, $m \geq 1$, and $n = i + j + \ell + m = 5, \ldots, k$ are the indexing parameters, $\lambda = \lambda(x) = \lambda(x_1, x_2, x_3, x_4) = (\lambda_1, \lambda_2, \lambda_3, \lambda_4, \lambda_5)$ are barycentric coordinates for the pentatope, $L_i$ are the integrated and scaled Legendre polynomials, and $L_j^{\alpha}$ are the integrated and scaled Jacobi polynomials, (see Remark~\ref{functional_remark} for details). 

In addition, consider the following H(skwGrad)-conforming interior functions $\Phi_{ij\ell m}^{r}$ of degree $k-1$
\begin{align*}
    \Phi_{ij\ell m}^{r}(x_1, x_2, x_3, x_4) = &P_{i}\left(\frac{\lambda_{b}}{\lambda_{a} + \lambda_{b}} \right)L_{j}^{2i+1}\left(\frac{\lambda_{c}}{\lambda_{a}+\lambda_{b}+\lambda_{c}} \right) L_{\ell}^{2(i+j)}\left(\frac{\lambda_{d}}{\lambda_{a}+\lambda_{b}+\lambda_{c}+\lambda_{d}} \right) L_{m}^{2(i+j+\ell)} \left(\lambda_{e}\right) \\[1.0ex] &\cdot 
     \left(\lambda_{a} + \lambda_{b}\right)^{i} \left(\lambda_{a} + \lambda_{b} + \lambda_{c} \right)^{j} \left(\lambda_{a} + \lambda_{b} + \lambda_{c} + \lambda_{d} \right)^{\ell}  \left(\lambda_{a} \nabla \lambda_{b} - \lambda_{b} \nabla \lambda_{a} \right), 
\end{align*}
where $i \geq 0$, $j\geq 1$, $\ell \geq 1$, $m \geq 1$, and $n = i + j + \ell + m = 3, \ldots, k-1$ are the indexing parameters, and $P_i$ are the shifted and scaled Legendre polynomials. In addition, for $r = 1,2,3,4$ we set $(a,b,c,d,e) = (1,2,3,4,5)$, $(a,b,c,d,e) = (2,3,4,5,1)$, $(a,b,c,d,e) = (3,4,5,1,2)$, and $(a,b,c,d,e) = (4,5,1,2,3)$, respectively.

The explicit formula given above for the H(skwGrad)-conforming polynomial functions is justified via Lemma~\ref{one_form_lemma}. In this lemma, we focus on the case in which $a = 1$ and $b = 2$, as all other cases are justified using similar arguments.

\begin{lemma}
    The following quantity belongs to the space $\Vk{k+1}{1}(\T{4})$ of H(skwGrad)-conforming functions
    \begin{align*}
        f_{k}\left(x_1, x_2, x_3, x_4\right) \left(\lambda_{1} \nabla \lambda_{2} - \lambda_{2} \nabla \lambda_{1} \right),
    \end{align*}
    where $\lambda_{1} = \lambda_{1} (x_1, x_2, x_3, x_4) $ and $\lambda_{2} = \lambda_{2} (x_1, x_2, x_3, x_4) $ are barycentric coordinates on the pentatope $\T{4}$, and where $f_{k}(x_1, x_2, x_3, x_4) \in P^{k}(\T{4})$.
    \label{one_form_lemma}
\end{lemma}

\begin{proof}
    The proof follows immediately from Lemma 2 of Fuentes et al.~\cite{fuentes2015orientation}, upon setting the number of dimensions $N = 4$, and the parameters $a = 1$ and $b =2$. 
\end{proof}

Next, consider the H(curl)-conforming interior functions $\Theta_{ij\ell m}^{r}$ of degree $k-1$
\begin{align*}
    &\Theta_{ij\ell m}^{r}(x_1, x_2, x_3, x_4) = \\[1.0ex] &P_{i}\left(\frac{\lambda_{b}}{\lambda_{a} + \lambda_{b}} \right)P_{j}^{2i+1}\left(\frac{\lambda_{c}}{\lambda_{a}+\lambda_{b}+\lambda_{c}} \right) L_{\ell}^{2(i+j+1)}\left(\frac{\lambda_{d}}{\lambda_{a}+\lambda_{b}+\lambda_{c}+\lambda_{d}} \right)  L_{m}^{2(i+j+\ell)} \left(\lambda_{e}\right) \\[1.0ex] 
     &\cdot \left(\lambda_{a} + \lambda_{b}\right)^{i} \left(\lambda_{a} + \lambda_{b} + \lambda_{c} \right)^{j} \left(\lambda_{a} + \lambda_{b} + \lambda_{c} + \lambda_{d} \right)^{\ell} \\[1.0ex]
     &\cdot \big[\lambda_{a} \left(\nabla \lambda_{b} \otimes \nabla \lambda_{c} - \nabla \lambda_{c} \otimes \nabla \lambda_{b} \right) + \lambda_{b} \left(\nabla \lambda_{c} \otimes \nabla \lambda_{a} - \nabla \lambda_{a} \otimes \nabla \lambda_{c} \right) + \lambda_{c} \left(\nabla \lambda_{a} \otimes \nabla \lambda_{b} - \nabla \lambda_{b} \otimes \nabla \lambda_{a} \right)\big],
\end{align*}
where $i \geq 0$, $j\geq 0$, $\ell \geq 1$, $m \geq 1$, and $n = i + j + \ell + m = 2, \ldots, k-1$ are the indexing parameters. In addition, for $r = 1,2,3,4,5,6$ we set $(a,b,c,d,e) = (1,2,3,4,5)$,  $(a,b,c,d,e) = (2,3,4,5,1)$, $(a,b,c,d,e) = (3,4,5,1,2)$, $(a,b,c,d,e) = (4,5,1,2,3)$, $(a,b,c,d,e) = (5,1,2,3,4)$, and $(a,b,c,d,e) = (1,2,4,3,5)$, respectively.

The explicit formula given above for H(curl)-conforming polynomial functions is justified via Lemma~\ref{two_form_lemma}. In this lemma, we focus on the case in which $a = 1$, $b = 2$, and $c=3$ as all other cases are justified using similar arguments.

\begin{lemma}
    The following quantity belongs to the space $\Vk{k+1}{2}(\T{4})$ of H(curl)-conforming functions
    \begin{align}
        \nonumber f_{k}\left(x_1, x_2, x_3, x_4\right) &\big[\lambda_{1} \left(\nabla \lambda_{2} \otimes \nabla \lambda_{3} - \nabla \lambda_{3} \otimes \nabla \lambda_{2} \right) + \lambda_{2} \left(\nabla \lambda_{3} \otimes \nabla \lambda_{1} - \nabla \lambda_{1} \otimes \nabla \lambda_{3} \right) \\[1.0ex] &+ \lambda_{3} \left(\nabla \lambda_{1} \otimes \nabla \lambda_{2} - \nabla \lambda_{2} \otimes \nabla \lambda_{1} \right) \big], \label{curl_function}
    \end{align}
    where $\lambda_{1} = \lambda_{1} (x_1, x_2, x_3, x_4) $, $\lambda_{2} = \lambda_{2} (x_1, x_2, x_3, x_4)$, and $\lambda_{3} = \lambda_{3} (x_1, x_2, x_3, x_4)$  are barycentric coordinates on the pentatope $\T{4}$, and where $f_{k}(x_1, x_2, x_3, x_4) \in P^{k}(\T{4})$.
    \label{two_form_lemma}
\end{lemma}

\begin{proof}
    Let us recall that
    \begin{align}
        \Vk{k}{2}(\T{4}) :=\VtoM{(P^{k-1}(\T{4}))^6}  \oplus \left\{B\in \VtoM{(\tilde{P}^{k}(\T{4}))^6} \vert B x = 0\right\}. \label{two_form_restate}
    \end{align}
    It remains for us to show that the function in Eq.~\eqref{curl_function} belongs to $\Vk{k+1}{2}(\T{4})$. Towards this end, we introduce the following identities
    \begin{align*}
        \lambda_{i} = \eta_{i} + \beta_{i} \cdot x, \qquad \nabla \lambda_{i} = \beta_{i},
    \end{align*}
    where $\eta_i \in \mathbb{R}$, $\beta_i \in \mathbb{R}^{4}$, and $i = 1, 2, 3$. It immediately follows that
    \begin{align*}
        &\lambda_{1} \left(\nabla \lambda_{2} \otimes \nabla \lambda_{3} - \nabla \lambda_{3} \otimes \nabla \lambda_{2} \right) + \lambda_{2} \left(\nabla \lambda_{3} \otimes \nabla \lambda_{1} - \nabla \lambda_{1} \otimes \nabla \lambda_{3} \right) + \lambda_{3} \left(\nabla \lambda_{1} \otimes \nabla \lambda_{2} - \nabla \lambda_{2} \otimes \nabla \lambda_{1} \right) \\[1.0ex]
        &= \left(\eta_1 + \beta_1 \cdot x\right)\left(\beta_2 \otimes \beta_3 - \beta_3 \otimes \beta_2 \right) + \left(\eta_2 + \beta_2 \cdot x\right)\left(\beta_3 \otimes \beta_1 - \beta_1 \otimes \beta_3\right) \\[1.0ex]
        &+\left(\eta_3 + \beta_3 \cdot x\right) \left(\beta_1 \otimes \beta_2 - \beta_2 \otimes \beta_1\right) = A + C(x),
    \end{align*}
    where 
    \begin{align*}
        A &:= \eta_1 \left(\beta_2 \otimes \beta_3 - \beta_3 \otimes \beta_2\right) + \eta_2 \left(\beta_3 \otimes \beta_1 - \beta_1 \otimes \beta_3\right) + \eta_3 \left(\beta_1 \otimes \beta_2 - \beta_2 \otimes \beta_1 \right), \\[1.0ex]
        C(x) &:= (\beta_1 \cdot x) \left(\beta_2 \otimes \beta_3 - \beta_3 \otimes \beta_2\right) + (\beta_2 \cdot x) \left(\beta_3 \otimes \beta_1 - \beta_1 \otimes \beta_3\right) + (\beta_3 \cdot x) \left(\beta_1 \otimes \beta_2 - \beta_2 \otimes \beta_1 \right).
    \end{align*}
    By inspection, we have that $A \in \VtoM{(P^0(\T{4}))^6}$ and $C(x) \in \VtoM{(\tilde{P}^{1}(\T{4}))^{6}}$. In addition, following some algebraic manipulations, it turns out that $C(x) x = 0$. Therefore 
    \begin{align*}
        C(x) \in \left\{E \in \VtoM{(\tilde{P}^{1}(\T{4}) )^{6}} \, | \, Ex = 0 \right\}.
    \end{align*}
    Next, we can perform the following decomposition
    \begin{align*}
        f_k \in P^{k}(\T{4}) &= P^{k-1}(\T{4}) \oplus \tilde{P}^{k}(\T{4}), \\[1.0ex]
        f_k &= f_{k-1} + \tilde{f}_{k},
    \end{align*}
    where $f_{k-1} \in P^{k-1}(\T{4})$ and $\tilde{f}_{k} \in \tilde{P}^{k}(\T{4})$. As a result, we have that
    \begin{align*}
        f_k \left[A + C(x) \right] = f_k A + f_{k-1} C(x) + \tilde{f}_{k} C(x).
    \end{align*}
    Naturally, by inspection, we have that
    \begin{align*}
        &f_k A + f_{k-1} C(x) \in \VtoM{(P^{k}(\T{4}))^{6}}, \qquad \tilde{f}_{k} C(x) \in \left\{B \in \VtoM{(\tilde{P}^{k+1}(\T{4}))^{6}} | Bx = 0 \right\}.
    \end{align*}
    Based on these identities and the definition in Eq.~\eqref{two_form_restate}, we immediately obtain the desired result
    \begin{align*}
        f_k \left[A + C(x) \right] \in \Vk{k+1}{2}(\T{4}).
    \end{align*}
\end{proof}

Next, consider the H(div)-conforming interior functions $\Psi_{ij\ell m}^{r}$ of degree $k-1$
\begin{align*}
    \Psi_{ij\ell m}^{r}(x_1, x_2, x_3, x_4) = &P_{i}\left(\frac{\lambda_{b}}{\lambda_{a} + \lambda_{b}} \right)P_{j}^{2i+1}\left(\frac{\lambda_{c}}{\lambda_{a}+\lambda_{b}+\lambda_{c}} \right) P_{\ell}^{2(i+j+1)}\left(\frac{\lambda_{d}}{\lambda_{a}+\lambda_{b}+\lambda_{c}+\lambda_{d}} \right)  L_{m}^{2(i+j+\ell)+3} \left(\lambda_{e}\right) \\[1.0ex] 
     &\cdot \left(\lambda_{a} + \lambda_{b}\right)^{i} \left(\lambda_{a} + \lambda_{b} + \lambda_{c} \right)^{j} \left(\lambda_{a} + \lambda_{b} + \lambda_{c} + \lambda_{d} \right)^{\ell} \\[1.0ex]
     &\cdot \big[ \lambda_{a} \left(\nabla \lambda_{b} \times \nabla \lambda_{c} \times \nabla \lambda_{d} \right) - \lambda_{b} \left(\nabla \lambda_{c} \times \nabla \lambda_{d} \times \nabla \lambda_{a} \right) \\[1.0ex] 
     &+ \lambda_{c} \left(\nabla \lambda_{d} \times \nabla \lambda_{a} \times \nabla \lambda_{b} \right) - \lambda_{d} \left(\nabla \lambda_{a} \times \nabla \lambda_{b} \times \nabla \lambda_{c} \right) \big],
\end{align*}
where $i \geq 0$, $j\geq 0$, $\ell \geq 0$, $m \geq 1$, and $n = i + j + \ell + m = 1, \ldots, k-1$ are the indexing parameters. In addition, for $r = 1,2,3,4$ we set $(a,b,c,d,e) = (1,2,3,4,5)$, $(a,b,c,d,e) = (2,3,4,5,1)$, $(a,b,c,d,e) = (3,4,5,1,2)$, and $(a,b,c,d,e) = (4,5,1,2,3)$, respectively.

The explicit formula given above for the H(div)-conforming polynomial functions is justified via Lemma~\ref{three_form_lemma}. In this lemma, we focus on the case in which $a = 1$, $b = 2$, $c=3$, and $d = 4$ as all other cases are justified using similar arguments.

\begin{lemma}
    The following quantity belongs to the space $\Vk{k+1}{3}(\T{4})$ of H(div)-conforming functions
    \begin{align}
        \nonumber f_{k}\left(x_1, x_2, x_3, x_4\right) &\big[\lambda_{1} \left(\nabla \lambda_{2} \times \nabla \lambda_{3} \times \nabla \lambda_{4} \right) - \lambda_{2} \left(\nabla \lambda_{3} \times \nabla \lambda_{4} \times \nabla \lambda_{1} \right) \\[1.0ex] 
     &+ \lambda_{3} \left(\nabla \lambda_{4} \times \nabla \lambda_{1} \times \nabla \lambda_{2} \right) - \lambda_{4} \left(\nabla \lambda_{1} \times \nabla \lambda_{2} \times \nabla \lambda_{3} \right) \big], \label{div_function}
    \end{align}
    where $\lambda_{1} = \lambda_{1} (x_1, x_2, x_3, x_4) $, $\lambda_{2} = \lambda_{2} (x_1, x_2, x_3, x_4)$, $\lambda_{3} = \lambda_{3} (x_1, x_2, x_3, x_4)$, and $\lambda_{4} = \lambda_{4} (x_1, x_2, x_3, x_4)$  are barycentric coordinates on the pentatope $\T{4}$, and where $f_{k}(x_1, x_2, x_3, x_4) \in P^{k}(\T{4})$. \label{three_form_lemma}
\end{lemma}

\begin{proof}
    Let us recall that
    \begin{align}
        \Vk{k}{3}(\T{4}) &:= (P^{k-1}(\T{4}))^4 \oplus \tilde{P}^{k-1}(\T{4}) x. \label{three_form_restate}
    \end{align}
    In addition, the following identities hold
    \begin{align*}
        \lambda_{i} = \eta_i + \beta_i \cdot x, \qquad \nabla \lambda_{i} = \beta_i,
    \end{align*}
    where $\eta_i \in \mathbb{R}$, $\beta_i \in \mathbb{R}^{4}$, and $i = 1, 2, 3, 4$. Next, upon expanding the triple products in Eq.~\eqref{div_function} in terms of these identities, one obtains
    \begin{align*}
        &\lambda_{1} \left(\nabla \lambda_{2} \times \nabla \lambda_{3} \times \nabla \lambda_{4} \right) - \lambda_{2} \left(\nabla \lambda_{3} \times \nabla \lambda_{4} \times \nabla \lambda_{1} \right) + \lambda_{3} \left(\nabla \lambda_{4} \times \nabla \lambda_{1} \times \nabla \lambda_{2} \right) - \lambda_{4} \left(\nabla \lambda_{1} \times \nabla \lambda_{2} \times \nabla \lambda_{3} \right) \\[1.0ex]
        &=\left(\eta_1 + \beta_1 \cdot x \right) \left(\beta_2 \times \beta_3 \times \beta_4\right) - \left(\eta_2 + \beta_2 \cdot x \right) \left(\beta_3 \times \beta_4 \times \beta_1\right) \\[1.0ex]
        &+\left(\eta_3 + \beta_3 \cdot x \right) \left(\beta_4 \times \beta_1 \times \beta_2\right) - \left(\eta_4 + \beta_4 \cdot x \right) \left(\beta_1 \times \beta_2 \times \beta_3\right) = A + C(x),
    \end{align*}
    where
    \begin{align*}
        A &:= \eta_1 \left(\beta_2 \times \beta_3 \times \beta_4\right) - \eta_2 \left(\beta_3 \times \beta_4 \times \beta_1\right) +\eta_3 \left(\beta_4 \times \beta_1 \times \beta_2\right) - \eta_4 \left(\beta_1 \times \beta_2 \times \beta_3\right), \\[1.0ex]
        C(x) &:= \left(\beta_1 \cdot x \right) \left(\beta_2 \times \beta_3 \times \beta_4\right) - \left(\beta_2 \cdot x \right) \left(\beta_3 \times \beta_4 \times \beta_1\right) +\left(\beta_3 \cdot x \right) \left(\beta_4 \times \beta_1 \times \beta_2\right) - \left(\beta_4 \cdot x \right) \left(\beta_1 \times \beta_2 \times \beta_3\right). 
    \end{align*}
    After some algebraic manipulations, we find that
    \begin{align*}
        C(x) = \beta_1 \cdot \left(\beta_2 \times \beta_3 \times \beta_4 \right) x.
    \end{align*}
    By inspection, we have that
    \begin{align*}
        A \in \left(P^{0}(\T{4})\right)^{4}, \qquad C(x) \in \left\{Q \in (\tilde{P}^{1}(\T{4}))^{4} \, | \, Q(x) = \phi(x)x \right\}.
    \end{align*}
     Next, we can perform the following decomposition
    \begin{align*}
        f_k &= f_{k-1} + \tilde{f}_{k},
    \end{align*}
    where $f_{k-1} \in P^{k-1}(\T{4})$ and $\tilde{f}_{k} \in \tilde{P}^{k}(\T{4})$. As a result, we have that
    \begin{align*}
        f_k \left[A + C(x) \right] = f_k A + f_{k-1} C(x) + \tilde{f}_{k} C(x).
    \end{align*}
    Naturally, by inspection
    \begin{align*}
        &f_k A + f_{k-1} C(x) \in (P^{k}(\T{4}))^{4}, \qquad \tilde{f}_{k} C(x) \in \left\{Q \in (\tilde{P}^{k+1}(\T{4}))^{4} \, | \, Q(x) = \phi(x)x \right\}.
    \end{align*}
    Based on these identities and the definition in Eq.~\eqref{three_form_restate}, we immediately obtain the desired result
    \begin{align*}
        f_k \left[A + C(x) \right] \in \Vk{k+1}{3}(\T{4}).
    \end{align*}
\end{proof}

Finally, for the sake of completeness, consider the L2-conforming interior functions $v_{ij\ell m}$ of degree $k-1$
\begin{align*}
    v_{ij\ell m}(x_1, x_2, x_3, x_4) = &P_{i}\left(\frac{\lambda_{2}}{\lambda_{1} + \lambda_{2}} \right)P_{j}^{2i+1}\left(\frac{\lambda_{3}}{\lambda_{1}+\lambda_{2}+\lambda_{3}} \right) P_{\ell}^{2(i+j+1)}\left(\frac{\lambda_{4}}{\lambda_{1}+\lambda_{2}+\lambda_{3}+\lambda_{4}} \right)  P_{m}^{2(i+j+\ell)+3} \left(\lambda_{5}\right) \\[1.0ex] 
     &\cdot \left(\lambda_{1} + \lambda_{2}\right)^{i} \left(\lambda_{1} + \lambda_{2} + \lambda_{3} \right)^{j} \left(\lambda_{1} + \lambda_{2} + \lambda_{3} + \lambda_{4} \right)^{\ell},
\end{align*}
where $i \geq 0$, $j\geq 0$, $\ell \geq 0$, $m \geq 0$, and $n = i + j + \ell + m = 0, \ldots, k-1$ are the indexing parameters.

\begin{remark}
In the above discussion, the Legendre and Jacobi polynomials are critical for developing explicit expressions for the bubble spaces. For the sake of brevity, these polynomials will be not defined in this work, but we encourage the curious reader to consult Fuentes et al.~\cite{fuentes2015orientation} for their precise definitions.
\label{functional_remark}
\end{remark}

\subsection{Degrees of Freedom on the Reference Pentatope, $\T{4}$}

We now return our attention to the sequence of spaces in Eqs.~\eqref{rt_pentatope_0}--\eqref{rt_pentatope_4}. Our objective is to construct degrees of freedom for these spaces. There are already well-known sets of degrees of freedom for simplicial elements using wedge products, as in \cite{Arnold13}, etc. Unfortunately, while the wedge product is mathematically elegant, it is frequently difficult for engineers and programmers to interpret and use for implementation purposes.  In order to address this issue, in this section we provide an alternative, more explicit construction of the degrees of freedom.  In addition, these degrees of freedom are shown to be unisolvent.

To specify the degrees of freedom for an $s$-form on the pentatope, we will make use of the preceding discussions; in particular, several degrees of freedom can be specified by using well-known trace degrees of freedom on $\T{3}$ (tetrahedra), $\T{2}$ (triangles), $\T{1}$ (edges), and $\T{0}$ (vertices). Recall from \autoref{Table:4DelementsSubs} that $\T{4}$ has 5 vertices, 10 edges, 10 triangular faces, and  5 tetrahedral facets. Our task is reduced to specifying the remaining interior degrees of freedom, and ensuring unisolvency.

\subsubsection{Dofs for 0-forms on $\T{4}$ }

The polynomial 0-forms on $\T{4}$ are denoted by $\Vk{k}{0}(\T{4}) := P^k(\T{4})$. This space has dimension 
\begin{align*}
    \text{dim}(\Vk{k}{0}(\T{4}))= {k+4 \choose 4} = \frac{1}{24}(k+1)(k+2)(k+3)(k+4).
\end{align*}
The dual space $\Sk{k}{0}(\T{4})$ must have the same dimension.

We can decompose $\Sk{k}{0}(\T{4})$ into trace and volume degrees of freedom. For the trace degrees of freedom, $\Sigma_{trace}^{k,0}(\T{4})$, we use vertex, edge, face, and facet degrees of freedom from Eqs.~\eqref{eq:edge-0}, \eqref{eq:triangle-0}, and \eqref{eq:tet0}.
The total number of trace degrees of freedom is, therefore
\begin{align*}
	\text{dim}\left(\Sigma_{trace}^{k,0}(\T{4}) \right) &= 5 + 10 \,\text{dim}( P^{k-2}(\T{1}))  + 10 \,\text{dim}( P^{k-3}(\T{2})) + 5 \,\text{dim}( P^{k-4}(\T{3}))\\
	&=5+10{k-1 \choose k-2} + 10{k-1 \choose k-3} + 5{k-1 \choose k-4}\\
	&= \frac{5}{6}k (k^2+5).
\end{align*} 
We can also specify volume degrees of freedom on $\T{4}$ for the 0-form proxy $u$ as follows
\begin{equation}\label{eq:pent0} 
	\Sk{k}{0}_{vol}(\T{4}) := \left\{u \rightarrow \int_{\T{4}} uq, \qquad q \in P^{k-5}(\T{4})\right\}.
\end{equation}
It immediately follows that
\begin{align*}
\text{dim}\left(\Sk{k}{0}_{vol}(\T{4}) \right) = {k-1\choose 4} = \frac{1}{24}(k-4)(k-3)(k-2)(k-1).
\end{align*}

\begin{lemma} The degrees of freedom
	\begin{equation}
		\Sk{k}{0}(\T{4}) := \Sigma_{trace}^{k,0}(\T{4}) \cup \Sk{k}{0}_{vol}(\T{4}),
	\end{equation} form a unisolvent set for $\Vk{k}{0}(\T{4}).$
	
\end{lemma}
\begin{proof}
	We begin by noting that
	\begin{align*}
        \text{dim}(\Vk{k}{0}(\T{4})) = \text{dim}(\Sk{k}{0}(\T{4})) = \text{dim}(\Sigma_{trace}^{k,0}(\T{4})) + \text{dim}(\Sk{k}{0}_{vol}(\T{4})).
    \end{align*}
	It will therefore suffice to show that the vanishing of all degrees of freedom for $u$ implies $u=0$.
	
	Suppose that for a particular $u\in \Vk{k}{0}(\T{4})$ that all the degrees of freedom vanish. The vanishing of the trace degrees of freedom means, successively, that $u$ has zero traces on the vertices, edges, faces, and facets of $\T{4}$. It is therefore a bubble function in $\Vkcirc{k}{0}(\T{4})$ and can be expressed as 
    \begin{align*}
	   u= \lambda_1\lambda_2\lambda_3\lambda_4\lambda_5 \psi, \qquad \psi \in P^{k-5}(\T{4}).
    \end{align*}
	
    Now, since all degrees of freedom of the form given by Eq.~\eqref{eq:pent0} also vanish, upon setting $q=\psi$ we see that $\psi \equiv 0$. This establishes that $u\equiv 0$.
\end{proof}

\subsubsection{Dofs for 1-forms on $\T{4}$ }
We recall that
\begin{align*}
    \Vk{k}{1}(\T{4}) &:= ({P}^{k-1}(\T{4}))^4 \oplus \left\{ p\in (\tilde{P}^k(\T{4}))^4 \vert p \cdot x =0\right\}.
\end{align*}
Also, we note that any polynomial in $\tilde{P}^{k+1}(\T{4})$ can be written as $p \cdot x$ for $p \in (\tilde{P}^k(\T{4}))^4.$ Therefore, the polynomial 1-forms on $\T{4}$ have the following dimension
\begin{align*}
\text{dim}(\Vk{k}{1}(\T{4})) &= \text{dim} ((P^{k-1}(\T{4}))^4)+ \text{dim}( (\tilde{P}^k(\T{4}))^4) - \text{dim}(\tilde{P}^{k+1}(\T{4})) \\
&= 4{k+3 \choose 4} + 4 {k+3 \choose 3}- {k+4 \choose 3} = \frac{1}{6}k(k+2)(k+3)(k+4).
\end{align*}
The dual space $\Sk{k}{1}(\T{4})$ must have the same dimension.

We decompose $\Sk{k}{1}(\T{4})$ into the trace and volume degrees of freedom. For the trace degrees of freedom, $\Sigma_{trace}^{k,1}(\T{4})$, we use edge, face, and facet degrees of freedom from Eqs.~\eqref{eq:edge-1}, \eqref{eq:triangle-1}, and \eqref{eq:tet1}.
The total number of trace degrees of freedom is, therefore
\begin{align*}
	\text{dim}( \Sigma_{trace}^{k,1}(\T{4})) &=  10 \, \text{dim}( P^{k-1}(\T{1}))  + 10 \, \text{dim}( (P^{k-2}(\T{2}))^2) + 5   \, \text{dim}( (P^{k-3}(\T{3}) )^3)\\
	&=10{k \choose 1} + 20{k\choose 2} + 15{k \choose 3}\\
	&= \frac{5}{2}k (k^2+k+2).
\end{align*} 
We can also specify volume degrees of freedom on $\T{4}$ for the 1-form proxy $E$ as follows
\begin{equation}\label{eq:pent1} 
	\Sk{k}{1}_{vol}(\T{4}) := \left\{ \int_{\T{4}} E \cdot q, \qquad q \in (P^{k-4}(\T{4}))^4\right\}.
\end{equation}
The corresponding dimension is
\begin{align*}
    \text{dim}\left(\Sk{k}{1}_{vol}(\T{4}) \right) = 4{k\choose 4} = \frac{1}{6}(k-3)(k-2)(k-1)k.
\end{align*}
We see that
\begin{align*}
\text{dim}(\Vk{k}{1}(\T{4})) = \text{dim}(\Sk{k}{1}(\T{4})) = \text{dim}(\Sigma_{trace}^{k,1}(\T{4})) + \text{dim}(\Sk{k}{1}_{vol}(\T{4})),
\end{align*}
from which unisolvency will follow if we can show that the only element of $\Vk{k}{1}(\T{4})$ with vanishing degrees of freedom is the zero element. To establish this, we follow the argument in Lemma 5.36 of~\cite{monkbook}.

During the proof of unisolvency of 1-forms, for ease of exposition, we work on a pentatope $K$ whose vertices are
\begin{align} \label{eq:Kvertices}(0,0,0,0), (1,0,0,0), (0,1,0,0), (0,0,1,0), (0,0,0,1).
\end{align} 
There exists an affine map between the reference pentatope $\T{4}$ and the element $K$, and hence the polynomial spaces $\Vk{k}{s}(\T{4})$ are easily defined on $K$. In what follows, we show that the degrees of freedom, $\Sk{k}{1}(K)$, are unisolvent.

The strategy of the proof is as follows: we first show that if $E\in \Vk{k}{1}(K)$ satisfies $\mathrm{skwGrad}(E)=0$ then $E=\mathrm{grad}(p)$ for some scalar $p\in P^k(K)$. Next, we show that if all the degrees of freedom of $E\in \Vk{k}{1}(K)$ vanish, then $\mathrm{Tr}(\mathrm{skwGrad}(E))=0$ on the facets of $K$.  Furthermore, we show that $\mathrm{skwGrad}(E)=0$ on the entirety of $K$. Based on our first result (above), it immediately follows that $E=\mathrm{grad}(p)$. 
We finally show that the vanishing of volume degrees of freedom for $E$ implies that~$p=0$. 

\begin{lemma}
    If $E \in \Vk{k}{1}(K)$ satisfies $ \mathrm{skwGrad} (E) =0$ then $E \equiv \mathrm{grad}(p)$ for some $p \in P^k(K)$.
    \label{p_grad_lemma}
\end{lemma}
\begin{proof} 
    This proof follows closely the analogous proof for the tetrahedron in   Lemma 5.28 of \cite{monkbook}. We first observe that if $E \in \Vk{k}{1}(K), E \in (P^k(K))^4$. Moreover, $ \mathrm{skwGrad} (E) =0 \Rightarrow E = \mathrm{grad}(p)$ for some $p\in P^{k+1}(K)$. Now, we need to show that $p\in P^{k}(K)$. 

    We can decompose $p$ such that $p= p_1+p_2$, where $p_1\in P^k(K)$ and $p_2\in \tilde{P}^{k+1}(K).$ However, the form of $\Vk{k}{1}(K)$ in Eq.~\eqref{rt_pentatope_1} forces $\mathrm{grad}(p_2) \cdot x =0$. Since $p_2$ is homogeneous, $x\cdot \mathrm{grad}(p_2) = (k+1) p_2 =0$, and therefore $E= \mathrm{grad}(p)$ for some $p \in P^k(K)$. 
\end{proof}
The implication of the previous lemma is that while a generic $w \in \Vk{k}{1}(K)$ could contain homogeneous polynomials of degree $k$, if it satisfies $ \mathrm{skwGrad}(w) =0$, then $w$ must be the gradient of a degree-$k$ form. Hence $w\in (P^{k-1}(K))^4$. 

We next show that the vanishing of all dofs for a polynomial 1-form $E$ on $K$ implies that not only the trace of $E$ but also  $\mathrm{Tr}(\mathrm{skwGrad}(E))$ vanishes on the facets.
\begin{lemma}
    Let $E\in \Vk{k}{1}(K)$ be a polynomial 1-form for which all the degrees of freedom $\Sk{k}{1}(K)$ vanish. Then $\mathrm{Tr}(\mathrm{skwGrad}(E))\equiv 0$ on the facets of $K$.
\end{lemma}

\begin{proof}
    Since all the dofs for $E$ vanish, then in particular those associated with the traces vanish. 
    Moreover, the trace of $E$ on to any facet $\mathcal{F}$ is a $1$-form on this tetrahedron, and $\mathrm{Tr}(E)$ vanishes on $\mathcal{F}$. We now integrate by parts on $\mathcal{F}$ to see that
    \begin{align*}
    \int_\mathcal{F} q \cdot \mathrm{Tr}(\mathrm{skwGrad}(E)) \, dx = \int_\mathcal{F} q\cdot \nabla \times(\mathrm{Tr}(E)) \, dx = \int_\mathcal{F} \left(\nabla \times q \right) \cdot \mathrm{Tr}(E) \, dx =0.
    \end{align*}
    This equation holds for any sufficiently smooth $q$, and in particular for $q  \in (P^{k-1}(\mathcal{F}))^3$. Choosing $q = \mathrm{Tr}(\mathrm{skwGrad}(E))$ on $\mathcal{F}$ shows that $\mathrm{Tr}(\mathrm{skwGrad}(E))=0$ on $\mathcal{F}$.
\end{proof}

The next theorem uses the previous lemmas to establish unisolvency.
\begin{theorem}
	Let $E\in \Vk{k}{1}(K)$ be a polynomial 1-form for which all the degrees of freedom $\Sk{k}{1}(K)$ vanish. Then $E \equiv 0$.
\end{theorem}
\begin{proof}

In accordance with Eq.~\eqref{ibp_one_C}
\begin{align}
 \nonumber \left(\text{tr}^{(1)} E \right)(F) &= \frac{1}{2} \int_{\partial K}  \left[  E \otimes n - n\otimes E\right] : F \, ds \\[1.0ex]
    &= \int_{K} \left( \text{Div} \, F \right) \cdot E \, dx - \int_{K} F : \left(\text{skwGrad} \, E \right) \, dx,  \label{trace_one_exp}
\end{align}
for $F \in H(\mathrm{Div},K,\mathbb{K})$. We note that if $F \in \VtoM{(P^{k-3}(K))^6}$ then it is automatically in $H(\mathrm{Div},K,\mathbb{K})$. 
Since $\mathrm{tr}(E)=0$, we have
\begin{align*}
   \int_{K} \left( \text{Div} \, F \right) \cdot E \, dx = \int_{K} F : \left(\text{skwGrad} \, E \right) \, dx,
\end{align*} 
for each $F \in \VtoM{(P^{k-3}(K))^6}$. Since the volumetric degrees of freedom vanish, we can set $q = \mathrm{Div} \,F$ in Eq.~\eqref{eq:pent1}, and obtain the following 
\begin{equation}  \label{eq:skwgraduf}\int_{K} F : \left(\text{skwGrad} \, E \right) \, dx=0 \qquad \forall F \in \VtoM{(P^{k-3}(K))^6}.
\end{equation}
Now, let $\mathcal{F}$ be a tetrahedral facet of the element $K$. Using the previous lemma establishes that $\mathrm{skwGrad}(E)$ has vanishing traces on $\mathcal{F}$. Let us denote $B:=\mathrm{skwGrad}(E) = \VtoM{\begin{bmatrix}
	B_{12}, B_{13}, B_{14},  B_{23}, B_{24}, B_{34}
	\end{bmatrix}^T}.$  Consider the trace on to the facet on the hyperplane $x_4=0$, (see Eq.~\eqref{trace_formula_two}). Since
 \begin{align*}
 \mathrm{tr}(B) =2\begin{bmatrix}
	B_{23}(x_1,x_2,x_3,0)\\[1.0ex] -B_{13}(x_1,x_2,x_3,0)\\[1.0ex] B_{12}(x_1,x_2,x_3,0)\\[1.0ex]  0
	\end{bmatrix} =0,
\end{align*}
 it follows that $B_{23}(x_1, x_2, x_3,0)= B_{13}(x_1, x_2, x_3,0)= B_{12}(x_1, x_2, x_3,0)=0.$
Similarly, the trace of $B$ on to the plane $x_3=0$ vanishes, from which we see $B_{12}(x_1,x_2,0,x_4)=
B_{14}(x_1,x_2,0,x_4) = B_{24}(x_1,x_2,0,x_4)=0$. Consequently, $B_{12}(x_1,x_2,x_3,x_4)=x_3 x_4 r_{12}$ for some $r_{12}\in P^{k-3}(K)$. Similar considerations on all the other facets imply that 
\begin{align*}
 \mathrm{skwGrad}(E)=B = \VtoM{\begin{bmatrix}
	x_3 x_4 r_{12}\\[1.0ex]
	x_2 x_4 r_{13}\\[1.0ex]
	x_2 x_3 r_{14}\\[1.0ex]
	x_1 x_4 r_{23}\\[1.0ex]
	x_1 x_3 r_{24}\\[1.0ex]
	x_1 x_2 r_{34}
	\end{bmatrix}}, \qquad r_{ij}\in P^{k-3}(K), \qquad \text{for} \quad i = 1, 2, 3,4, \quad j = 1, 2, 3, 4.
 \end{align*}
 But then choosing $ F= \VtoM{\begin{bmatrix}
r_{12},
r_{13},
r_{14},
r_{23},
r_{24},
r_{34}
\end{bmatrix}^T}$ in Eq.~\eqref{eq:skwgraduf}, we get that
\begin{align*}
&0 = \int_{K} F : \left(\text{skwGrad} \, E \right) \, dx \\[1.0ex]
&= \int_K \left(x_3 x_4 r_{12}^2+x_2 x_4 r_{13}^2+ x_2 x_3 r_{14}^2+ x_1 x_4 r_{23}^2+x_1 x_3 r_{24}^2+ x_1 x_2 r_{34}^2 \right) dx,
\end{align*}
from which it follows that $r_{ij}=0$, (as the products of the form $x_3 x_4$, $x_2 x_4$, etc. are strictly non-negative on $K$). From this we conclude that $B=\mathrm{skwGrad}(E)=0$ in $K$, and consequently from Lemma~\ref{p_grad_lemma}, $E=\mathrm{grad}(p)$ for some $p\in P^{k}(K)$. Since the traces of $E$ vanish, we can choose $p=0$ on the facets, faces, edges, and vertices of $K$, which allows us to write
\begin{align*}
p = x_1 x_2 x_3 x_4 \hat{r}, \quad \hat{r}\in P^{k-4}(K).
\end{align*}
But since  the volumetric degrees of freedom of $E$ vanish, we can pick 
\begin{align*}
q = \begin{bmatrix}
x_1 \partial_{1} (\hat{r}) + \hat{r}\\[1.0ex]
x_2 \partial_{2} (\hat{r}) + \hat{r}\\[1.0ex]
x_3 \partial_{3} (\hat{r}) + \hat{r}\\[1.0ex]
x_4 \partial_{4} (\hat{r}) + \hat{r}
\end{bmatrix} = \begin{bmatrix}
x_1 \hat{r}_{x_1} + \hat{r}\\[1.0ex]
x_2 \hat{r}_{x_2} + \hat{r}\\[1.0ex]
x_3 \hat{r}_{x_3} + \hat{r}\\[1.0ex]
x_4 \hat{r}_{x_4} + \hat{r}
\end{bmatrix},
\end{align*}
in Eq.~\eqref{eq:pent1}, in order to obtain
\begin{align*}
&0=\int_K E \cdot q \, dx = \int_K\mathrm{grad}(p)\cdot q \, dx \\[1.0ex]
&= \int_K \left( x_2 x_3 x_4(x_1\hat{r}_{x_1} + \hat{r})^2+  x_1 x_3 x_4(x_2 \hat{r}_{x_2} + \hat{r})^2+
x_1 x_2 x_4(x_3 \hat{r}_{x_3} + \hat{r})^2+
x_1 x_2 x_3(x_4 \hat{r}_{x_4} + \hat{r})^2 \right) dx.
\end{align*}
All the coordinate functions of $x_1, x_2, x_3, x_4$ are non-negative in $K$. Therefore, the integral above only vanishes if  $(x_1\hat{r}_{x_1}+\hat{r})=0$, $(x_2\hat{r}_{x_2}+\hat{r})=0$, etc.. This in turn is impossible unless $\hat{r}=0$. But then $p$, and consequently $ E= \mathrm{grad}(p)$ vanishes.
\end{proof}

\subsubsection{Dofs for 2-forms on $\T{4}$ }
The polynomial 2-forms on $\T{4}$ are associated with skew-symmetric matrices
\begin{align*}
	\Vk{k}{2}(\T{4}) &= \VtoM{(P^{k-1}(\T{4}))^6} \oplus \span\{ \tilde{P}^{k-1}(\T{4})B_1+\tilde{P}^{k-1}(\T{4})B_2+\tilde{P}^{k-1}(\T{4})B_3+\tilde{P}^{k-1}(\T{4})B_4\}\\
	&=\VtoM{(P^{k-1}(\T{4}))^6} \oplus  \left\{ B \in \mathcal{L}\left((\tilde{P}^k(\T{4}))^6 \right) \vert B x= 0 \right\}.
\end{align*} 
The dimension of the second space above is the same as that of $(\tilde{P}^{k-1}(\T{4}))^3 + \tilde{P}^{k-1}(\T{3})$. Please consult~\ref{pent_exp} for proof of this fact. Altogether, the dimension of the entire space is
\begin{align*}
\text{dim}\left(\Vk{k}{2}(\T{4})\right) &= 6 {k+3 \choose k-1} + 3 { k+2 \choose 3} + {k+1 \choose 2}\\
&=\frac{1}{4}k\left(k^3 + 8k^2 + 19k + 12\right).
\end{align*}

Face and facet traces are well-defined for polynomial 2-forms on $\T{4}$. Therefore, we specify the degrees of freedom corresponding to $\Vk{k}{2}(\T{4})$ as 
\begin{align*}
    \Sk{k}{2}(\T{4}) = \Sk{k}{2}_{vol}(\T{4}) \cup \Sigma_{trace}^{k,2}(\T{4}),
\end{align*}
where $\Sigma_{trace}^{k,2}(\T{4})$ are the trace degrees of freedom corresponding to the 10 triangular faces and 5 tetrahedral facets, as given by Eqs.~\eqref{eq:triangle-2} and \eqref{eq:tet2}. The dimension of this space is
\begin{align*}
\text{dim}\left(\Sigma_{trace}^{k,2}(\T{4}) \right) &= 10\,\text{dim}( P^{k-1}(\T{2}) )  + 5 \, \text{dim} \left((P^{k-2}(\T{3}))^{3}\right) \\
&= 10{k+1\choose 2} + 15 {k+1 \choose 3} \\
&= \frac{5}{2}k(k^2+2k+1).
\end{align*}
We can also specify volume degrees of freedom on $\T{4}$ for a 2-form proxy $F$ as
\begin{equation}\label{eq:pent2} 
	\Sk{k}{2}_{vol}(\T{4}) := \left\{ \int_{\T{4}} F : q, \qquad q \in 
	\mathcal{L}\left((P^{k-3}(\T{4}))^6 \right)\right\}.
\end{equation}
The dimension of this space is
\begin{align*}
    \text{dim} \left(\Sk{k}{2}_{vol}(\T{4}) \right) = 6{k+1\choose 4} = \frac{1}{4}(k-2)(k-1)k(k+1).
\end{align*}
It can easily be confirmed that
\begin{align*}
\text{dim} \left(\Vk{k}{2}(\T{4}) \right) = \text{dim}\left(\Sk{k}{2}(\T{4}) \right) = \text{dim} \left(\Sigma_{trace}^{k,2}(\T{4})\right) + \text{dim} \left(\Sk{k}{2}_{vol}(\T{4}) \right).
\end{align*}
Once again, unisolvency of the finite element will follow if we can establish that the vanishing of all dofs for an arbitrary $u \in \Vk{k}{2}(\T{4})$ implies that $u \equiv 0.$

Following our analysis of the 1-forms, it is more convenient to work on the mapped element $K$ whose vertices are given in Eq.~\eqref{eq:Kvertices}. In addition, we follow a similar strategy as before in order to establish unisolvency. First, we show that if $F\in \Vk{k}{2}(K)$ has vanishing curl, it must be that $F = \mathrm{skwGrad}(E)$ for some $E\in (P^k(K))^4$. Next, we show that if the trace degrees of freedom of $F$ vanish, then $\mathrm{curl}(F) = 0$ on the facets of $K$. This helps us establish that a 2-form $F \in \Vk{k}{2}(K)$ with vanishing degrees of freedom has vanishing curl on the entirety of $K$, and furthermore that $F$ itself vanishes.

\begin{lemma} \label{lemma:2form3form} If $F\in \Vk{k}{2}(K)$ has vanishing curl, then $F \equiv \mathrm{skwGrad}(E)$ for some $E \in (P^k(K))^4$.
\end{lemma}
\begin{proof}
Since $\mathrm{curl}(F)=0$, $F=\mathrm{skwGrad}(E)$ for some sufficiently smooth 1-form $E$. In addition, since $F\in \Vk{k}{2}(K)$, in accordance with Eq.~\eqref{rt_pentatope_1} we deduce that 
\begin{align*}
E = A+C, \quad A \in (P^{k}(K))^{4}, \quad  C\in (\widetilde{P}^{k+1}(K))^4.
\end{align*}
It remains for us to show that $C=[c_1,c_2,c_3,c_4]^T=0$ where $c_i \in \widetilde{P}^{k+1}(K).$ 

From Remark~\eqref{remark:pent}, we easily verify that $ \mathrm{skwGrad}(C) \in \left\{B \in \VtoM{(\tilde{P}^{k}(K))^6} \vert B x =0\right\}$ and hence
\begin{align*}
\mathrm{skwGrad}(C)x = \begin{bmatrix}
x \cdot \partial_{1}(C) - x \cdot\mathrm{grad} (c_1)\\[1.0ex]
x \cdot \partial_{2}(C) - x \cdot\mathrm{grad}(c_2) \\[1.0ex]
x \cdot \partial_{3}(C) - x\cdot \mathrm{grad} (c_3)\\[1.0ex]
x \cdot \partial_{4}(C) - x\cdot \mathrm{grad} (c_4)
\end{bmatrix} = 0.
\end{align*}
But since $c_i$ is a homogeneous polynomial, $x \cdot\mathrm{grad}(c_i) =
(k+1)c_i$, and therefore 
\begin{align*}
x \cdot \partial_i (C) = (k+1)c_i, \qquad i=1,2,3,4.
\end{align*}
This is only possible if $c_i=0$ for each $i$. As a result, it immediately follows that $C=0$. 
\end{proof}

\begin{lemma}
 Let $F\in \Vk{k}{2}(K)$ be a polynomial 2-form for which all the degrees of freedom $\Sk{k}{2}(K)$ vanish. Then $\mathrm{Tr}(\mathrm{curl}(F))\equiv 0$ on the facets of $K$.
\end{lemma}
\begin{proof}
	Let $\mathcal{F}$ be a tetrahedral facet of $K$. Since $\text{curl} (F)$ is a 3-form, its trace on $\mathcal{F}$ is a 3-form. The divergence theorem on $\mathcal{F}$, and the vanishing of traces of $F$ gives
    \begin{align}
	\int_\mathcal{F}  \mathrm{Tr}(\mathrm{curl}(F)) q \, dx = \int_{\mathcal{F}} \nabla\cdot(\mathrm{Tr}(F)) q \, dx = - \int_{\mathcal{F}} (\mathrm{Tr}(F))       \cdot \nabla q \, dx =0, \label{two_form_facet}
    \end{align}
    for any sufficiently smooth $q$, and in particular for $q \in P^{k-1}(K)$.
Therefore, upon setting $q = \nabla\cdot(\mathrm{Tr}(F))$ in Eq.~\eqref{two_form_facet}, we find that $\nabla\cdot(\mathrm{Tr}(F)) = \mathrm{Tr}(\text{curl}(F))=0$ on each $\mathcal{F}$, and on the entire boundary of $K$. 
\end{proof}

\begin{theorem}
	Let $F \in \Vk{k}{2}(K)$ be a polynomial 2-form for which all the degrees of freedom $\Sk{k}{2}(K)$ vanish. Then $F \equiv 0.$
\end{theorem}

\begin{proof} We first observe from Eq.~\eqref{ibp_two_A} that
\begin{align*}
    \nonumber \left(\text{tr}^{(2)} F \right)(E) &= \int_{\partial K} \left( n \times F \right) \cdot E \, ds \\[1.0ex]
     &= \int_{K} \left(\text{curl} \, F \right) \cdot E \, dx -\int_{K} \left(\text{Curl} \, E \right) : F \, dx, 
\end{align*}
where $E \in H\left( \text{Curl}, K, \mathbb{R}^{4} \right)$. Since all the trace degrees of freedom of $F\in \Vk{k}{2}(K)$ vanish, $F$ has zero trace, and therefore
\begin{align*}
\int_{K} \left(\text{curl} \, F \right) \cdot E \, dx = \int_{K} \left(\text{Curl} \, E \right) : F \, dx.
\end{align*}
Now, if we pick $E \in (P^{k-2}(K))^4$ then $E \in  H\left( \text{Curl}, K, \mathbb{R}^{4} \right)$, and $\mathrm{Curl}(E) \in \VtoM{(P^{k-3}(K))^6}.$ Since the volumetric dofs vanish for $F$, we set $q = \mathrm{Curl}(E)$ in Eq.~\eqref{eq:pent2}, and we obtain
\begin{equation}\label{eq:intermediate2-pent}
\int_{K} \left(\text{curl} \, F \right) \cdot E \, dx = 0, \qquad \forall E \in (P^{k-2}(K))^4.
\end{equation}
From the previous lemma, the trace of $\mathrm{curl}(F)$ vanishes on the facets and so it is a 3-form bubble in $(P^{k-1}(K))^4$, and we can write
\begin{align*}
 \mathrm{curl}(F) = \begin{bmatrix}
 x_1 \psi_1\\[1.0ex]
 x_2\psi_2\\[1.0ex]
 x_3\psi_3\\[1.0ex]
 x_4\psi_4
 \end{bmatrix}, \qquad \psi_i \in P^{k-2}(K), \qquad i = 1,2 ,3, 4.
 \end{align*}
Upon choosing $E=[\psi_1,\psi_2,\psi_3,\psi_4]^T$ in Eq.~\eqref{eq:intermediate2-pent}, we obtain
\begin{align*}
0=\int_{K} \left(\text{curl} \, F \right) \cdot E \, dx = \int_K \sum_{i=1}^{4} x_i \psi_i^2 \, dx.
\end{align*}
But $x_i\geq0$ in $K$ and so we are guaranteed that $\psi_i=0$ for  $i=1,2,3,4$. This shows that $\mathrm{curl}(F)=0$ in $K$. 

Next, in accordance with Lemma~\ref{lemma:2form3form}, we can immediately deduce that $F = \mathrm{skwGrad}(\mathcal{E})$ for some $\mathcal{E} \in (P^k(K))^4$. In turn, the vanishing of the traces of $F$ allows us to pick $\mathcal{E}$ to also have vanishing traces. Recalling Eq.~\eqref{trace_formula_one}, we can obtain the following trace formula on the facet $x_4=0$ 
\begin{align*}
\mathrm{tr}(\mathcal{E}) &=\frac{1}{2} \begin{bmatrix}
            0 & 0 & 0 & \mathcal{E}_1(x_1,x_2,x_3,0) \\[1.0ex]
            0 & 0 & 0 & \mathcal{E}_2(x_1,x_2,x_3,0) \\[1.0ex]
            0 & 0 & 0 & \mathcal{E}_3(x_1,x_2,x_3,0) \\[1.0ex]
            -\mathcal{E}_1(x_1,x_2,x_3,0) & -\mathcal{E}_2(x_1,x_2,x_3,0) & -\mathcal{E}_3(x_1,x_2,x_3,0) & 0
        \end{bmatrix}=0, \\[1.0ex] &\Rightarrow \mathcal{E}_1(x_1,x_2,x_3,0)=\mathcal{E}_2(x_1,x_2,x_3,0)=\mathcal{E}_3 (x_1,x_2,x_3,0)=0. 
\end{align*}
Similar considerations apply for the other facets, allowing us to obtain the following expression for $\mathcal{E}$
\begin{align*}
\mathcal{E}=\begin{bmatrix}x_2x_3x_4g_1\\[1.0ex]
x_1x_3x_4g_2\\[1.0ex]
x_1x_2x_4g_3\\[1.0ex]
x_1x_2x_3g_4\end{bmatrix}, \quad g_i \in P^{k-3}(K), \qquad i = 1,2,3, 4.
\end{align*}
Furthermore
\begin{align*}
    \mathrm{skwGrad}(\mathcal{E}) = \VtoM{ 
        \begin{bmatrix}
           x_3 x_4 \left(\partial_{1}(x_1 g_2) - \partial_{2}(x_2 g_1)\right) \\[1.0ex]
            x_2 x_4\left(\partial_{1}(x_1 g_3) - \partial_{3}(x_3 g_1)\right) \\[1.0ex]
            x_2 x_3 \left(\partial_{1}(x_1 g_4) - \partial_{4}(x_4 g_1)\right) \\[1.0ex]
            x_1 x_4\left(\partial_{2}(x_2 g_3) - \partial_{3}(x_3 g_2)\right) \\[1.0ex]
            x_1 x_3\left(\partial_{2}(x_2 g_4) - \partial_{4}(x_4 g_2)\right) \\[1.0ex]
            x_1 x_2\left(\partial_{3}(x_3 g_4) - \partial_{4}(x_4 g_3)\right)
        \end{bmatrix}
    }.
\end{align*}
We can now pick $q$ in Eq.~\eqref{eq:pent2} as follows
\begin{align*}
q= \VtoM{\begin{bmatrix}
	q_{12}\\[1.0ex]
	q_{13}\\[1.0ex]
	q_{14}\\[1.0ex]
	q_{23}\\[1.0ex]
	q_{24}\\[1.0ex]
	q_{34}
	\end{bmatrix}}=\VtoM{ 
        \begin{bmatrix}
            \partial_{1}(x_1 g_2) - \partial_{2}(x_2 g_1) \\[1.0ex]
            \partial_{1}(x_1 g_3) - \partial_{3}(x_3 g_1) \\[1.0ex]
            \partial_{1}(x_1 g_4) - \partial_{4}(x_4 g_1) \\[1.0ex]
            \partial_{2}(x_2 g_3) - \partial_{3}(x_3 g_2) \\[1.0ex]
            \partial_{2}(x_2 g_4) - \partial_{4}(x_4 g_2) \\[1.0ex]
            \partial_{3}(x_3 g_4) - \partial_{4}(x_4 g_3)
        \end{bmatrix}
    },
\end{align*}
in order to obtain
\begin{align*}
    &0=\int_{K}F:q \, dx = \int_K\mathrm{skwGrad}(\mathcal{E}):q \, dx \\[1.0ex]
    &= \int_K \left(x_3x_4q_{12}^2+x_2x_4q_{13}^2+x_2x_3q_{14}^2+x_1x_4q_{23}^2+x_1x_3q_{24}^2+x_1x_2q_{34}^2 \right) dx.
\end{align*}
But then each $q_{ij}=0$, and in turn, it is easy to check that each $g_i$ must vanish. Finally, it follows that $\mathcal{E} =0$ and $F = \mathrm{skwGrad}(\mathcal{E})$ must vanish.
\end{proof}

\subsubsection{Dofs for 3-forms on $\T{4}$ }
The polynomial 3-forms on $\T{4}$ are associated with 4-vectors. The corresponding degrees of freedom have the following dimension
\begin{align*}
\text{dim}\left(\Vk{k}{3}(\T{4}) \right)& = \text{dim}( (P^{k-1}(\T{4}))^4) + \text{dim} ( \tilde{P}^{k-1}(\T{4})) \\
&= 4 {k+3 \choose 4} + {k+2 \choose 3 } = \frac{1}{6}k(k+1)(k+2)(k+4).
\end{align*}
Only facet traces are defined for 3-forms, as given by Eq.~\eqref{eq:tet3}. Therefore
\begin{align*}
\text{dim}\left(\Sigma_{trace}^{k,3}(\T{4})\right) &= 5 \, \text{dim} \left(P^{k-1}(\T{3}) \right) \\
&= 5 {k+2 \choose 3} = \frac{5}{6}k(k+1)(k+2).
\end{align*}
We define the volume degrees of freedom for the 3-form proxy $G$ as follows
\begin{equation} \label{eq:pent3}
	\Sk{k}{3}_{vol}(\T{4}) := \left\{ \int_{\T{4}} G\cdot q, \qquad q \in (P^{k-2}(\T{4}))^4\right\}.
\end{equation}
The dimension of this space is
\begin{align*}
    \text{dim}\left( \Sk{k}{3}_{vol}(\T{4}) \right) = 4{k+2 \choose 4} = \frac{1}{6}(k-1)k(k+1)(k+2).
\end{align*}
Then, we define 
\begin{align*}
\Sk{k}{3}(\T{4}):= \Sigma_{trace}^{k,3}(\T{4}) \cup \Sk{k}{3}_{vol}(\T{4}),
\end{align*}
and note that
\begin{align*}
\text{dim} \left(\Vk{k}{3}(\T{4}) \right) = \text{dim} \left(\Sk{k}{3}(\T{4}) \right) = \text{dim}\left(\Sigma_{trace}^{k,3}(\T{4}) \right) + \text{dim}\left(\Sk{k}{3}_{vol}(\T{4})\right).
\end{align*}
Therefore, unisolvency will be guaranteed by establishing the following result.

\begin{lemma}
	   Consider $G \in \Vk{k}{3}(K)$ a polynomial 3-form for which all the degrees of freedom $\Sk{k}{3}(K)$ vanish. Then $G \equiv 0.$
\end{lemma}
\begin{proof}
    This proof closely follows the strategy outlined in~\cite{monkbook}, with reference element $K$ given by Eq.~\eqref{eq:Kvertices}.

    We begin by introducing $v \in P^{k-1}(K)$. In accordance with integration by parts (Eq.~\eqref{ibp_three}), the vanishing of trace and volume degrees of freedom for $G$, and Eq.~\eqref{eq:pent3}, one obtains
    \begin{align*}
	   \int_{K} (\text{di}v \, G ) v = -\int_{K} G \cdot (\text{grad} \, v) =0.
    \end{align*}
    By choosing $v = \text{div}(G)$, we see that $\text{div} (G) =0$ in $K$.
	
    Now, $G = p + \hat{r}x$ for $p \in (P^{k-1}(K))^4$ and $\hat{r}\in \tilde{P}^{k-1}(K)$ by definition (Eq.~\eqref{rt_pentatope_3}). In addition, it is easy to check that $\text{div}(\hat{r}x) = (k+4)\hat{r}$. Therefore, 
    \begin{align*}
        \text{div}(G) = \text{div} (p) + (k+4)\hat{r} \Rightarrow \hat{r} = -\frac{1}{k+4} \text{div}(p) \in P^{k-2}(K).
    \end{align*}
    But this is not possible unless the degree $k-1$ polynomial $\hat{r}=0$. With this in mind, we observe the following
    \begin{align*}
        G=p \in (P^{k-1}(K))^4  \Rightarrow G= \begin{bmatrix} x_1 \phi_1 \\[1.0ex] x_2 \phi_2 \\[1.0ex] x_3\phi_3 \\[1.0ex] x_4 \phi_4 \end{bmatrix},\quad \phi_i \in P^{k-2}(K).
    \end{align*}
    This reformulation is possible because $G$ is a 3-form bubble with vanishing traces given by Eq.~\eqref{trace_formula_three}. If $k>1$, we can pick $q=[\phi_1, \phi_2, \phi_3, \phi_4]^T$ in Eq.~\eqref{eq:pent3}, from which it will follow that $\phi_i=0$, (and hence $G\equiv 0)$. If $k=1$, then trivially $\phi_i=0.$
\end{proof}

\subsubsection{Dofs for 4-forms on $\T{4}$}
The polynomial 4-forms on $\T{4}$ are associated with scalars. Traces for 4-forms are not well-defined. Instead, we specify interior degrees of freedom for the 4-form proxy $q$ as
\begin{equation}\label{eq:pent4} 
	\Sk{k}{4}_{vol}(\T{4}):=\left\{ \int_{\T{4}} qp, \qquad p\in P^{k-1}(\T{4})\right\}. 
\end{equation}
In a natural fashion, we have that
\begin{align*}
    \text{dim}\left(\Vk{k}{4}(\T{4}) \right) &= \text{dim} \left(\Sk{k}{4}(\T{4}) \right) =\text{dim} \left(\Sk{k}{4}_{vol}(\T{4}) \right) \\
    &= {k+3 \choose 4} = \frac{1}{24}k(k+1)(k+2)(k+3).
\end{align*}
It then remains to prove unisolvency.
\begin{lemma}
    Consider $q \in \Vk{k}{4}(\T{4})$ a polynomial 4-form for which all the degrees of freedom $\Sk{k}{4}(\T{4})$ vanish. Then $q \equiv 0.$
\end{lemma}

\begin{proof}
Suppose that $q \in \Vk{k}{4}(\T{4})$ has vanishing degrees of freedom as given by Eq.~\eqref{eq:pent4}; then setting $p=q$ shows that $q \equiv 0$.
\end{proof}

%% file: prism.tex
\section{Finite Elements on a Reference Tetrahedral Prism}

In this section, we introduce the finite element approximation spaces for $s$-forms on the tetrahedral prism $\W{4}$. These finite element spaces are developed by taking tensor products of spaces on tetrahedra $\T{3}$ with spaces on line segments $\T{1}$. In accordance with the work of~\cite{arnold2015finite} and \cite{mcrae2016automated}, we can construct tensor product elements using the spaces from two different sequences
\begin{align*}
	&U_{0} \arrow{r}{d^{\left(0\right)}} U_{1} \arrow{r}{d^{\left(1\right)}} \cdots \arrow{r}{d^{\left(n-1\right)}} U_{n}, \\[1.0ex]
	&W_{0} \arrow{r}{d^{\left(0\right)}} W_{1} \arrow{r}{d^{\left(1\right)}} \cdots \arrow{r}{d^{\left(m-1\right)}} W_{m},
\end{align*}
which are defined on domains $\Omega \in \mathbb{R}^{n}$ and $\underline{\Omega} \in \mathbb{R}^{m}$, respectively.
The associated tensor product sequence can be written as follows
\begin{align*}
	\left(U \times W\right)_{0} \arrow{r}{d^{\left(0\right)}}  \left(U \times W\right)_{1} \arrow{r}{d^{\left(1\right)}} \cdots \arrow{r}{d^{\left(n+m-1\right)}}  \left(U \times W\right)_{n+m},
\end{align*}
which is defined on the domain $\underline{\underline{\Omega}} \in \mathbb{R}^{n+m}$. Each entry in the tensor product sequence above can be written as 
\begin{align*}
	\left(U\times W\right)_k = \bigoplus_{i+j=k} \left(U_i \times W_j\right),
\end{align*}
where $k = 0, \ldots, n+m$. 

On the tetrahedral prism, we have that $n=3$ and $m=1$. As a result, we recover the following sequence
\begin{align*}
	\left(U \times W\right)_{0} \arrow{r}{d^{\left(0\right)}}  \left(U \times W\right)_{1} \arrow{r}{d^{\left(1\right)}} \left(U \times W\right)_{2} \arrow{r}{d^{\left(2 \right)}}  \left(U \times W\right)_{3} \arrow{r}{d^{\left(3 \right)}} \left(U \times W\right)_{4},
\end{align*}
where
\begin{subequations}
\begin{align}
	\left(U \times W\right)_{0} &= U_0 \times W_0, \\[1.0ex]
	\left(U \times W\right)_{1} &= \left(U_1 \times W_0 \right) \oplus \left(U_0 \times W_1 \right), \\[1.0ex]
	\left(U \times W\right)_{2} &= \left(U_1 \times W_1\right) \oplus \left(U_2 \times W_0\right), \\[1.0ex]
	\left(U \times W\right)_{3} &= \left(U_3 \times W_0 \right) \oplus \left(U_2 \times W_1\right), \\[1.0ex]
	\left(U \times W\right)_{4} &= U_3 \times W_1. 
\end{align}\label{cross_def}
\end{subequations}
The precise construction of the resulting tensor-product spaces is given in~\ref{tet_exp}. This construction is expressed in terms of differential forms. In what follows, we provide the equivalent spaces in terms of vector and matrix notation
\begin{align*}
	\Vk{k}{0}(\W{4}) := & P^{k}\left(\T{1}\right) \times P^{k}\left(\T{3}\right), \\[1.0ex]
	\Vk{k}{1}(\W{4}) := &P^{k} \left(\T{1}\right) \times \Bigg( \Big\{  p \; \Big| \; p \in \left[\tilde{P}^{k}\left(\T{3}\right), \tilde{P}^{k}\left(\T{3}\right), \tilde{P}^{k}\left(\T{3}\right), 0 \right]^{T}, p\cdot x = 0 \Big\} \\[1.0ex]
	&\qquad \qquad \oplus \left[P^{k-1}\left(\T{3}\right), P^{k-1}\left(\T{3}\right), P^{k-1}\left(\T{3}\right), 0 \right]^{T} \Bigg) \oplus P^{k}\left(\T{3}\right) \times \left[0, 0, 0, P^{k-1} \left(\T{1}\right) \right]^{T},
\end{align*}
\begin{align*}
	\Vk{k}{2}(\W{4}) := & P^{k-1} \left(\T{1}\right) \times \Bigg( \Bigg\{ \begin{bmatrix}
		0 & 0 & 0 & p_1 \\
		0 & 0 & 0 & p_2 \\
		0 & 0 & 0 & p_3 \\
		\ast & \ast & \ast & 0
	\end{bmatrix}  \Big| \; p \in \left[\tilde{P}^{k}\left(\T{3}\right), \tilde{P}^{k}\left(\T{3}\right), \tilde{P}^{k}\left(\T{3}\right), 0 \right]^{T}, p\cdot x = 0 \Bigg\} \\[1.0ex]
	&\qquad \qquad \quad \oplus \begin{bmatrix}
		0 & 0 & 0 & P^{k-1} \left(\T{3}\right) \\
		0 & 0 & 0 & P^{k-1} \left(\T{3}\right) \\
		0 & 0 & 0 & P^{k-1} \left(\T{3}\right) \\
		\ast & \ast & \ast & 0
	\end{bmatrix} \Bigg) \\[1.0ex]
	&\oplus P^{k} \left(\T{1}\right) \times \Bigg( \begin{bmatrix}
		0 & P^{k-1} \left(\T{3}\right) & P^{k-1} \left(\T{3}\right) & 0 \\
		\ast & 0 & P^{k-1} \left(\T{3}\right) & 0 \\
		\ast & \ast & 0 & 0 \\
		0 & 0 & 0 & 0
	\end{bmatrix} \oplus
	\tilde{P}^{k-1} \left(\T{3}\right)\begin{bmatrix}
		0 & x_3 & -x_2 & 0 \\
		\ast & 0 & x_1 & 0 \\
		\ast & \ast & 0 & 0 \\
		0 & 0 & 0 & 0
	\end{bmatrix} \Bigg),
\end{align*}
\begin{align*}
	\Vk{k}{3}(\W{4}) := & P^{k-1} \left(\T{3}\right) \times \left[0, 0, 0, P^{k}\left(\T{1}\right)  \right]^{T} \\
	&\oplus P^{k-1}\left(\T{1}\right) \times \Bigg( \left[P^{k-1}\left(\T{3}\right), P^{k-1}\left(\T{3}\right), P^{k-1}\left(\T{3}\right), 0 \right]^{T} \oplus \tilde{P}^{k-1} \left(\T{3}\right) \left[x_1, x_2, x_3, 0 \right]^{T} \Bigg), \\[1.0ex]
	\Vk{k}{4}(\W{4}) := & P^{k-1}\left(\T{1}\right) \times P^{k-1}\left(\T{3}\right).
\end{align*}
We can now construct the following exact sequence
\begin{align*}
	\begin{matrix}
		\Vk{k}{0}(\W{4}) & \arrow{r}{d^{\left(0\right)}} & \Vk{k}{1}(\W{4}) & \arrow{r}{d^{\left(1\right)}} & \Vk{k}{2}(\W{4}) & \arrow{r}{d^{\left(2\right)}} & \Vk{k}{3}(\W{4})  & \arrow{r}{d^{\left(3\right)}} & \Vk{k}{4}(\W{4}).
	\end{matrix}
\end{align*}

\subsection{N{\'e}d{\'e}lec-Raviart-Thomas Sequence}
We can construct a convenient sequence using the well-known N{\'e}d{\'e}lec finite elements of the first kind and Raviart-Thomas finite elements
\begin{subequations}
\begin{align}
	\Vk{k}{0}(\W{4}) = & CG^{k}\left(\T{1}\right) \times CG
	^{k} \left(\T{3}\right), \label{rt_prism_0} \\[1.0ex]
	\Vk{k}{1}(\W{4}) = & \left(CG^{k}\left(\T{1}\right) \times \begin{bmatrix}
		N^{k-1}\left(\T{3}\right)\\ 0
	\end{bmatrix} \right) \oplus \left( \left[0, 0, 0, DG^{k-1}\left(\T{1}\right) \right]^{T} \times CG^{k}\left(\T{3}\right) \right), \label{rt_prism_1} \\[1.0ex]
	\nonumber \Vk{k}{2}(\W{4}) = & \left( DG^{k-1}\left(\T{1} \right) \times \begin{bmatrix} \begin{bmatrix}
			0 & 0 & 0 \\
			0 & 0 & 0 \\
			0 & 0 & 0
		\end{bmatrix}
		& N^{k-1}\left(\T{3} \right) \\
		\ast & 0
	\end{bmatrix} \right) \\[1.0ex]
	&\oplus \left( CG^{k} \left(\T{1}\right) \times \begin{bmatrix}
		0 & RT^{k-1}_{3} (\T{3})  & -RT^{k-1}_{2} (\T{3}) & 0 \\
		\ast & 0 & RT^{k-1}_{1} (\T{3}) & 0 \\
		\ast & \ast & 0 & 0 \\
		0 & 0 & 0 & 0
	\end{bmatrix} \right), \label{rt_prism_2} \\[1.0ex]
	\Vk{k}{3}(\W{4}) = & \left(\left[0,0,0, CG^{k}\left(\T{1}\right) \right]^{T} \times DG^{k-1} \left(\T{3}\right) \right) \oplus \left( DG^{k-1}\left(\T{1}\right) \times \begin{bmatrix}
		RT^{k-1} \left(\T{3}\right) \\ 0
	\end{bmatrix} \right), \label{rt_prism_3} \\[1.0ex]
	\Vk{k}{4}(\W{4}) = & DG^{k-1}\left(\T{1}\right) \times DG^{k-1}\left(\T{3}\right). \label{rt_prism_4}
\end{align}
\end{subequations}
Here, we can define the following well-known scalar spaces
\begin{align*}
	CG^{k}\left(\mathcal{T}_h\right) &:=\left\{u \in H^{1}\left(\Omega\right) : u|_{\T{3}} \in P^{k}(\T{3}), \; \forall \T{3} \in \mathcal{T}_h\right\}, \\[1.0ex]
	DG^{k}\left(\mathcal{T}_h\right) &:=\left\{u \in L^{2}\left(\Omega\right) : u|_{\T{3}} \in P^{k}(\T{3}), \; \forall \T{3} \in \mathcal{T}_h\right\},
\end{align*}
and the following vector spaces
\begin{align*}    
	N^{k}\left(\mathcal{T}_h\right) &:= \left\{u \in H\left(\text{curl}, \Omega\right) : u|_{\T{3}} \in \left(P^{k}\left(\T{3}\right)\right)^{3} \oplus \left[x \times \left(P^{k}\left(\T{3}\right)\right)^{3} \right], \; \forall \T{3} \in \mathcal{T}_h \right\}, \\[1.0ex]
	RT^{k}\left(\mathcal{T}_h\right) &:= \left\{u \in H\left(\text{div},\Omega\right) : u|_{\T{3}}\in \left(P^{k}\left(\T{3}\right)\right)^{3} + x P^{k}\left(\T{3}\right), \; \forall \T{3} \in \mathcal{T}_h \right\}.
\end{align*}
\subsection{Bubble Spaces}
Here, we introduce the following bubble spaces which act as complementary spaces to the full spaces in Eqs.~\eqref{rt_prism_0}--\eqref{rt_prism_4}
\begin{subequations}
\begin{align}
    \Vkcirc{k}{0}(\W{4}) &:= \text{span} \left\{ \vartheta_{ij\ell} (x_1, x_2, x_3) \vartheta_{m}(x_4)\right\}, \label{rt_prism_0_bubble} \\[1.0ex]
    \nonumber \Vkcirc{k}{1}(\W{4}) &:= \text{span} \left\{ \begin{bmatrix} \Phi_{ij\ell}^{r}(x_1, x_2, x_3) \\ 0 \end{bmatrix} \vartheta_{m}(x_4) \right\} \\
    & \oplus \; \, \text{span} \left\{ \vartheta_{ij\ell}(x_1, x_2, x_3) \left[0, 0, 0, \varrho_{m} (x_4) \right]^T  \right\}, \label{rt_prism_1_bubble} \\[1.0ex]
    \nonumber \Vkcirc{k}{2}(\W{4}) &:= \text{span} \left\{\begin{bmatrix} \begin{bmatrix} 0 & 0 & 0 \\
    0 & 0 & 0 \\
    0 & 0 & 0 \end{bmatrix} & \Phi_{ij\ell}^{r}(x_1, x_2, x_3) \\
    \ast & 0 \end{bmatrix} \varrho_{m}(x_4) \right\}  \\
    &\oplus \; \, \text{span} \left\{ \begin{bmatrix} 0 & \left[\Psi_{ij\ell}^{r}(x_1, x_2, x_3)\right]_{3} & -\left[\Psi_{ij\ell}^{r}(x_1, x_2, x_3)\right]_{2} & 0 \\
    \ast & 0 & \left[\Psi_{ij\ell}^{r}(x_1, x_2, x_3)\right]_{1} & 0 \\
    \ast & \ast & 0 & 0 \\
    0 & 0 & 0 & 0 \end{bmatrix} \vartheta_{m}(x_4) \right\}, \label{rt_prism_2_bubble} \\[1.0ex]
    \nonumber \Vkcirc{k}{3}(\W{4}) &:= \text{span} \left\{ \varrho_{ij\ell}(x_1, x_2, x_3) \left[0, 0, 0, \vartheta_{m}(x_4) \right]^{T} \right\} \\
    &\oplus \; \, \text{span} \left\{  \begin{bmatrix} \Psi_{ij\ell}^{r} (x_1, x_2, x_3) \\ 0 \end{bmatrix} \varrho_{m}(x_4) \right\}. \label{rt_prism_3_bubble}
\end{align}
\end{subequations}
Here, the $\vartheta_{m}$'s are H1-conforming bubble functions of degree $k$ on line segments
\begin{align*}
    \vartheta_{m}(x_4) = L_m\left(\nu_2\right),
\end{align*}
where $m = 2, \ldots, k$ is the indexing parameter, $\nu = \nu(x_4) = (\nu_1, \nu_2)$ are the barycentric coordinates for the segment, and $L_m$ are integrated and scaled Legendre polynomials (see~\cite{fuentes2015orientation}). In a similar fashion, the $\vartheta_{ij\ell}$'s are H1-conforming bubble functions of degree $k$ on tetrahedra
\begin{align*}
    \vartheta_{ij\ell}(x_1, x_2, x_3) &= L_i \left(\frac{\lambda_2}{\lambda_1 + \lambda_2} \right) L_{j}^{2i} \left(\frac{\lambda_3}{\lambda_1 + \lambda_2 + \lambda_3} \right) L_{\ell}^{2(i+j)}\left(\lambda_4\right) \left(\lambda_1 + \lambda_2\right)^{i} \left(\lambda_1 + \lambda_2 + \lambda_3\right)^{j},
\end{align*}
where $i \geq 2$, $j\geq 1$, $\ell \geq 1$, $n = i +j + \ell = 4, \ldots, k$ are the indexing parameters, $\lambda = \lambda(x) = \lambda(x_1, x_2, x_3) = (\lambda_1, \lambda_2, \lambda_3, \lambda_4)$ are barycentric coordinates for the tetrahedron, and $L_j^{\alpha}$ are integrated and scaled Jacobi polynomials. 

In addition, the $\Phi_{ij\ell}^{r}$'s are H(curl)-conforming bubble functions of degree $k-1$ on tetrahedra
\begin{align*}
    \Phi_{ij\ell}^{r}(x_1, x_2, x_3) = &P_i\left(\frac{\lambda_b}{\lambda_a + \lambda_b} \right)  L_{j}^{2i+1}\left(\frac{\lambda_c}{\lambda_a + \lambda_b + \lambda_c} \right) L_{\ell}^{2(i+j)}\left(\lambda_d \right) \\[1.0ex]
    &\cdot \left(\lambda_a \nabla \lambda_b - \lambda_b \nabla \lambda_a \right) \left(\lambda_a + \lambda_b\right)^{i} \left(\lambda_a + \lambda_b + \lambda_c\right)^{j},
\end{align*}
%
where $i \geq 0$, $j \geq 1$, $\ell \geq 1$, $n = i +j +\ell = 2, \ldots, k-1$ are the indexing parameters, and $P_i$ are the shifted and scaled Legendre polynomials. In addition, for $r = 1,2,3$ we set $(a,b,c,d) = (1,2,3,4)$, $(a,b,c,d) = (2,3,4,1)$, and $(a,b,c,d) = (3,4,1,2)$, respectively.

Next, the $\Psi_{ij\ell}^{r}$'s are H(div)-conforming bubble functions of degree $k-1$ on tetrahedra
\begin{align*}
    \Psi_{ij\ell}^{r}(x_1, x_2, x_3) = &P_i\left(\frac{\lambda_b}{\lambda_a + \lambda_b}\right) P_j^{2i+1}\left(\frac{\lambda_c}{\lambda_a + \lambda_b + \lambda_c}\right) L_{\ell}^{2(i+j+1)}\left(\lambda_d \right) \\[1.0ex]
    &\cdot \left(\lambda_a \nabla \lambda_b \times \nabla \lambda_{c} + \lambda_b \nabla \lambda_{c} \times \nabla \lambda_{a} + \lambda_{c} \nabla\lambda_{a} \times \nabla \lambda_{b} \right) \left(\lambda_a + \lambda_b\right)^{i} \left(\lambda_a + \lambda_b + \lambda_c\right)^{j},
\end{align*}
%
%
where $i \geq 0$, $j \geq 0$, $\ell \geq 1$, $n = i +j +\ell = 1, \ldots, k-1$ are the indexing parameters, and $P_j^{\alpha}$ are the shifted Jacobi polynomials. In addition, for $r = 1,2,3$ we set $(a,b,c,d) = (1,2,3,4)$, $(a,b,c,d) = (2,3,4,1)$, and $(a,b,c,d) = (3,4,1,2)$, respectively.

Furthermore, the $\varrho_{m}$'s are the L2-conforming bubble functions of degree $k-1$ on line segments
\begin{align*}
    \varrho_{m}(x_4) = P_m\left(\nu_2\right), 
\end{align*}
where $m = 0, \ldots, k-1$ is the indexing parameter. Similarly, the $\varrho_{ij\ell}$'s are the L2-conforming bubble functions of degree $k-1$ on tetrahedra
\begin{align*}
    \varrho_{ij\ell}(x_1, x_2, x_3) &= P_i\left(\frac{\lambda_2}{\lambda_1 + \lambda_2}\right) P_{j}^{2i+1}\left(\frac{\lambda_3}{\lambda_1+\lambda_2+\lambda_3}\right) P_{\ell}^{2(i+j+1)}\left(\lambda_4 \right) \left(\lambda_1 + \lambda_2\right)^{i} \left(\lambda_1 + \lambda_2 + \lambda_3\right)^{j},
\end{align*}
where $i \geq 0$, $j \geq 0$, $\ell \geq 0$, $n = i+j+\ell = 0, \ldots, k-1$ are the indexing parameters.

\subsection{Restatement of Polynomial Functions}

We can now restate the polynomial functions from the previous section in terms of polynomial spaces $P^{k}(\T{1})$, $P^{k}(\T{3})$, $dP^{k}(\T{1})$, and $dP^{k}(\T{3})$. The former two polynomial spaces, $P^{k}(\T{1})$ and $P^{k}(\T{3})$, are associated with dofs that reside on the boundary \emph{and} within the interior of each element, whereas the latter two spaces, $dP^{k}(\T{1})$ and $dP^{k}(\T{3})$, are associated with dofs which reside only within the interior. With this in mind, let us consider
\begin{align*}
       \vartheta_{m}(x_4) = \vartheta^{b}(x_4) p_m(x_4), \qquad \forall p_{m}(x_4)\in P^{k-2}(\T{1}),
\end{align*}
where 
\begin{align*}
    \vartheta^{b}(x_4) = \nu_1 \nu_2.
\end{align*}
Here, $\vartheta^{b}(x_4)$ is a non-negative bubble function shared by all members of the set. Next, consider
\begin{align*}
    \vartheta_{ij\ell}(x_1, x_2, x_3) = \vartheta^{b}(x_1, x_2, x_3) p_{ij\ell}(x_1, x_2, x_3), \qquad \forall p_{ij\ell}(x_1, x_2, x_3) \in P^{k-4}(\T{3}),
\end{align*}
where
\begin{align*}
    \vartheta^{b}(x_1, x_2, x_3) &= \lambda_1 \lambda_2 \lambda_3 \lambda_4.
\end{align*}
In addition, consider
\begin{align*}
    \Phi_{ij\ell}^{r}(x_1, x_2, x_3) = \Phi^{b,r}(x_1, x_2, x_3) g_{ij\ell}^{r}(x_1, x_2, x_3) N^{r}(x_1, x_2, x_3), \qquad \forall g_{ij\ell}^{r}(x_1, x_2, x_3) \in P^{k-3}(\T{3}),
\end{align*}
where
\begin{align*}
    \Phi^{b,r}(x_1, x_2, x_3) &= \lambda_{c} \lambda_{d}, \qquad N^{r}(x_1,x_2,x_3) = \lambda_{a} \nabla \lambda_{b} - \lambda_{b} \nabla \lambda_{a}.
\end{align*}
Furthermore, consider 
\begin{align*}
    \Psi_{ij\ell}^{r}(x_1, x_2, x_3) = \Psi^{b,r}(x_1, x_2, x_3) w_{ij\ell}^{r}(x_1, x_2, x_3) \mathcal{N}^{r}(x_1, x_2, x_3), \qquad \forall w_{ij\ell}^{r}(x_1, x_2, x_3) \in P^{k-2}(\T{3}),
\end{align*}
where
\begin{align*}
    \Psi^{b,r}(x_1, x_2, x_3) &=  \lambda_{d}, \qquad \mathcal{N}^{r}(x_1,x_2,x_3) = \lambda_{a} \nabla \lambda_{b} \times \nabla \lambda_{c} + \lambda_b \nabla \lambda_{c} \times \nabla \lambda_{a} + \lambda_{c} \nabla\lambda_{a} \times \nabla \lambda_{b}.
\end{align*}
Lastly, consider
\begin{align*}
    \varrho_{m}(x_4) = v_m(x_4), \qquad \forall v_{m}(x_4) \in dP^{k-1}(\T{1}),
\end{align*}
%
and
\begin{align*}
    \varrho_{ij\ell}(x_1, x_2, x_3) = v_{ij\ell}(x_1, x_2, x_3), \qquad \forall v_{ij\ell}(x_1, x_2, x_3) \in dP^{k-1}(\T{3}).
\end{align*}
%

\subsection{Restatement of Bubble Spaces}
We can now restate the bubble space definitions (Eqs.~\eqref{rt_prism_0_bubble}--\eqref{rt_prism_3_bubble}) in terms of the polynomial functions from the previous section
\begin{subequations}
\begin{align}
    \Vkcirc{k}{0}(\W{4}) := & \, \text{span} \left\{ \vartheta^{b}(x_1,x_2,x_3)\vartheta^{b}(x_4) p_{ij\ell}(x_1,x_2,x_3)p_m(x_4) \right\}, \label{rt_prism_0_bubble_simp} \\[1.0ex] \nonumber &\forall p_{m}(x_4)\in P^{k-2}(\T{1}), \quad p_{ij\ell}(x_1,x_2,x_3) \in P^{k-4}(\T{3}), \\[1.0ex]
    \nonumber \Vkcirc{k}{1}(\W{4}) := & \text{span} \left\{  \begin{bmatrix} \Phi^{b,r}(x_1,x_2,x_3) \vartheta^{b}(x_4) g_{ij\ell}^{r}(x_1,x_2,x_3) p_m(x_4) N^{r}(x_1, x_2, x_3)  \\[1.0ex] 0 \end{bmatrix} \right\} \\[1.0ex]
    \oplus & \, \text{span} \left\{ \left[0, 0, 0, \vartheta^{b}(x_1,x_2,x_3) p_{ij\ell}(x_1,x_2,x_3) v_{m}(x_4) \right]^T    \right\}, \label{rt_prism_1_bubble_simp} \\[1.0ex]
    \nonumber &\forall p_{m}(x_4)\in P^{k-2}(\T{1}), \quad  g_{ij\ell}^{r}(x_1,x_2,x_3) \in P^{k-3}(\T{3}), \\[1.0ex]
    \nonumber &\forall v_{m}(x_4) \in dP^{k-1}(\T{1}), \quad p_{ij\ell}(x_1,x_2,x_3) \in P^{k-4}(\T{3}),
\end{align}

\begin{align}
    \nonumber \Vkcirc{k}{2}(\W{4}) &:= \text{span} \left\{ \VtoM{\begin{bmatrix} 0 \\[1.0ex] 0 \\[1.0ex] \left[\Phi^{b,r}(x_1,x_2,x_3)  g_{ij\ell}^{r}(x_1,x_2,x_3) v_{m}(x_4) N^{r}(x_1,x_2,x_3) \right]_{1}  \\[1.0ex] 0 \\[1.0ex] \left[\Phi^{b,r}(x_1,x_2,x_3) g_{ij\ell}^{r}(x_1,x_2,x_3) v_{m}(x_4) N^{r}(x_1, x_2, x_3) \right]_{2}  \\[1.0ex] \left[\Phi^{b,r}(x_1,x_2,x_3)  g_{ij\ell}^{r}(x_1,x_2,x_3) v_{m}(x_4) N^{r}(x_1, x_2, x_3) \right]_{3}  \end{bmatrix}} \right\} \\[1.0ex]
    & \oplus \text{span} \left\{\VtoM{\begin{bmatrix}
	\left[ \Psi^{b,r}(x_1,x_2,x_3) \vartheta^{b}(x_4) w_{ij\ell}^{r}(x_1,x_2,x_3) p_{m}(x_4) \mathcal{N}^{r}(x_1,x_2,x_3) \right]_{3} \\[1.0ex]
	\left[ \Psi^{b,r}(x_1,x_2,x_3) \vartheta^{b}(x_4) w_{ij\ell}^{r}(x_1,x_2,x_3) p_{m}(x_4) \mathcal{N}^{r}(x_1,x_2,x_3) \right]_{2} \\[1.0ex]
	0\\[1.0ex]
	\left[ \Psi^{b,r}(x_1,x_2,x_3) \vartheta^{b}(x_4) w_{ij\ell}^{r}(x_1,x_2,x_3) p_{m}(x_4) \mathcal{N}^{r}(x_1,x_2,x_3) \right]_{1} \\[1.0ex]
	0\\[1.0ex]
	0
	\end{bmatrix}} \right\}, \label{rt_prism_2_bubble_simp} 
    \\[1.0ex]
    \nonumber & \forall v_{m}(x_4)\in dP^{k-1}(\T{1}), \quad 
    g_{ij\ell}^{r}(x_1,x_2,x_3) \in P^{k-3}(\T{3}), \\[1.0ex]
    \nonumber & \forall p_{m}(x_4)\in P^{k-2}(\T{1}), \quad 
    w_{ij\ell}^{r}(x_1,x_2,x_3) \in P^{k-2}(\T{3}),
\end{align}
\begin{align}
    \nonumber \Vkcirc{k}{3}(\W{4}) &:= \text{span} \left\{ \left[0, 0, 0, \vartheta^{b}(x_4)  v_{ij\ell}(x_1, x_2, x_3) p_{m}(x_4) \right]^{T} \right\} \\
    &\oplus \text{span} \left\{ \begin{bmatrix}  \Psi^{b,r} (x_1, x_2, x_3) w_{ij\ell}^{r}(x_1, x_2, x_3) v_{m}(x_4) \mathcal{N}^{r}(x_1, x_2, x_3)  \\[1.0ex] 0  \end{bmatrix} \right\}, \label{rt_prism_3_bubble_simp}
    \\[1.0ex]
    \nonumber & \forall p_{m}(x_4)\in P^{k-2}(\T{1}), \quad v_{ij\ell}(x_1,x_2,x_3)\in dP^{k-1}(\T{3}), \\[1.0ex]
    \nonumber &\forall v_{m}(x_4) \in dP^{k-1}(\T{1}), \quad 
    w_{ij\ell}^{r}(x_1,x_2,x_3) \in P^{k-2}(\T{3}). 
\end{align}
\end{subequations}


\subsection{Degrees of Freedom on the Reference Tetrahedral Prism, $\W{4}$}

Our objective is to construct degrees of freedom for the N{\'e}d{\'e}lec-Raviart-Thomas-based sequence (Eqs.~\eqref{rt_prism_0}--\eqref{rt_prism_4}) on the reference tetrahedral prism $\W{4}$. We recall from \autoref{Table:4DelementsSubs}, that the reference tetrahedral prism has 8 vertices, 16 edges, 8 triangular faces, 6 quadrilateral faces, 2 tetrahedral facets, and 4 triangular-prismatic facets. The degrees of freedom on these vertices, edges, faces, and facets of the tetrahedral prism are inherited directly from the degrees of freedom for lower dimensional entities in 0, 1, 2, and 3 dimensions, respectively. Therefore, it remains for us to construct degrees of freedom for the interior of the tetrahedral prism. We will construct explicit expressions for these degrees of freedom for 0-, 1-, 2-, 3-, and 4-forms in what follows.

\subsection{Dofs for 0-forms on $\W{4}$}
The polynomial 0-forms on the tetrahedral prism, $\Vk{k}{0}(\W{4})$, have the following total dimension 
\begin{align*}
\text{dim}\left(\Vk{k}{0}(\W{4})\right) &= \text{dim}(\Sigma^{k,0}(\W{4})) \\
&= \text{dim}(CG^{k}\left(\T{1}\right) \times CG
	^{k} \left(\T{3}\right)) = \frac{1}{6}(k+1)^{2}(k+2)(k+3).
\end{align*}
The 0-forms have vertex, edge, face, and facet traces in accordance with Eqs.~\eqref{eq:edge-0}, \eqref{eq:triangle-0}, \eqref{eq:square-0}, \eqref{eq:tet0}, and \eqref{eq:triprism0}. As a result, the dimension of the trace degrees of freedom, $\Sigma_{trace}^{k,0}(\W{4})$, can be computed as follows
\begin{align*}
\text{dim} (\Sigma_{trace}^{k,0}(\W{4})) &= 8 + 16 \, \text{dim}\left(P^{k-2}(\T{1}) \right) + 8 \, \text{dim} \left( P^{k-3}(\T{2}) \right) + 6 \, \text{dim} \left( Q^{k-2,k-2}(\C{2}) \right) \\
&+ 2 \, \text{dim}\left( P^{k-4}(\T{3}) \right) + 4 \left( \text{dim}(Q^{k-2}(\C{1})) \times \text{dim}(P^{k-3}(\T{2})) \right)
\\
&= 8 + 16(k-1) + \frac{8}{2}(k-2)(k-1) + 6(k-1)^2 \\
&+ \frac{2}{6}(k-3)(k-2)(k-1) + \frac{4}{2}(k-2)(k-1)^2 \\
&= \frac{1}{3}k(7k^2 + 17).
\end{align*}
In addition, the volumetric degrees of freedom for the 0-form proxy $u$ are given by
\begin{align}
    \Sigma_{vol}^{k,0}(\W{4}):= \left\{u \rightarrow \int_{\W{4}} u q \, \vartheta^{b}(x_1, x_2, x_3) \vartheta^{b}(x_4), \qquad q\in P^{k-2}(\T{1}) \times P^{k-4}(\T{3})\right\}, \label{rt_zero_form}
\end{align}
and
\begin{align*}
    \text{dim}(\Sigma_{vol}^{k,0}(\W{4})) = \frac{1}{6}(k-3)(k-2)(k-1)^2.
\end{align*}
In a natural fashion, one can show that the total number of degrees of freedom on the tetrahedral prism is equal to the sum of the trace and volumetric degrees of freedom
\begin{align*}
    \text{dim}(\Sigma^{k,0}(\W{4})) &= \text{dim} (\Sigma_{trace}^{k,0}(\W{4})) + \text{dim}(\Sigma_{vol}^{k,0}(\W{4})).
\end{align*}
It remains for us to prove unisolvency.
\begin{lemma}
    Let $u\in \Vk{k}{0}(\W{4})$ be a polynomial 0-form for which all the degrees of freedom $\Sk{k}{0}(\W{4})$ vanish. Then $u \equiv 0.$
    \label{prism_lemma_0}
\end{lemma}

\begin{proof}
	
    Since all the trace degrees of freedom of the form given by Eqs.~\eqref{eq:edge-0}, \eqref{eq:triangle-0}, \eqref{eq:square-0}, \eqref{eq:tet0}, and \eqref{eq:triprism0} vanish, then we conclude that $u$ resides in the bubble space, i.e.~$u \in \Vkcirc{k}{0}(\W{4})$. Therefore, we can express $u$ in accordance with Eq.~\eqref{rt_prism_0_bubble_simp} as follows
    \begin{align*}
        u &= \sum_{ij\ell m} \widehat{u}_{ij\ell m} \, p_{ij\ell}(x_1,x_2,x_3) p_m(x_4) \vartheta^{b} (x_1,x_2,x_3) \vartheta^{b}(x_4),
    \end{align*}
    where
    \begin{align*}
         p_{ij\ell}(x_1,x_2,x_3) \in P^{k-4}(\T{3}), \quad p_{m}(x_4)\in P^{k-2}(\T{1}).
    \end{align*}
    We complete the proof by substituting $u$ (from above) and 
    \begin{align*}
        q = \sum_{ij\ell m} \widehat{u}_{ij\ell m} \, p_{ij\ell}(x_1,x_2,x_3) p_m(x_4),
    \end{align*}
    into Eq.~\eqref{rt_zero_form}. Under these circumstances, the only way the volumetric degrees of freedom are guaranteed to vanish, is if $u$ vanishes.
\end{proof}

\subsection{Dofs for 1-forms on $\W{4}$}
The polynomial 1-forms on the tetrahedral prism, $\Vk{k}{1}(\W{4})$, have the following total dimension 
\begin{align*}
    \text{dim}(\Vk{k}{1}(\W{4})) &= \text{dim}(\Sigma^{k,1}(\W{4})) \\&= \text{dim} \left(CG^{k}\left(\T{1}\right) \times \begin{bmatrix}
		N^{k-1}\left(\T{3}\right)\\ 0
	\end{bmatrix} \right) \\ 
    &+ \text{dim} \left(  \left[0, 0, 0, DG^{k-1}\left(\T{1}\right) \right]^{T} \times CG^{k}\left(\T{3}\right) \right) \\
    & = \frac{1}{2} k(k+1)(k+2)(k+3) + \frac{1}{6} k(k+1)(k+2)(k+3) \\
    & = \frac{2}{3} k(k+1)(k+2)(k+3).
\end{align*}
The 1-forms have edge, face, and facet traces in accordance with Eqs.~\eqref{eq:edge-1}, \eqref{eq:triangle-1}, \eqref{eq:square-1}, \eqref{eq:tet1}, and \eqref{eq:triprism1}. As a result, the dimension of the trace degrees of freedom, $\Sigma^{k,1}_{trace}(\W{4})$, can be computed as follows
\begin{align*}
\text{dim} \left(\Sigma^{k,1}_{trace}(\W{4})\right) &= 16 \, \text{dim}\left(P^{k-1}(\T{1}) \right) + 8 \, \text{dim} \left( (P^{k-2}(\T{2}))^{2} \right) \\
&+ 6 \left( \text{dim} (Q^{k-2,k-1}(\C{2})) + \text{dim} (Q^{k-1,k-2}(\C{2})) \right) + 2 \, \text{dim}\left( (P^{k-3}(\T{3}))^3 \right) \\
&+ 4 \left( 2\text{dim}(Q^{k-2}(\C{1})) \times \text{dim}(P^{k-2} (\T{2})) + \text{dim}(Q^{k-1}(\C{1})) \times \text{dim}(P^{k-3}(\T{2})) \right)
\\
&= 16 k +8(k-1)k + 6(2(k-1)k) + \frac{2}{2} (k-2)(k-1)k \\
&+ 4\left( k(k-1)^2 + \frac{1}{2} (k-2)(k-1) k \right) \\
&= k (7k^2 + 3k + 6).
\end{align*}
In addition, the volumetric degrees of freedom for the 1-form proxy $E$ are as follows
\begin{align}
    \nonumber \Sigma^{k,1}_{vol,1}(\W{4}) := &\Bigg\{E \rightarrow \int_{\W{4}} E \cdot \left( \sum_{r} q^{r} \Phi^{b,r}(x_1, x_2, x_3) \vartheta^{b}(x_4) \begin{bmatrix} N^{r}(x_1, x_2, x_3) \\[1.0ex] 0 \end{bmatrix} \right), \\[1.0ex] &q^{r} \in P^{k-2}(\T{1}) \times P^{k-3} (\T{3}) \Bigg\}, \qquad r = 1,2,3,   \label{rt_one_form_a} \\[1.0ex]
    \nonumber \Sigma^{k,1}_{vol,2}(\W{4}) := &\Bigg\{E \rightarrow \int_{\W{4}} E \cdot \left[0, 0, 0, q^{4} \vartheta^{b}(x_1, x_2, x_3)\right]^{T}, \\[1.0ex] &q^{4}\in dP^{k-1}(\T{1}) \times P^{k-4}(\T{3}) \Bigg\}. \label{rt_one_form_b}
\end{align}
Thereafter, we define
\begin{align*}
    \Sigma^{k,1}_{vol}(\W{4}) &:= \Sigma^{k,1}_{vol,1}(\W{4}) \cup \Sigma^{k,1}_{vol,2}(\W{4}),
\end{align*}
and
\begin{align*}
    \text{dim}\left(\Sigma^{k,1}_{vol}(\W{4})\right) &= \text{dim}\left(\Sigma^{k,1}_{vol,1}(\W{4})\right) + \text{dim}\left(\Sigma^{k,1}_{vol,2}(\W{4}) \right) \\
    &=\frac{1}{2} k(k-1)^2 (k-2) + \frac{1}{6}(k-3)(k-2)(k-1)k.
\end{align*}
Evidently, we can show that the total number of degrees of freedom is composed from the sum of interior and facet degrees of freedom:
\begin{align*}
    \text{dim}\left(\Sk{k}{1}(\W{4}) \right) &= \text{dim} \left(\Sigma^{k,1}_{trace}(\W{4}) \right) + \text{dim}\left(\Sigma^{k,1}_{vol}(\W{4})\right).
\end{align*}
It then remains for us to prove unisolvency.

\begin{lemma}
    Let $E\in \Vk{k}{1}(\W{4})$ be a polynomial 1-form for which all the degrees of freedom $\Sk{k}{1}(\W{4})$ vanish. Then $E \equiv 0.$
    \label{prism_lemma_1}
\end{lemma}

\begin{proof}
	
    Since all the trace degrees of freedom of the form given by Eqs.~\eqref{eq:edge-1}, \eqref{eq:triangle-1}, \eqref{eq:square-1}, \eqref{eq:tet1}, and \eqref{eq:triprism1} vanish, the polynomial 1-form $E$ has zero traces, and is hence in $\Vkcirc{k}{1}(\W{4})$. It therefore has the form in Eq.~\eqref{rt_prism_1_bubble_simp}, i.e.,
    \begin{align*}
        E = \begin{bmatrix}
           \sum_{ij\ell m r} \widehat{E}_{ij \ell m}^{r} g_{ij \ell}^{r}(x_1, x_2, x_3) p_m(x_4) \Phi^{b,r}(x_1,x_2,x_3) \vartheta^{b}(x_4) N^{r}(x_1, x_2, x_3)  \\[1.0ex]
           \sum_{ij \ell m} \widehat{E}_{ij \ell m}^{4} p_{ij\ell}(x_1, x_2, x_3) v_m(x_4) \vartheta^{b}(x_1,x_2,x_3)
        \end{bmatrix},
    \end{align*}
    where
    \begin{align*}
        \nonumber & p_{m}(x_4)\in P^{k-2}(\T{1}), \quad g_{ij\ell}^{r}(x_1,x_2,x_3) \in P^{k-3}(\T{3}), \\[1.0ex]
        \nonumber & v_{m}(x_4) \in dP^{k-1}(\T{1}), \quad p_{ij\ell}(x_1,x_2,x_3) \in P^{k-4}(\T{3}).
    \end{align*}
    The proof follows immediately by choosing test functions,
    \begin{align*}
        q^{r} &= \sum_{ij\ell m} \widehat{E}_{ij \ell m}^{r} g_{ij \ell}^{r}(x_1, x_2, x_3) p_m(x_4), \\[1.0ex]
        q^{4} &= \sum_{ij \ell m} \widehat{E}_{ij \ell m}^{4} p_{ij\ell}(x_1, x_2, x_3) v_m(x_4),
    \end{align*}
    and thereafter substituting these test functions and $E$ (from above) into Eqs.~\eqref{rt_one_form_a} and \eqref{rt_one_form_b}. Under these conditions, the vanishing of the associated volumetric degrees of freedom is only possible if $E$ vanishes.
\end{proof}

\subsection{Dofs for 2-forms on $\W{4}$}
The polynomial 2-forms on the tetrahedral prism, $\Vk{k}{2}(\W{4})$, have the following total dimension 
\begin{align*}
    \text{dim}\left(\Vk{k}{2}(\W{4}) \right) &= \text{dim}\left(\Sk{k}{2}(\W{4}) \right) \\
    &= \text{dim} \left( DG^{k-1}\left(\T{1} \right) \times \begin{bmatrix} \begin{bmatrix}
			0 & 0 & 0 \\
			0 & 0 & 0 \\
			0 & 0 & 0
		\end{bmatrix}
		& N^{k-1}\left(\T{3} \right) \\
		\ast & 0
	\end{bmatrix} \right) \\[1.0ex]
	&+ \text{dim} \left( CG^{k} \left(\T{1}\right) \times \begin{bmatrix}
		0 & RT^{k-1}_{3} (\T{3})  & -RT^{k-1}_{2} (\T{3}) & 0 \\
		\ast & 0 & RT^{k-1}_{1} (\T{3}) & 0 \\
		\ast & \ast & 0 & 0 \\
		0 & 0 & 0 & 0
	\end{bmatrix} \right) \\
    & = \frac{1}{2} k^2 (k+2)(k+3) + \frac{1}{2} k (k+1)^2 (k+3).
\end{align*}
The 2-forms have face and facet traces in accordance with Eqs.~\eqref{eq:triangle-2}, \eqref{eq:square-2}, \eqref{eq:tet2}, and \eqref{eq:triprism2}. As a result, the dimension of the trace degrees of freedom, $\Sk{k}{2}_{trace}(\W{4})$, can be computed as follows
\begin{align*}
    \text{dim}\left(\Sk{k}{2}_{trace}(\W{4}) \right) &= 8 \,\text{dim} (P^{k-1}(\T{2})) + 6 \, \text{dim}(Q^{k-1,k-1}(\C{2}))  +2 \, \text{dim}\left((P^{k-2}(\T{3}))^{3} \right)\\
    &+ 4\left(2\text{dim}(Q^{k-1}(\C{1})) \times \text{dim}(P^{k-2}(\T{2})) +  \text{dim}(Q^{k-2}(\C{1})) \times \text{dim}(P^{k-1}(\T{2})) \right) \\
    &= \frac{8}{2} k(k+1) + 6 k^2 + \frac{2}{2}(k-1)k(k+1) + 4 \left((k-1)k^2 + \frac{1}{2}(k-1)k(k+1) \right) \\
    &= k(7k^2+6k+1).
\end{align*}
In addition, the volumetric degrees of freedom for the 2-form proxy $F$ are as follows
\begin{align}
    \nonumber \Sk{k}{2}_{vol,1}(\W{4}) &:= \Bigg\{F \rightarrow \int_{\W{4}} F : \VtoM{\begin{bmatrix} 0 \\[1.0ex] 0 \\[1.0ex] \left[ \sum_{r} q^{r} \Phi^{b,r}(x_1, x_2, x_3) N^{r}(x_1, x_2, x_3) \right]_{1}  \\[1.0ex]
    0 \\[1.0ex] \left[ \sum_{r} q^{r} \Phi^{b,r}(x_1, x_2, x_3) N^{r}(x_1, x_2, x_3) \right]_{2} \\[1.0ex] \left[ \sum_{r} q^{r} \Phi^{b,r}(x_1, x_2, x_3) N^{r}(x_1, x_2, x_3) \right]_{3} \end{bmatrix}}
    \\[1.0ex]
  &q^{r} \in dP^{k-1}(\T{1}) \times P^{k-3}(\T{3})  \Bigg\}, \qquad r = 1,2,3, \label{rt_two_form_a}
\end{align}
\begin{align}
    \nonumber \Sk{k}{2}_{vol,2}(\W{4}) &:= \Bigg\{F \rightarrow \int_{\W{4}} F : \VtoM{\begin{bmatrix} \left[\sum_{r} q^{r} \Psi^{b,r}(x_1, x_2, x_3) \vartheta^{b}(x_4) \mathcal{N}^{r} (x_1, x_2, x_3) \right]_{3} \\[1.0ex] -\left[\sum_{r} q^{r} \Psi^{b,r}(x_1, x_2, x_3) \vartheta^{b}(x_4) \mathcal{N}^{r} (x_1, x_2, x_3) \right]_{2} \\[1.0ex] 0 \\[1.0ex] \left[\sum_{r} q^{r} \Psi^{b,r}(x_1, x_2, x_3) \vartheta^{b}(x_4) \mathcal{N}^{r} (x_1, x_2, x_3) \right]_{1} \\[1.0ex] 0 \\[1.0ex] 0 \end{bmatrix}} \\[1.0ex]
    &q^{r} \in P^{k-2}(\T{1}) \times P^{k-2}(\T{3}) \Bigg\}, \qquad r = 1,2,3. \label{rt_two_form_b}
\end{align}
Thereafter, we define
\begin{align*}
    \Sk{k}{2}_{vol}(\W{4}) &:= \Sk{k}{2}_{vol,1}(\W{4}) \cup \Sk{k}{2}_{vol,2}(\W{4}),
\end{align*}
and
\begin{align*}
    \text{dim}\left(\Sk{k}{2}_{vol}(\W{4}) \right) &= \text{dim}\left(\Sk{k}{2}_{vol,1}(\W{4}) \right) + \text{dim}\left(\Sk{k}{2}_{vol,2}(\W{4}) \right) \\
    &=\frac{1}{2}k^2(k-1)(k-2) + \frac{1}{2}(k-1)^2 k (k+1).
\end{align*}
Evidently, we can show that the following holds
\begin{align*}
    \text{dim}\left(\Sk{k}{2}(\W{4})\right) &= \text{dim} \left(\Sk{k}{2}_{trace}(\W{4}) \right) + \text{dim}\left(\Sk{k}{2}_{vol}(\W{4}) \right).
\end{align*}
It then remains for us to prove unisolvency.

\begin{lemma} \label{prism_lemma_2}
    Let $F\in \Vk{k}{2}(\W{4})$ be a polynomial 2-form for which all the degrees of freedom $\Sk{k}{2}(\W{4})$ vanish. Then $F \equiv 0.$
\end{lemma}

\begin{proof}
	
    Since all the trace degrees of freedom of the form given by Eqs.~\eqref{eq:triangle-2}, \eqref{eq:square-2}, \eqref{eq:tet2}, and \eqref{eq:triprism2} vanish, the polynomial 2-form $F$ has zero traces, and is hence in $\Vkcirc{k}{2}(\W{4})$. It therefore has the form in Eq.~\eqref{rt_prism_2_bubble_simp}, i.e.,
    \begin{align*}
        F = \VtoM{\begin{bmatrix} \left[\sum_{ij \ell m r} \widetilde{F}_{ij \ell m}^{r} w_{ij\ell}^{r}(x_1, x_2, x_3) p_m(x_4) \Psi^{b,r}(x_1, x_2, x_3) \vartheta^{b}(x_4) \mathcal{N}^{r}(x_1, x_2, x_3) \right]_{3} \\[1.0ex]
        \left[\sum_{ij \ell m r} \widetilde{F}_{ij \ell m}^{r} w_{ij\ell}^{r}(x_1, x_2, x_3) p_m(x_4) \Psi^{b,r}(x_1, x_2, x_3) \vartheta^{b}(x_4) \mathcal{N}^{r}(x_1, x_2, x_3) \right]_{2} \\[1.0ex] 
        \left[\sum_{ij \ell m r} \widehat{F}_{ij \ell m}^{r} g_{ij\ell}^{r}(x_1, x_2, x_3) v_m(x_4) \Phi^{b,r}(x_1, x_2, x_3) N^{r}(x_1, x_2, x_3) \right]_{1} \\[1.0ex]
        \left[\sum_{ij \ell m r} \widetilde{F}_{ij \ell m}^{r} w_{ij\ell}^{r}(x_1, x_2, x_3) p_m(x_4) \Psi^{b,r}(x_1, x_2, x_3) \vartheta^{b}(x_4) \mathcal{N}^{r}(x_1, x_2, x_3) \right]_{1} \\[1.0ex] 
        \left[\sum_{ij \ell m r} \widehat{F}_{ij \ell m}^{r} g_{ij\ell}^{r}(x_1, x_2, x_3) v_m(x_4) \Phi^{b,r}(x_1, x_2, x_3) N^{r}(x_1, x_2, x_3) \right]_{2} \\[1.0ex]
        \left[\sum_{ij \ell m r} \widehat{F}_{ij \ell m}^{r} g_{ij\ell}^{r}(x_1, x_2, x_3) v_m(x_4) \Phi^{b,r}(x_1, x_2, x_3) N^{r}(x_1, x_2, x_3) \right]_{3} \end{bmatrix}},
    \end{align*}
    where
    \begin{align*}
        \nonumber & \forall v_{m}(x_4)\in dP^{k-1}(\T{1}), \quad g_{ij \ell}^{r}(x_1, x_2, x_3) \in P^{k-3}(\T{3}), \\[1.0ex]
        \nonumber & \forall p_{m}(x_4)\in P^{k-2}(\T{1}), \quad w_{ij \ell}^{r}(x_1, x_2, x_3) \in P^{k-2}(\T{3}).
    \end{align*}
    The proof follows immediately by choosing the following test functions 
    \begin{align*}
        q^{r} &= \sum_{ij \ell m} \widehat{F}_{ij \ell m}^{r} g_{ij \ell}^{r}(x_1, x_2, x_3) v_{m}(x_4), \\[1.0ex]
        q^{r} &= \sum_{ij \ell m} \widetilde{F}_{ij \ell m}^{r} w_{ij\ell}^{r}(x_1, x_2, x_3) p_{m}(x_4),
    \end{align*}
    and thereafter substituting these test functions and $F$ (from above) into Eqs.~\eqref{rt_two_form_a} and \eqref{rt_two_form_b}, respectively. Under these conditions, the vanishing of the associated volumetric degrees of freedom is only possible if $F$ vanishes.
\end{proof}

\subsection{Dofs for 3-forms on $\W{4}$}
The polynomial 3-forms on the tetrahedral prism, $\Vk{k}{3}(\W{4})$, have the following total dimension 
\begin{align*}
    \text{dim}\left(\Vk{k}{3}(\W{4}) \right) &=  \text{dim}(\Sk{k}{3}(\W{4})) \\&= \text{dim}\left( \left[0,0,0, CG^{k}\left(\T{1}\right) \right]^{T} \times DG^{k-1} \left(\T{3}\right) \right) \\
    &+ \text{dim} \left( DG^{k-1}\left(\T{1}\right) \times \begin{bmatrix}
		RT^{k-1} \left(\T{3}\right) \\ 0
	\end{bmatrix} \right) \\
  &=\frac{1}{6} k(k+1)^2 (k+2) + \frac{1}{2} k^2 (k+1)(k+3).
\end{align*}
The 3-forms only have facet traces in accordance with Eqs.~\eqref{eq:tet3} and \eqref{eq:triprism3}. As a result, the dimension of the trace degrees of freedom, $\Sk{k}{3}_{trace}(\W{4})$, can be computed as follows
\begin{align*}
\text{dim} \left(\Sk{k}{3}_{trace}(\W{4}) \right) &= 2 \,\text{dim}\left( P^{k-1}(\T{3}) \right) + 4 \left( \text{dim}(Q^{k-1}(\C{1})) \times \text{dim}(P^{k-1}(\T{2})) \right)
\\
&= \frac{2}{6} k(k+1)(k+2) + \frac{4}{2}k^2 (k+1) \\
&= \frac{1}{3} k(7k^2 + 9k + 2).
\end{align*}
In addition, the volumetric degrees of freedom for the 3-form proxy $G$ are as follows
\begin{align}
    \nonumber \Sk{k}{3}_{vol,1}(\W{4}) := &\Bigg\{G \rightarrow \int_{\W{4}} G \cdot \left( \sum_{r} q^{r} \Psi^{b,r}(x_1, x_2, x_3) \begin{bmatrix} \mathcal{N}^{r}(x_1, x_2, x_3) \\[1.0ex] 0 \end{bmatrix} \right), \\[1.0ex] 
    &q^{r} \in  dP^{k-1}(\T{1}) \times P^{k-2}(\T{3}) \Bigg\}, \qquad r = 1,2,3,\label{rt_three_form_b} \\[1.0ex]
    \nonumber \Sk{k}{3}_{vol,2}(\W{4}) := &\Bigg\{G \rightarrow \int_{\W{4}} G \cdot \left[0, 0, 0, q^{4} \vartheta^{b}(x_4) \right]^{T}, \\[1.0ex] &q^{4}\in P^{k-2}(\T{1}) \times dP^{k-1}(\T{3})  \Bigg\}. \label{rt_three_form_a} 
\end{align}
Thereafter, we define
\begin{align*}
    \Sk{k}{3}_{vol}(\W{4}) &:= \Sk{k}{3}_{vol,1}(\W{4}) \cup \Sk{k}{3}_{vol,2}(\W{4}),
\end{align*}
and
\begin{align*}
    \text{dim}\left(\Sk{k}{3}_{vol}(\W{4}) \right) &= \text{dim}\left(\Sk{k}{3}_{vol,1}(\W{4}) \right) + \text{dim}\left(\Sk{k}{3}_{vol,2}(\W{4}) \right) \\
    &=\frac{1}{2}(k-1)k^2 (k+1) + \frac{1}{6}(k-1)k(k+1)(k+2).
\end{align*}
Evidently, we can show that the following holds
\begin{align*}
    \text{dim}\left(\Sk{k}{3}(\W{4})\right) &= \text{dim} \left(\Sk{k}{3}_{trace}(\W{4})\right) + \text{dim}\left(\Sk{k}{3}_{vol}(\W{4})\right). 
\end{align*}
It then remains for us to prove unisolvency.

\begin{lemma}
    Let $G\in \Vk{k}{3}(\W{4})$ be a polynomial 3-form for which all the degrees of freedom $\Sk{k}{3}(\W{4})$ vanish. Then $G \equiv 0.$
\end{lemma}

\begin{proof}
    The proof is straightforward, as it directly follows the proofs of Lemmas~\ref{prism_lemma_0}, \ref{prism_lemma_1}, and \ref{prism_lemma_2} with $G$, $q^r$, and $q^4$ constructed using the definition of the bubble space, $\Vkcirc{k}{3}(\W{4})$, in Eq.~\eqref{rt_prism_3_bubble_simp}.
\end{proof}

\subsection{Dofs for 4-forms on $\W{4}$}
The polynomial 4-forms on the tetrahedral prism, $\Vk{k}{4}(\W{4})$, have the following total dimension 
\begin{align*}
    \text{dim}\left(\Vk{k}{4}(\W{4})\right) &= \text{dim}\left(\Sk{k}{4}(\W{4}) \right) \\
    &= \text{dim}\left(DG^{k-1}\left(\T{1}\right) \times DG^{k-1}\left(\T{3}\right) \right) \\
    &= \frac{1}{6}k^2 (k+1)(k+2).
\end{align*}
The 4-forms have no facet degrees of freedom. As a result, all the degrees of freedom are volumetric. The volumetric degrees of freedom for the 4-form proxy $q$ can be expressed as follows
\begin{align}
   \Sk{k}{4}_{vol}(\W{4}):= \left\{q \rightarrow \int_{\W{4}} q p, \qquad p \in dP^{k-1}(\T{1}) \times dP^{k-1}(\T{3})\right\}. \label{rt_four_form}
\end{align}
It then remains for us to prove unisolvency.

\begin{lemma}
    Let $q\in \Vk{k}{4}(\W{4})$ be a polynomial 4-form for which all the degrees of freedom $\Sk{k}{4}(\W{4})$ vanish. Then $q \equiv 0.$
\end{lemma}

\begin{proof}
    We begin by choosing a generic $q$, such that
    \begin{align*}
        q = \sum_{ij\ell m} q_{ij\ell m} v_{ij\ell}(x_1, x_2, x_3) v_{m}(x_4),
    \end{align*}
    where
    \begin{align*}
        v_{ij\ell}(x_1,x_2,x_3) \in dP^{k-1}(\T{3}), \quad v_{m}(x_4) \in dP^{k-1}(\T{1}).
    \end{align*}
    The proof follows immediately by setting $p = q$ in Eq.~\eqref{rt_four_form}. Under these circumstances, the degrees of freedom are only guaranteed to vanish if $q$ vanishes.
\end{proof}